\pdfoutput=1
\documentclass[a4paper,12pt]{article}
\usepackage[english]{babel}
\usepackage{amsfonts}
\usepackage{amsthm}
\usepackage{amsmath}
\usepackage{amssymb}
\usepackage{enumerate}
\usepackage{mathtools}
\usepackage{tikz-cd}
\usepackage[all]{xy}
\usepackage{graphicx}
\usepackage{amsmath}
\usepackage{graphicx}
\usepackage{extpfeil}
\usepackage{comment}
\usepackage{float}
\usepackage[labelformat=empty]{caption}
\usepackage[toc]{appendix}
\usepackage[left=3.2cm,top=2.5cm,right=3.2cm,bottom=2.5cm]{geometry}
\usepackage{hyperref}
\usepackage{comment}
\usepackage{resizegather}
\usepackage[activate={true, nocompatibility}, final, tracking=true, kerning=true, spacing=true]{microtype}
\usepackage{authblk}

\theoremstyle{plain}
\newtheorem{thm}{Theorem}[section]
\newtheorem{coroll}[thm]{Corollary}
\newtheorem{defn}[thm]{Definition}
\newtheorem{lemma}[thm]{Lemma}

\newtheorem{example}[thm]{Example}
\newtheorem{prop}[thm]{Proposition}

\newtheorem{remark}[thm]{Remark}

\newtheorem{algor}[thm]{Algorithm}

\newcommand{\bseries}[1]{ [\hspace{-0,5mm}[ {#1} ]\hspace{-0,5mm}] } 
\newcommand{\pseries}[1]{ (\hspace{-0,7mm}( {#1} )\hspace{-0,7mm}) }

\microtypecontext{spacing=nonfrench}

\begin{document}
\thispagestyle{empty}

\title{Reduction theory for connections over the formal punctured disc}
\author{Andres Fernandez Herrero}
\date{}
\affil[]{\small{Department of Mathematics, Cornell University,
    310 Malott Hall, Ithaca, New York 14853, USA;}}
\affil[]{\small{email: \url{ajf254@cornell.edu}}}

\maketitle
\begin{abstract}
We give a purely algebraic treatment of reduction theory for connections over the formal punctured disc. Our proofs apply to arbitrary connected linear algebraic groups over an algebraically closed field of characteristic $0$. We also state and prove some new quantitative results.
\end{abstract}

\tableofcontents

\section{Introduction}
Let $\mathbf{k}$ be an algebraically closed field of characteristic $0$. Fix a connected linear algebraic group $\mathbf{G}$ over $\mathbf{k}$. Let $D^{*}\vcentcolon = \text{Spec} \, \mathbf{k}\pseries{t}$ denote the formal punctured disc over $\mathbf{k}$. 

In this paper we give an algebraic classification of formal $\mathbf{G}$-connections over $D^{*}$ up to gauge equivalence. The regular singular case is more explicit and is presented first (see Subsection \ref{subsection: descent for gauge}). The more general irregular case is Proposition \ref{prop: classif. rat. conn}. These constitute the two main results of the paper.

In order to get our parametrizations, we first prove that every connection can be put into canonical form \cite{Babbitt.varadarajan.formal} after passing to a ramified cover (Theorems \ref{thm:general regular} and \ref{thm: reduction conn general}). 

Next, we describe a set of representatives for canonical forms for which we can develop a clean description of Galois cohomology cocycles (see the set of conditions in Theorem \ref{thm: reduction conn general}). We then proceed to describe the Galois cocycles in Subsections \ref{subsection: descent for gauge} and \ref{subsection: galois for irregular connections}. As a consequence of our arguments we obtain some new quantitative results in Subsection \ref{subsection: quantitative}.

Our approach to the existence of canonical forms is based on the work of Babbitt and Varadarajan \cite{Babbitt.varadarajan.formal}. Some of the crucial parts in their argument are analytic in nature, so they only apply when the ground field is $\mathbb{C}$. We sidestep those parts to provide a completely algebraic proof. In addition, we simplify the global structure of their inductive arguments.

Our treatment of uniqueness of canonical forms is substantially different from the one in \cite{Babbitt.varadarajan.formal}. We choose a different set of representatives for canonical classes in order to set up our Galois cohomology argument (see the list of properties in Theorem \ref{thm: reduction conn general}). This allows us to avoid the use of the complex exponential map. In our setup the proof of uniqueness and the identification of the gauge transformation centralizer become elementary power series arguments.

We develop separate treatments of reduction theory depending on whether $\mathbf{G}$ is reductive, unipotent or solvable. This allows us to give sharper separate statements, including some new determinacy results in the unipotent and solvable cases (see Propositions \ref{prop: determinacy unipotent}, \ref{prop: determinacy regular solvable} and \ref{prop: determinacy irregular solvable}).

There is some related work by Schn\"{u}rer on the subject. \cite{schnurer.regular} gives a purely algebraic proof of the reduction of formal connections when the group $\mathbf{G}$ is reductive and the connection has a regular singularity. In contrast, our arguments apply more generally to arbitrary linear algebraic groups and irregular singular connections.

Let us briefly describe the body of the paper. Section \ref{section: notation} sets up the notation and some of the facts that are used throughout. 

In Section \ref{section: regular connections} we develop reduction theory for regular singular connections. Section \ref{section: regular connections} culminates with the first main result of this paper: an explicit parametrization of regular singular connections over the formal punctured disc. This is achieved in Subsection \ref{subsection: descent for gauge}. This parametrization is rather concrete in the case of classical groups, see Example \ref{example: classical groups regular connections}. 

Section \ref{section: irregular reductive} treats reduction theory of irregular connections for reductive groups. Section \ref{section: irregular reductive} also includes some of the quantitative results mentioned in the abstract. In Subsection \ref{subsection: quantitative} we give an explicit description of the reduction algorithm for reductive groups. An analysis of this algorithm yields determinacy results for both the irregular and the regular part of the canonical form in the reductive case (Proposition \ref{prop: determinacy.irregular}). We also prove a new uniform bound on the ramification needed to put connections into canonical form (see Propositions \ref{prop: ramification reduction reductive} and \ref{prop: determinacy everything}).

Section \ref{section: general irregular} develops reduction theory for irregular connections in the case of an arbitrary linear algebraic group. Subsection \ref{subsection: galois for irregular connections} gives a parametrization of irregular connections over $D^*$ up to gauge equivalence; this is the second main result of this paper. Precisely, this parametrization consists of two pieces of data. The first is a connection $B$ in canonical form satisfying all five conditions listed in Theorem \ref{thm: reduction conn general}. $B$ determines the connection up to ramified gauge equivalence (i.e. when we allow passing to ramified covers of the formal disc). The second piece of data is a $B$-twisted $\mu_b$-cocycle, which is an element in the centralizer of the residue of $B$ satisfying certain additional condition. See Definition \ref{defn: twisted cocycles} and the paragraph following it for an explanation. Section \ref{section: general irregular} ends with Proposition \ref{prop: coxeter bound levels}, where we use this parametrization to give another proof of a result found in \cite{chen-depthpreservation}. 
\section{Some notation and definitions} \label{section: notation}
\subsection{Preliminaries on formal connections}
We will always work over a fixed algebraically closed field $\mathbf{k}$ of characteristic $0$. An undecorated product of $\mathbf{k}$-schemes (e.g. $X\times S$) should always be interpreted as a fiber product over $\mathbf{k}$. $\mathbf{G}$ will be a connected linear algebraic group over $\mathbf{k}$ and $\mathfrak{g} = \text{Lie}(\mathbf{G})$ will be the corresponding Lie algebra. We let $\mathcal{O} = \mathbf{k}\bseries{t}$ denote the ring of formal power series over $\mathbf{k}$ and $F = \mathbf{k}\pseries{t}$ denote the corresponding field of Laurent series. $\mathcal{O}$ is a discrete valuation ring with maximal ideal $t\mathcal{O}$.

Recall that the module of K\"{a}hler differentials $\Omega^{1}_{\mathcal{O} / \mathbf{k}}$ classifies $\mathbf{k}$-derivations from $\mathcal{O}$. It is spanned as an $\mathcal{O}$-module by formal elements $df$ for every $f \in \mathcal{O}$, subject to the relations $d(fg) = f dg + g df$. We will work with the module of continuous K\"{a}hler differentials $\hat{\Omega}_{\mathcal{O}/\mathbf{k}}^{1}$ , which is defined to be the completion \[\hat{\Omega}_{\mathcal{O}/\mathbf{k}}^{1} \vcentcolon = \underset{n}{\varprojlim} \, \,  \, {\Omega}_{\mathcal{O} / \mathbf{k}}^{1} \, / \, \, t^n \, {\Omega}_{\mathcal{O}/\mathbf{k}}^{1}\]
This is a free $\mathcal{O}$-module of rank $1$. The natural completion map $\widehat{(-)} : \Omega_{\mathcal{O}/\mathbf{k}}^{1} \rightarrow \hat{\Omega}_{\mathcal{O}/\mathbf{k}}^{1}$ can be thought of as the projection onto the quotient obtained by adding the extra relations coming from allowing termwise differentiation of power series.
\begin{remark}
The module of ordinary K\"{a}hler differentials $\Omega^1_{\mathcal{O} / \mathbf{k}}$ is not finitely generated as an $\mathcal{O}$-module. We don't want to work with $\Omega^1_{\mathcal{O} / \mathbf{k}}$, because the relations above do not include classical intuitive identities like $d (e^t) = e^t dt$. That is the reason why we use continuous K\"{a}hler differentials instead.
\end{remark}
For any positive natural number $b$, let $F_{b} \vcentcolon = \mathbf{k}\pseries{t^{\frac{1}{b}}}$. This is a finite Galois extension of $F$ with Galois group canonically isomorphic to $\mathbf{\mu}_{b}$, the group of $b$-roots of unity in $\mathbf{k}$. Under this isomorphism, we have that $\gamma \in \mu_b$ acts by $\gamma \cdot t^{\frac{1}{b}} = \gamma^{-1} t^{\frac{1}{b}}$. Notice that the choice of a primitive root of unity yields an identification $\mu_b \cong \mathbb{Z} / b\, \mathbb{Z}$, since we are working in characteristic $0$. A well known theorem of Puiseux states that the algebraic closure of $F$ is $\overline{F} = \bigcup_{b \geq 1} F_{b}$. 

In this paper we will work with a (right) $\mathbf{G}$-torsor $P$ over the formal punctured disc $D^* \vcentcolon = \text{Spec}\,F$. We know that $P$ can be trivialized, meaning that $P \cong \text{Spec}\, F \times \mathbf{G}$ as right $\mathbf{G}$-torsors. This follows from theorems of Tsen and Springer, see \cite{serre.galoiscoh} page 80 - 3.3(b) and page 132 - 2.3(c). A formal connection $A$ on $P$ is a function from the set of trivializations of $P$ into $\mathfrak{g} \, \otimes_{\mathbf{k}} \, \hat{\Omega}_{\mathcal{O}/\mathbf{k}}^{1} \left[\frac{1}{t}\right]$ that satisfies a certain transformation law. In order to describe the transformation law we need some notation.

Let $T_{\mathbf{G}}$ be the tangent sheaf of $\mathbf{G}$. There is a natural trivialization $T_{\mathbf{G}} \cong \mathfrak{g} \, \otimes_{\mathbf{k}} \mathcal{O}_{\mathbf{G}}$ given by left translation. Therefore, we get an isomorphism  $\mathfrak{g} \, \otimes_{\mathbf{k}} \, \Omega^{1}_{\mathbf{G} / \mathbf{k}} \cong \mathfrak{g} \, \otimes_{\mathbf{k}} \, \text{Hom}_{\mathcal{O}_{\mathbf{G}}} ( T_{\mathbf{G}}, \mathcal{O}_{\mathbf{G}}) \cong \text{Hom}_{\mathbf{k}}(\mathfrak{g}, \mathfrak{g}) \otimes_{\mathbf{k}} \, \mathcal{O}_{\mathbf{G}}$. The invariant $\mathfrak{g}$-valued 1-form on $\mathbf{G}$ that corresponds to $\text{id}_{\mathfrak{g}} \otimes 1$ under this isomorphism is called the Maurer-Cartan form. We will denote it by $\omega \in   \mathfrak{g}\, \otimes_{\mathbf{k}} \Omega^{1}_{\mathbf{G} / \mathbf{k}}$.

Suppose that we are given an element $g \in \mathbf{G}(F)$. We can think of it as a map $g: \text{Spec} \, F \longrightarrow \mathbf{G}$. We can use $g$ to pull back the Maurer-Cartan form to $\text{Spec}\, F$ in order to obtain $g^{*}\omega \in  \mathfrak{g} \, \otimes_{\mathbf{k}} \, \Omega_{F / \mathbf{k}}^{1} =  \mathfrak{g} \, \otimes_{\mathbf{k}} \, \Omega_{\mathcal{O} / \mathbf{k}}^{1}\left[\frac{1}{t}\right]$. By applying the completion map $\widehat{(-)} : \Omega_{\mathcal{O}/\mathbf{k}}^{1} \rightarrow \hat{\Omega}_{\mathcal{O}/\mathbf{k}}^{1}$, we get an element $\widehat{{g}^{*} \omega} \in \mathfrak{g} \, \otimes_{\mathbf{k}} \, \hat{\Omega}_{\mathcal{O}/\mathbf{k}}^{1} \left[\frac{1}{t}\right]$. Now we can define the gauge action of $\mathbf{G}\left(F\right)$ on $\mathfrak{g} \, \otimes_{\mathbf{k}} \, \hat{\Omega}_{\mathcal{O}/\mathbf{k}}^{1} \left[\frac{1}{t}\right]$. For any 
$g \in \mathbf{G}\left(F\right)$ and $B \in \mathfrak{g} \, \otimes_{\mathbf{k}} \, \hat{\Omega}_{\mathcal{O}/\mathbf{k}}^{1} \left[\frac{1}{t}\right]$, we set $g \cdot B  \vcentcolon = \text{Ad}(g) B + \widehat{{g}^{*} \omega}$ .
\begin{defn}
By a formal connection $A$ for $P$ we mean a function 
\[A \; : \; \left\{\text{trivializations} \; \;P \xrightarrow{\sim}\text{Spec} \, F \times \mathbf{G} \right\} \; \; \longrightarrow  \; \; \mathfrak{g} \, \otimes_{\mathbf{k}} \, \hat{\Omega}_{\mathcal{O}/\mathbf{k}}^{1} \left[\frac{1}{t}\right]\]
satisfying the following transformation law. Let $\phi_1, \, \phi_2 \, : \, P \xrightarrow{\sim}\text{Spec}\, F \times \mathbf{G}$ be two trivializations of $P$. We know that $\phi_2 \circ \phi_1^{-1}$ is given by left multiplication by a unique element $g \in \mathbf{G}\left(F\right)$. We then require $A(\phi_2) = g \cdot A(\phi_1)$.
\end{defn}
\begin{remark}
Th reader might have encountered a different definition of formal connection. Using the action of $\mathbf{G}$ on $\mathfrak{g} \otimes_{\mathbf{k}} \Omega^1_{\mathcal{O}/\mathbf{k}} \left[ \frac{1}{t}\right]$ one can define a formal version of the Atiyah sequence \cite{atiyah.connections}. Splittings of such sequence will correspond to formal connections as we have defined them.
\end{remark}
Such a connection $A$ is completely determined by its value at any given trivialization. We will often assume that we have chosen a fixed trivialization of $P$. Hence we can think of $P$ as the trivial bundle, and think of $A$ as the element of $\mathfrak{g} \, \otimes_{\mathbf{k}} \, \hat{\Omega}_{\mathcal{O}/\mathbf{k}}^{1} \left[\frac{1}{t}\right]$ given by the image of this trivialization. Note that we have implicitly fixed a choice of uniformizer $t$ for $\mathcal{O}$. This yields an isomorphism $\hat{\Omega}_{\mathcal{O} / \mathbf{k}}^{1} = \mathcal{O} \,dt \cong \mathcal{O}$. We will often think of connections as elements of $\mathfrak{g}_F \vcentcolon = \mathfrak{g} \otimes_{\mathbf{k}} F$ obtained under the induced isomorphism $\Omega_{\mathcal{O} / \mathbf{k}}^1 \left[ \frac{1}{t} \right] = F \, dt \cong F$.

All of the discussion above also applies over any finite field extension $F_b$ of $F$. The choice of a uniformizer $u \vcentcolon = t^{\frac{1}{b}}$ for $F_b$ yields an isomorphism from $F$ onto $F_{b}$ sending $t$ to $u$. This allows us to``lift" $\mathbf{G}$-bundles and trivializations from $\text{Spec}\,F_b$ to $\text{Spec} \,F$ by transport of structure. We can therefore lift connections from $F_b$ to $F$.

There are some subtleties for the lift of connections when we think of them as elements of $\mathfrak{g}_F$. We generally take derivatives with respect to $t$, and not $u =t^{\frac{1}{b}}$. That is, we fix the trancendental element $t = u^b$ of $F_{b}$ in order to get the isomorphism $ \hat{\Omega}_{\mathcal{O}_{b}/\mathbf{k}}^{1}\left[ \frac{1}{u}\right] = \left( \mathcal{O}_b \, dt \right) \left[\frac{1}{u}\right] \cong \mathcal{O}_b  \left[\frac{1}{u}\right] = F_b$. Under this identification, the lift of a $\mathbf{G}$-connection is not the obvious one given by replacing $u$ by $t$. Instead, the lift of a connection $A = \sum_{j = r}^{\infty} A_j \, t^{\frac{j}{b}} \in \mathfrak{g}_{F_{b}}$ is given by $\tilde{A} \vcentcolon = b t^{b-1} \sum_{j = r}^{\infty} A_j \, t^j$. This is called the $b$-lift of the connection.

Let $\mathbf{T} \subset \mathbf{G}$ be a maximal torus in $\mathbf{G}$. We will denote by $X_{*}(\mathbf{T})$ (resp. $X^{*}(\mathbf{T})$) the cocharacter (resp. character) lattice of $\mathbf{T}$. We will write $\langle-, -\rangle : X_{*}(\mathbf{T}) \otimes X^{*}(\mathbf{T}) \longrightarrow \mathbb{Z}$ for the canonical pairing. There is a natural inclusion $X_{*}(\mathbf{T}) \subset \text{Lie}(\mathbf{T})$ given by taking differentials at the identity. We will freely use this identification without further notice. Note that a cocharacter $\lambda : \mathbb{G}_m \longrightarrow \mathbf{T} \subset \mathbf{G}$ yields a point $\lambda \in \mathbf{G}(\mathbf{k}[t, t^{-1}])$. We denote by $t^{\lambda}$ the element of $\mathbf{G}(F)$ obtained via the natural inclusion $\mathbf{k}[t, t^{-1}]\hookrightarrow F$.

We will make use of the algebraic exponential map, as in \cite{demazure-gabriel} pg. 315. For $X \in t \mathfrak{gl}_n (\mathcal{O})$ we have an exponential $\text{exp}(X) \in \mathbf{GL_n}(\mathcal{O})$ defined by $\text{exp}(X) \vcentcolon = \sum_{i = 0}^{\infty} \frac{1}{i!}X^i$. By choosing a closed embedding $\mathbf{G} \hookrightarrow \mathbf{GL_n}$ we can similarly define an exponential map $\text{exp} \vcentcolon t \mathfrak{g}(\mathcal{O}) \longrightarrow \mathbf{G}(\mathcal{O})$. It can be checked that this does not depend on the choice of embedding. We will only use one property of this map: for any $X \in \mathfrak{g}$, the image of $\text{exp}(t^n \, X)$ when we reduce modulo $t^{n+1}$ is given by $ 1 + t^n X \in  \mathbf{G}\left(\mathcal{O}/ t^{n+1} \mathcal{O} \right)$.
\subsection{Adjoint orbits in semisimple Lie algebras}
Here we include some facts about semisimple algebraic groups and their Lie algebras. Most of these results are standard and can be found in the book \cite{collingwood.nilpotent}. For the rest of this section we will assume that $\mathbf{G}$ is connected semisimple.

 Recall that an element of a semisimple Lie algebra is called semisimple (resp. nilpotent) if the the image under the adjoint representation is semisimple (resp. nilpotent) as a linear transformation of $\mathfrak{g}$. It turns out that we can check these conditions on any faithful representation. This fact follows from the following theorem.
\begin{thm}[Additive Jordan Decomposition] 
Let $\mathfrak{g}$ semisimple. For any $A \in \mathfrak{g}$ there exist unique a  $A_s$ semisimple and $A_n$ nilpotent such that\vspace{-0.25cm}
\begin{enumerate}[(i)]
    \item $A = A_s + A_n$
    \item $[A_s , A_n ] = 0$
\end{enumerate}
\end{thm}
\begin{remark}  \label{remark: semisimple general}
For a reductive Lie algebra, all elements of the center are considered semisimple. For the Lie algebra of an arbitrary linear algebraic group, we will usually fix a Levi subgroup $\mathbf{L}$ and speak of semisimple elements inside $\text{Lie}(\mathbf{L})$.
\end{remark}
Recall that $\mathfrak{sl}_2 = \{ X \in \mathfrak{gl}_2 \mid \text{tr}(X) = 0 \}$. The Lie bracket is given by the matrix commutator. Define $H= \begin{bmatrix} 1 & 0 \\ 0 & -1 \end{bmatrix}$, $X= \begin{bmatrix} 0 & 1 \\ 0 & 0 \end{bmatrix}$ and $Y= \begin{bmatrix} 0 & 0 \\ 1 & 0 \end{bmatrix}$. Then we have $\mathfrak{sl}_2 = \mathbf{k} H \oplus \mathbf{k} X \oplus \mathbf{k} Y$ as a $\mathbf{k}$-vector space.
\begin{defn}
A \textbf{$\mathfrak{sl}_2$-triple} in $\mathfrak{g}$ is a nonzero Lie algebra map $\phi: \mathfrak{sl}_2 \longrightarrow \mathfrak{g}$. We will often abuse notation and denote the images of $H, X, Y$ with the same letters.
\end{defn}
\begin{thm}[Jacobson-Morozov] \label{thm: jacobson} Let $\mathbf{G}$ be a connected semisimple algebraic group with Lie algebra $\mathfrak{g}$. Let $U \in \mathfrak{g}$ be a nilpotent element. Then there exists a homomorphism $\Phi \vcentcolon SL_2 \longrightarrow G$ such that the $\mathfrak{sl}_2$-triple corresponding to the differential $d\Phi: \mathfrak{sl}_2 \longrightarrow \mathfrak{g}$ satisfies $d \Phi (Y)= U$. Moreover such a homomorphism is uniquely determined up to conjugation by an element of the centralizer $Z_{\mathbf{k}}(U)(\mathbf{k})$.
\end{thm}
If $Y\neq 0$ is a nilpotent element in $\mathfrak{g}$, we will denote by $(H, X, Y)$ the $\mathfrak{sl}_2$-triple granted by Jacobson-Morozov. For any element $X \in \mathfrak{g}$, we will write $\mathfrak{g}_{X}$ for the centralizer of $X$ in $\mathfrak{g}$.

Let $G = \mathbf{G}(\mathbf{k})$ denote the $\mathbf{k}$-rational points of $\mathbf{G}$. Recall that for any $Y \in \mathfrak{g}$, the orbit under the adjoint action $\mathbf{G} \cdot Y$ can be equipped with the structure of a smooth locally closed subvariety of $\mathfrak{g}$. We will often harmlessly identify it with its closed points $G \cdot Y$. The following proposition is going to be the essential technical tool for the induction argument in the reductive case. The proof can be found in \cite{Babbitt.varadarajan.formal}  pages 17-18.
\begin{prop} \label{prop: dim increase}
Let $Y\neq 0$ be nilpotent in $\mathfrak{g}$. Let $(H, X, Y)$ be the corresponding $\mathfrak{sl}_2$-triple. Then the affine space $Y + \mathfrak{g}_X$ meets the orbit $G \cdot Y$ exactly at $Y$. For any other nilpotent $U \in Y+\mathfrak{g}_X$ with $U\neq Y$, we have $\text{dim}(G \cdot U) > \text{dim}(G \cdot Y)$. 
\end{prop}
\begin{example} If $Y$ is regular nilpotent, then it is the unique nilpotent element in $Y + \mathfrak{g}_X$.
\end{example}
Fix a maximal torus $\mathbf{T} \subset \mathbf{G}$. Let $\Phi$ be the set of roots of $\mathbf{G}$ with respect to $\mathbf{T}$.  The coweight lattice $Q_{\mathbf{G}}$ of $\mathbf{G}$ with respect to $\mathbf{T}$ is defined to be $Q_{\mathbf{G}} \vcentcolon = \text{Hom}(\mathbb{Z} \Phi, \, \mathbb{Z})$. Since $\mathbf{G}$ is semisimple, the cocharacter lattice $X_*(\mathbf{T})$ has finite index in the coweight lattice $Q_{\mathbf{G}}$.
\begin{defn}
The index $I(\mathbf{G})$ is defined to be the exponent of the finite group $Q_{\mathbf{G}} / \, X_*(\mathbf{T})$.
\end{defn}
Let $\Phi^{\vee}$ be the set of coroots of $\mathbf{G}$ with respect to $\mathbf{T}$. We have the following chain of inclusions
\[ \mathbb{Z} \Phi^{\vee} \, \subset \, X_*(\mathbf{T}) \subset \, Q_{\mathbf{G}}\]
\begin{defn} \label{defn: exponent coweights cooroots}
$J(\mathbf{G})$ is the exponent of the finite group $Q_{\mathbf{G}} / \, \mathbb{Z} \Phi^{\vee}$.
\end{defn}
\begin{remark}
Since all maximal tori in $\mathbf{G}$ are conjugate, both $I(\mathbf{G})$ and $J(\mathbf{G})$ do not depend on the choice of $\mathbf{T}$.
\end{remark}
Let us fix a Borel subgroup $\mathbf{B} \subset \mathbf{G}$ containing $\mathbf{T}$. This amounts a choice of positive roots $\Phi^{+}$. We let $\Delta$ be the corresponding subset of simple roots.
\begin{defn}
Let $\alpha = \sum_{\beta \in \Delta} m_{\beta} \beta$ be a positive root. The height of $\alpha$ is defined to be $ \text{hgt}\,(\alpha) \vcentcolon = \sum_{\beta \in \Delta} \; m_{\beta}$.
The height of the Lie algebra $\mathfrak{g}$ is $\text{hgt}\,(\mathfrak{g}) \vcentcolon = \text{sup}_{\alpha \in \Phi^{+}} \; \text{hgt}\,(\alpha)$.
\end{defn}
To conclude this section, we define a function that measures the ``size" of the semisimple element $H$ in the Jacobson-Morozov triple corresponding to a nilpotent $Y \in \mathfrak{g}$. We can always arrange $H \in X_{*}(\mathbf{T})$. We will implicitly assume this from now on.
\begin{defn} \label{defn: lambda}
Let $Y \in \mathfrak{g}$ be a nilpotent element. Let $H$ be the corresponding semisimple element in the Jacobson-Morozov triple of $Y$. Then, we define $\Lambda(Y) \vcentcolon = \text{sup}_{\alpha \in \Phi} \; \left( \frac{1}{2} \alpha(H) +1 \right)$. This function $\Lambda$ is constant on nilpotent orbits.
\end{defn}
\begin{example} \label{example: regular nilp} Suppose that $Y$ is regular nilpotent. We can choose $H$ so that $\alpha(H) =2$ for every $\alpha \in \Delta$ (see \cite{collingwood.nilpotent} Chapter 3). Therefore, $\Lambda(Y) = \text{hgt}\,(\mathfrak{g}) +1$ in this case. It turns out that this is the biggest possible value for $\Lambda$. In other words $\Lambda(Y) \leq \text{hgt} \,(\mathfrak{g}) +1$ for any nilpotent $Y \in \mathfrak{g}$.
\end{example}
\section{Regular connections} \label{section: regular connections}
Fix a connected linear algebraic group $\mathbf{G}$ over $\mathbf{k}$. What we call regular connections are also known as connections with at worst regular singularities.
\begin{defn} A connection $A = \sum_{j =r}^{\infty} A_j \, t^j\in \mathfrak{g}_F$ is said to be of the first kind if if it has at worst a simple pole (i.e. $r\geq -1$).
A connection $A$ is called regular if there exists $x \in \mathbf{G}(\overline{F})$ such that $x\cdot A$ is of the first kind.
\end{defn}
 In the analytic context, regular connections are classified by topological data. Indeed, such connections are determined by their monodromy representation. Our goal in this section is to classify formal regular connections over an arbitrary ground field. This will be achieved in Subsection \ref{subsection: descent for gauge}. It should be noted that the development of the regular case is a necessary preliminary step in our treatment of the general irregular singular case in Sections \ref{section: irregular reductive} and \ref{section: general irregular}.
 
 As mentioned in the introduction, the regular singular case for $\mathbf{k} = \mathbb{C}$ is treated in \cite{Babbitt.varadarajan.formal} using transcendental methods. The case of the group $\text{GL}_n$ was known well before. See Deligne's book \cite{deligne.regulier}[II \S 1] for a discussion of regular singular connections for $\text{GL}_n$ before the paper \cite{Babbitt.varadarajan.formal}. It should be noted that Levelt \cite{levelt.cyclic.vector} gave a proof of the existence of canonical forms for $\text{GL}_n$ that applies to any algebraically closed field $\mathbf{k}$. 
\subsection{Regular connections for semisimple groups}
We will start with the semisimple case. A regular connection is said to be in canonical form if it can be written as $t^{-1} C$ for some $C \in \mathfrak{g}$. In order to prove Theorem \ref{thm:semisimple regular} we can assume that $A$ is of first kind, because of the definition of regular connection. We first need the following definition and lemma from \cite[8.5]{Babbitt.varadarajan.formal}, which actually work for arbitrary $\mathbf{G}$. We include a detailed proof of the lemma in order to keep the exposition self-contained.
\begin{defn}
Let $\mathbf{G}$ be a connected linear algebraic group. Let $A=\sum_{j =-1}^{\infty} A_j \, t^j$ be a connection of the first kind in $\mathfrak{g}_F$. The endomorphism $\text{ad} \, (A_{-1}) \, \in \text{GL}_n(\mathfrak{g})$ yields a decomposition of $\mathfrak{g}$ into generalized eigenspaces $\mathfrak{g} = \bigoplus_{\lambda} \mathfrak{g}_{\lambda}$. We say that $A$ is aligned if $A_j \in \mathfrak{g}_{j+1}$ for all $j$.
\end{defn}
\begin{lemma} \label{lemma:aligned} 
Let $\mathbf{G}$ be a connected linear algebraic group and $A = \sum_{j =-1}^{\infty} A_j \, t^j$ a formal connection of the first kind in $\mathfrak{g}_F$. Then there exist $x \in \mathbf{G}(\mathcal{O})$ such that $x \cdot A$ is aligned.
\end{lemma}
\begin{proof}
We will inductively build a sequence $(B_j)_{j=1}^{\infty}$ of elements of $\mathfrak{g}$ such that the change of trivialization by $x \vcentcolon = \lim_{n \rightarrow \infty} \prod_{j =0}^{n-1} \text{exp}(t^{n-j} \, B_{n-j})$ puts $A$ in aligned form. Let $k \in \mathbb{N}$. Suppose that we have chosen $B_j \in \mathfrak{g}$ for all $j \leq k$ such that the connection $A^{(k)} = \sum_{l=-1}^{\infty} A^{(k)}_{l } \,t^l$ defined by $A^{(k)} \vcentcolon = \prod_{j =0}^{k-1} \text{exp}(t^{k-j} \, B_{k-j}) \cdot A$ satisfies $A_{l}^{(k)} \in \mathfrak{g}_{l+1}$ for all $l < k$. Notice that the base case $k=0$ is trivial and that we have $A^{(k)}_{-1} = A_{-1}$. Let's try to determine $B_{k+1}$.

Recall that $\text{exp}(t^{k+1} \, B_{k+1}) \equiv 1 + t^{k+1} \, B_{k+1} \; (\text{mod} \; t^{k+2})$. By an elementary matrix computation (choose an embedding of $\mathbf{G} \hookrightarrow \mathbf{\text{GL}_n}$), we can see that
\[\text{exp}(t^{k+1} B_{k+1}) \cdot A^{(k)} \equiv \sum_{l=-1}^{k-1} A^{(k)}_l \, t^l + [A^{(k)}_k - (ad(A_{-1}) - (k+1))B_{k+1}] \, t^k \; \; (\text{mod} \; t^{k+1}) \]
Decompose $\mathfrak{g}$ into generalized $ad(A_{-1})$ eigenspaces $\mathfrak{g} = \bigoplus_{\lambda} \mathfrak{g}_{\lambda}$. By definition the operator $ad(A_{-1}) - (k+1)$ restricts to an automorphism of $\mathfrak{g}_{\lambda}$ for all $\lambda \neq k+1$. In particular, we can choose $B_{k+1} \in \mathfrak{g}$ such that $A^{(k)}_k - (ad(A_{-1}) - (k+1))B_{k+1}$ is in $\mathfrak{g}_{k+1}$. This concludes the induction step. It follows by construction that the gauge transformation by $x \vcentcolon = \lim_{n \rightarrow \infty} \prod_{j =0}^{n-1} \text{exp}(t^{n-j} \, B_{n-j})$ puts $A$ in aligned form.
\end{proof}
\begin{remark} \label{remark: determinacy terms regular semisimple}
Any aligned connection is actually in $\mathfrak{g} \otimes \mathbf{k}[t, t^{-1}]$. The coefficient with largest exponent is $\left(x \cdot A\right)_j t^j$, where $j+1$ is the biggest integer eigenvalue of $ad\left( \,(A_{-1})_s \, \right)$. We denote this number  by $k(A_{-1})\vcentcolon = j+1$ for further reference. In order to determine the resulting aligned connection, we only need to multiply by $k(A_{-1})$-many exponentials in the proof above. Therefore the aligned form only depends on $A_j$ for $-1 \leq j \leq k(A_{-1})$. Note that $k(A_{-1})$ can drastically change if we multiply $A$ by a scalar in $\mathbf{k}$. This reflects the fact that gauge transformations are not $\mathbf{k}$-linear.
\end{remark}
\begin{example}
Suppose that $\text{ad}\left( \, \left(A_{-1}\right)_s \, \right)$ does not have any integer eigenvalues. Then the aligned connection will be in canonical form. 
\end{example}

\begin{thm} \label{thm:semisimple regular} Let $\mathbf{G}$ be a connected semisimple algebraic group. Let $A \in \mathfrak{g}_F$ be a regular connection. Then, there exists $x \in \mathbf{G}(\overline{F})$ such that $x \cdot A = t^{-1} C$ for some $C \in \mathfrak{g}$.
\end{thm}
\begin{proof}
This is a special case of \cite{schnurer.regular}[Thm. 4.2]. We include the proof with some modifications that will suit our needs in subsequent subsections. By Lemma \ref{lemma:aligned}, we can assume that $A$ is an aligned connection in $\mathfrak{g}_F$. Let $(A_{-1})_{s}$ be the semisimple part of $A_{-1}$. Choose a maximal torus $\mathbf{T}$ of $\mathbf{G}$ such that the corresponding Cartan subalgebra $\text{Lie}(\mathbf{T})$ contains $(A_{-1})_{s}$. Fix a choice of positive roots $\Phi^{+}$ of $\mathbf{G}$ relative to $\mathbf{T}$. Let $\Delta$ be the subset of simple roots. Choose a basis for $\mathbf{k}$ as a vector space over $\mathbb{Q}$. Suppose that $1$ is one of the basis elements. Let $\pi \vcentcolon \mathbf{k} \longrightarrow \mathbb{Q}$ be the corresponding projection. We can define $\tau$ in $\text{Lie}(\mathbf{T})$ given by $\tau(\alpha) = \pi(\alpha((A_{-1})_s))$ for all $\alpha \in \Delta$.

There exists $b \in \mathbb{N}$ such that $b \tau$ is in the cocharacter lattice of $\mathbf{T}$. We let $\mu \vcentcolon = b\tau$ be the corresponding cocharacter. Recall from the preliminaries that we have a $b$-lift $\tilde{A}= \sum_{j=-1}^{\infty} b A_j \, t^{bj+b-1}$. We can assume that we are working with $\tilde{A}$ by passing to the $b$-ramified cover. We claim that $t^{-\mu} \cdot \tilde{A}$ is in canonical form. In order to show this, it is convenient to use the $ad$ representation and view everything as matrices in $\text{End}(\mathfrak{g})$. The root decomposition $\mathfrak{g} = \bigoplus_{\alpha \in \Phi} \mathfrak{g}_{\alpha}$ gives us the spectral decomposition of $({A}_{-1})_{s}$.

We can view $\text{Ad}(t^{- \mu})$ as a matrix in $\text{GL}(\mathfrak{g}_F)$.  $\text{Ad}(t^{-\mu})$ acts as the scalar $t^{-\langle \mu,\beta \rangle}$ on the root space $\mathfrak{g}_{\beta}$.
By assumption $A$ is aligned. This means that $A_j$ is in a sum of root spaces $\mathfrak{g}_{\beta}$ where $(A_{-1})_{j+1}$ has eigenvalue $j+1$. These are the root spaces $\mathfrak{g}_{\beta}$ where $\beta((A_{-1})_s) = j+1$. By the construction of $\mu$, we know that $\langle \mu , \beta\rangle = b \beta((A_{-1})_s)$ whenever $\beta((A_{-1})_s) $ is an integer. Therefore, $\text{Ad}(t^{-\mu}) \, A_j = t^{-bj -b} A_j $. We conclude that 
\[ t^{-\mu} \cdot \tilde{A} = t^{-\mu} \cdot \left( \sum_{j=-1}^{\infty} b A_j \, t^{bj+b-1}\right) = \left( \sum_{j=-1}^{\infty} b A_j \right) \, t^{-1} \, + \, \frac{d}{dt} \, \left( t^{-\mu}\right) \, t^{\mu}\]

For the last term $\frac{d}{dt} \, \left( t^{-\mu}\right) \, t^{\mu}$ we are performing the calculation in $\text{End}(\mathfrak{g}_F)$. A matrix computation yields $\frac{d}{dt} \, \left( t^{-\mu}\right) \, t^{\mu} = -\mu \, t^{-1}$. The theorem follows.
\end{proof}
\begin{remark}
Babbitt and Varadarajan prove the theorem over the ground field $\mathbf{k} = \mathbb{C}$ using the analytic theory of regular singular connections (see section 8 of \cite{Babbitt.varadarajan.formal}). In the analytic setting, we need to pass to a ramified cover only when the monodromy class of the connection is not in the image of the exponential map. For example all conjugacy classes in $\mathbf{GL}_n$ are exponential, so we don't need to pass to a ramified cover to reduce regular $\mathbf{GL_n}$-connections. This latter fact can also be proven algebraically using the center of $\mathbf{GL}_n$. See the argument in pages 19-22 of \cite{Babbitt.varadarajan.formal}.
\end{remark}
\begin{remark}
We only need to fix a rational basis of $\text{span}_{\mathbb{Q}}\{ \alpha((A_{-1})_s) \, \vcentcolon \, \alpha \in \Delta \}$ in the proof above. So the argument is constructive.
\end{remark}
We can be a bit more careful in the proof of Theorem \ref{thm:semisimple regular}. This way we can get a uniform bound for the ramification needed. We record this as a small lemma.
\begin{lemma} \label{lemma: ramification bound regular semisimple}
 We can always choose $b \leq \text{hgt}(\mathfrak{g}) \cdot I(\mathbf{G})$ in the proof of Theorem \ref{thm:semisimple regular}.
\end{lemma}
\begin{proof}
Set $\tau(\alpha)$ to be the best approximation of $\pi(\alpha((A_{-1})_s))$ in $\frac{1}{\text{hgt}(\mathfrak{g}) }\mathbb{Z}$. By the definition of $\text{hgt}(\mathfrak{g})$, it follows that $\tau(\beta) = \beta((A_{-1})_s)$ whenever $\beta((A_{-1})_s)$ is an integer. So the proof of Theorem \ref{thm:semisimple regular} still goes through with this choice of $\tau$. By construction we have $\text{hgt}(\mathfrak{g}) \tau \in Q_{\mathbf{G}}$. Then the definition of $I(\mathbf{G})$ implies that $\text{hgt}(\mathfrak{g}) I(\mathbf{G}) \tau \in X_{*}(\mathbf{T})$. Hence we can choose $b = \text{hgt}(\mathfrak{g}) I(\mathbf{G})$.
\end{proof}
Choose a maximal torus $\mathbf{T} \subset \mathbf{G}$. Let $W$ be the Weyl group of $\mathbf{G}$ with respect $\mathbf{T}$. Fix a projection $\pi : \mathbf{k} \longrightarrow \mathbb{Q}$ as in the proof above. We can extend this projection to a natural map $\pi: \text{Lie}(\mathbf{T}) \cong X_{*}(\mathbf{T}) \otimes \mathbf{k} \longrightarrow X_{*}(\mathbf{T}) \otimes \mathbb{Q}$. We will once and for all fix a fundamental domain $\mathfrak{D}$ for the action of $W$ on the set $\Xi \vcentcolon =\left\{C \in \text{Lie}(\mathbf{T}) \, \mid \, \pi(C) = 0 \right\}$. Notice that $\Xi$ is a set of representatives for the quotient $\text{Lie}(\mathbf{T}) / \, X_{*}(\mathbf{G}) \otimes \mathbb{Q}$. 

It turns out that we can always choose $x$ in Theorem \ref{thm:semisimple regular} so that the semisimple part $C_s$ is in $\mathfrak{D}$.
\begin{coroll} \label{coroll: canonical unique semisimple}
Let $\mathbf{G}$ be connected semisimple with a choice of maximal torus $\mathbf{T} \subset \mathbf{G}$. Let $A \in \mathfrak{g}_F$ be a regular connection. Then, there exists $x \in \mathbf{G}(\overline{F})$ such that $x \cdot A = t^{-1} C$ for some $C \in \mathfrak{g}$ satisfying $C_s \in \mathfrak{D}$.
\end{coroll}
\begin{proof}
By Theorem \ref{thm:semisimple regular}, we can assume that $A = t^{-1} \, C$ for some $C \in \mathfrak{g}$. Since $\mathbf{k}$ is algebraically closed, we can conjugate the semisimple element $C_s$ to the torus $\mathbf{T}$. By applying the gauge transformation $t^{-\pi(C_s)}$, we can assume that $\pi(C_s) = 0$. Finally, we can conjugate by an element of $W$ to obtain $C_s \in \mathfrak{D}$.
\end{proof}
The following proposition will be crucial in establishing uniqueness of canonical reductions in general.
\begin{prop} \label{prop: uniqueness regular semisimple}
Let $\mathbf{G}$ be connected semisimple with a choice of maximal torus $\mathbf{T} \subset \mathbf{G}$. Let $C,D \in \mathfrak{g}$ with $C_s, D_s \in \mathfrak{D}$. Suppose that there exists $x \in \mathbf{G}(\overline{F})$ such that $x \cdot \left(t^{-1} \, C \right) = t^{-1} \, D$. Then we have $C_s = D_s$. Moreover $x$ is a $\mathbf{k}$-point in the centralizer $Z_{\mathbf{G}}(C_s)(\mathbf{k})$.
\end{prop}
\begin{proof}
By lifting everything to a ramified cover, we can assume for simplicity that $x \in \mathbf{G}(F)$. Choose a faithful representation $\mathbf{G} \hookrightarrow \mathbf{\text{GL}_n}$. We can view $x \in \mathbf{\text{GL}_n}(F)$ and $C, D \in \mathfrak{gl}_n$. Let's consider the linear transformation $U$ in $\text{End}(\mathfrak{gl}_n)$ given by $U\, v = D v \, - \, v C$ for all $v \in \mathfrak{gl}_n$. Notice that we can write $U = U_s + U_n$, where
\begin{align*}
    U_s \, v \vcentcolon = D_s v - v C_s\\
    U_n \, v \vcentcolon = D_nv - vC_n
\end{align*}
We know that $C_s$ and $D_s$ can be simultaneously diagonalized. Therefore $U_s$ is semisimple. The eigenvalues of $U_s$ are differences of eigenvalues of $C_s$ and $D_s$. Since $\pi(C_s) = \pi(D_s) = 0$, we conclude that $0$ is the only possible rational eigenvalue of $U_s$. By definition, we have that $U_n$ is nilpotent and $[U_s, U_n] = 0$. We conclude that $U = U_s + U_n$ is the additive Jordan decomposition of $U$. In particular the set of eigenvalues of $U$ is the same as the set of eigenvalues of $U_s$. Therefore, $0$ is the only possible rational eigenvalue of $U$.

The condition $x \cdot \left(t^{-1}\, C\right) = t^{-1}\, D$ can be expressed as $\frac{d}{dt}x = t^{-1}U \, x$. Here we are viewing $x$ as an invertible matrix in $\mathfrak{gl}_n(F)$. Set $x = \sum_{j = r}^{\infty} x_j \, t^j$. Then this condition reads
\[ \sum_{j=r}^{\infty} j x_j \, t^{j-1} = \sum_{j=r}^{\infty} U\,  x_j \, t^{j-1}   \]
Hence we have $jx_j = U\, x_j$ for all $j$. Since $0$ is the only possible rational eigenvalue of $U$, we conclude that $x_j = 0$ for all $j \neq 0$. Therefore, $x = x_0 \in \mathbf{G}(\mathbf{k})$. Hence the relation $x \cdot \left(t^{-1}\, C \right) = t^{-1}\, D$ implies that $\text{Ad}(x) \, C = D$. By uniqueness of Jordan decomposition for $\mathbf{\text{GL}_n}$, this means that $\text{Ad}(x)\, C_s = D_s$.

It is well-known that $\text{Lie}(\mathbf{T})/ W$ parametrizes semisimple conjugacy classes in $\mathfrak{g}$ (see \cite{collingwood.nilpotent} Chapter 2). In particular, $\mathfrak{D}$ is a set of representatives of conjugacy classes of semisimple elements that map to $0$ under $\pi$. We conclude that we must have $C_s = D_s$. Then $\text{Ad}(x) \, C_s = D_s$ implies that $x \in Z_{\mathbf{G}}(C_s)(\mathbf{k})$.
\end{proof}
\subsection{Regular connections for tori and reductive groups}
\begin{prop} \label{prop:tori regular} Let $\mathbf{G}$ be a torus and $A = \sum_{j=-1}^{\infty} A_j \, t^j$ a formal connection of the first kind. Then there exists $x \in \mathbf{G}(\mathcal{O})$ such that $x \cdot A = t^{-1} A_{-1}$. Moreover, there is a unique such $x$ with $x \equiv 1 \, \left( mod \; t\right)$.
\end{prop}
\begin{proof}
Since $\mathbf{k}$ is algebraically closed, $\mathbf{G}$ is split. Therefore the theorem follows from the special case $\mathbf{G} = \mathbb{G}_m$. We are reduced to an elementary computation. Let $v = \sum_{j =0}^{\infty} A_j t^j  \; \in \, \mathcal{O}$. It suffices to find $u = \sum_{j =0}^{\infty} B_j t^j \in \mathcal{O}^{\times}$ with $\frac{d}{dt}(u) = -vu $. By expanding we see that we want $(j+1) B_{j+1} = -\sum_{l=0}^{j} A_l B_{j-l}$ for all $j \geq 0$. This is a recurrence we can solve, because we are in characteristic $0$. We can set the initial condition $B_0 = 1$ and then the rest of the coefficients are uniquely determined.
\end{proof}
\begin{example}
For $\mathbf{G} = \mathbb{G}_m$ we can phrase this result concretely in terms of differential equations. In this case we have an equation $\frac{d}{dt} x = A x$, where $A \in \mathbf{k}\pseries{t}$ is a Laurent series with at worst a simple pole. The statement says that we can do a multiplicative change of variables $y = Bx$ for some power series $B \in \mathcal{O} ^{\times}$ such that the equation becomes $\frac{d}{dt} y = \frac{a}{t}y$ for some scalar $a \in \mathbf{k}$.
\end{example}
Let's state a uniqueness result for canonical forms of regular formal connections for tori.
\begin{prop} \label{prop: uniqueness regular tori}
Let $\mathbf{G}$ be a torus, and let $C_1, C_2 \in \mathfrak{g}$. Suppose that there exists $x \in \mathbf{G}(F)$ with $x \cdot \left(t^{-1}\, C_1 \right) = t^{-1}\, C_2$. Then, we have $x = g \, t^{\mu}$ for some cocharacter $\mu \in X_{*}(\mathbf{G})$ and some $g \in \mathbf{G}(\mathbf{k})$. Moreover $C_1 = C_2 - \mu$.
\end{prop}
\begin{proof}
We will do the computation for $\mathbf{G} = \mathbb{G}_m$. The general case follows from the same argument. Write $x = k \, t^r \, y$, where $k \in \mathbf{k}^{\times}$ and $y = 1+ \sum_{j=1}^{\infty} a_j \, t^j$. Then,
\[ x \cdot \left(t^{-1} \, C_1 \right)\; = \; t^{-1} \; C_1 + rt^{-1} + dy \, y^{-1} \; = \; t^{-1}\, C_2\]
Notice that $dy \, y^{-1}$ is in $\mathbf{k}\bseries{t}$. By looking at the nonnegative coefficients in the equation above, we conclude that $dy=0$. Therefore we have $y =1$. Hence $x = k\,t^r$, and the result follows.
\end{proof}
We can patch together some of the previous of results to get canonical forms for regular connections when the group is reductive. The following corollary is \cite{schnurer.regular}[Thm. 4.2].
\begin{coroll} \label{thm:reductive regular} 
Let $\mathbf{G}$ be reductive and $A \in \mathfrak{g}_F$ a regular formal connection. Then there exists $x \in \mathbf{G}(\overline{F})$ such that $x \cdot A = t^{-1} C$ for some $C \in \mathfrak{g}$.
\end{coroll}
\begin{proof}
For completness we explain how to deduce the corollary from previous propositions. We can assume that $A$ is of the first kind. Let $\mathbf{Z}^0_G$ be the neutral component of the center of $\mathbf{G}$. Set $\mathfrak{z}\vcentcolon = \text{Lie}(\mathbf{Z}^0_{\mathbf{G}})$. Let $\mathbf{G}_{\text{der}}$ the derived subgroup of $\mathbf{G}$. $\mathbf{G}_{der}$ is semisimple with Lie algebra $\mathfrak{g}_{\text{der}} \vcentcolon = [\mathfrak{g}, \mathfrak{g}]$. We have $\mathfrak{g} = \mathfrak{g}_{\text{der}} \oplus \mathfrak{z}$. Decompose $A = A_{\mathfrak{g}_{\text{der}}} + A_{\mathfrak{z}}$. By the semisimple case there exists $x \in \mathbf{G}_{\text{der}}(\overline{F})$ such that $x \cdot A_{\mathfrak{g}_{\text{der}}}$ is in canonical form. Now $x \cdot A = x \cdot A_{\mathfrak{g}_{\text{der}}} + A_{\mathfrak{z}}$. Use the result for tori to put $A_{\mathfrak{z}}$ in canonical form and conclude.
\end{proof}
\begin{remark}
By Remark \ref{remark: determinacy terms regular semisimple}, we only need to know $k\left(\, (A_{\mathfrak{g}_{\text{der}}})_{-1} \,\right)$-many coefficients of a connection of the first kind in order to determine its canonical form. The bound for the ramification needed in this case is reduced to the bound for the semisimple group $\mathbf{G}_{\text{der}}$ as explained in Lemma \ref{lemma: ramification bound regular semisimple}.
\end{remark}
Notice that the setup before Corollary \ref{coroll: canonical unique semisimple} applies to the reductive case. We formulate the analogous statement for convenience.
\begin{coroll} \label{coroll: uniqueness regular reductive}
Let $\mathbf{G}$ connected reductive with maximal torus $\mathbf{T} \subset \mathbf{G}$.\vspace{-0.25cm}
\begin{enumerate}[(i)]
    \item Let $A \in \mathfrak{g}_F$ be a regular connection. Then there exists $x \in \mathbf{G}(\overline{F})$ such that $x \cdot A = t^{-1} C$ for some $C \in \mathfrak{g}$ satisfying $C_s \in \mathfrak{D}$.
    
    \item Assume that $C,D \in \mathfrak{g}$ satisfy $C_s, D_s \in \mathfrak{D}$. Suppose that there exists $x \in \mathbf{G}(\overline{F})$ such that $x \cdot \left(t^{-1} \, C \right) = t^{-1} \, D$. Then, we have $C_s = D_s$. Moreover $x$ is in the centralizer $Z_{\mathbf{G}}(C_s)(\mathbf{k})$.
\end{enumerate}
\vspace{-0.25cm}
\end{coroll}
\begin{proof}
Part (i) follows by combining Proposition \ref{prop:tori regular} for tori and Corollary \ref{coroll: canonical unique semisimple} for semisimple groups. Part (ii) follows from the same argument as in Proposition \ref{prop: uniqueness regular semisimple}.
\end{proof}
This corollary allows us to give a concrete parametrization of regular $\mathbf{G}(\overline{F})$-gauge equivalence classes of formal connections. Let $A\in \mathfrak{g}_{F}$ be a regular formal connection. Suppose that $B= t^{-1}C$ is a connection in canonical form that is $\mathbf{G}(\overline{F})$-gauge equivalent to $A$. Assume that $C_s \in \mathfrak{D}$. By Corollary \ref{coroll: uniqueness regular reductive}, $C_s$ does not depend on the choice of canonical form $B$. Let $W$ denote the Weyl group of $\mathbf{G}$ with respect to $\mathbf{T}$. Recall that $\mathfrak{D}$ is a set of representatives for $\left(X_{*}(\mathbf{T}) \otimes \mathbf{k}/ \, \mathbb{Q}\right)/ \, W$. In particular we get a well defined element in $\left(X_{*}(\mathbf{T}) \otimes \mathbf{k}/ \, \mathbb{Q}\right)/ \, W$ corresponding to $C_s$.
\begin{defn} \label{defn: semisimple monodromy overline}
Let $A \in \mathfrak{g}_F$ be a regular formal connection as above. We define the semisimple $\overline{F}$-monodromy of $A$ to be the element $m^s_{A, \, \overline{F}} \in \left(X_{*}(\mathbf{T}) \otimes \mathbf{k}/ \, \mathbb{Q}\right)/ \, W$ corresponding to $C_s$ as described above.
\end{defn}
Let $m \in \left(X_{*}(\mathbf{T}) \otimes \mathbf{k}/ \, \mathbb{Q}\right)/ \, W$. We define $Z_{\mathbf{G}}(m)$ to be the centralizer in $\mathbf{G}$ of the unique representative of $m$ in $\mathfrak{D}$. It turns out that $Z_{\mathbf{G}}(m)$ is a Levi subgroup of a parabolic in $\mathbf{G}$. It is well-known that the Lie algebra centralizer of a semisimple element $\text{Lie}(Z_{\mathbf{G}}(m))= \mathfrak{g}_{m}$ is the Levi component of a parabolic subalgebra of $\mathfrak{g}$. For connectedness, we can pass to an isogenous cover $p: \tilde{\mathbf{G}} \longrightarrow \mathbf{G}$ with simply connected derived subgroup. Notice that $p(Z_{\tilde{\mathbf{G}}}(m)) = Z_{\mathbf{G}}(m)$. So it suffices to prove connectedness of $Z_{\tilde{\mathbf{G}}}(m)$, which follows from \cite{humphreys-conjugacy} pg. 33. Note that the isomorphism class of $Z_{\mathbf{G}}(m)$ does not depend on the choice of projection $\pi: \mathbf{k} \longrightarrow \mathbb{Q}$ and fundamental domain $\mathfrak{D}$. In fact, $Z_{\mathbf{G}}(m) \cong Z_{\mathbf{G}}(C)$ for any representative $C$ such that $\text{ad}(C)$ has no rational eigenvalues.

Corollary \ref{coroll: uniqueness regular reductive} implies that the nilpotent part of a canonical form $B$ that is $\mathbf{G}(\overline{F})$-gauge equivalent to $A$ is uniquely determined up to $Z_{\mathbf{G}}(m^s_{A, \, \overline{F}})$-conjugacy. We record this as a corollary.
\begin{coroll}
Fix $m \in \left(X_{*}(\mathbf{T}) \otimes \mathbf{k}/ \, \mathbb{Q}\right)/ \, W$. Let $\mathcal{N}_{Z_{\mathbf{G}}(m)}$ denote the nilpotent cone in the Lie algebra of $Z_{\mathbf{G}}(m)$. There is a natural correspondence
\[ \left\{ \text{regular} \; A \in \mathfrak{g}_{\overline{F}} \; \text{with} \; m^s_{A, \, \overline{F}} = m\right\}/\, \mathbf{G}(\overline{F}) \; \; \; \longleftrightarrow{\; \;\; \;} \mathcal{N}_{Z_{\mathbf{G}}(m)} / \, Z_{\mathbf{G}}(m) \]
\end{coroll}
Let $S \subset \Delta$ be a subset of simple roots. Each $\alpha \in S$ induces a linear functional $X_{*}(\mathbf{T}) \otimes \mathbf{k} \longrightarrow \mathbf{k}$. Let $H_{\alpha}$ be the hyperplane of $X_{*}(\mathbf{T}) \otimes \mathbf{k}$ where this functional vanishes. We denote by $\overline{H}_{\alpha}$ the image of $H_{\alpha}$ in $\left(X_{*}(\mathbf{T}) \otimes \mathbf{k}/ \, \mathbb{Q}\right)/ \, W$. Define $\overline{H}_{S} \vcentcolon =  \bigcap_{\alpha \in S} \overline{H}_{\alpha}$.  We say that $m \in \left(X_{*}(\mathbf{T}) \otimes \mathbf{k}/ \, \mathbb{Q}\right)/ \, W$ is of type $S$ if we have $m \in \overline{H}_{S}$ and $m \notin \overline{H}_{V}$ for all $S \subsetneq V \subset \Delta$. Let $Q_S$ be the set of all $m$ of type $S$. For any $m \in Q_S$, the centralizer $Z_{\mathbf{G}}(m)$ is conjugate to the standard Levi $\mathbf{M}_S$ associated to the subset of simple roots $S \subset \Delta$. We get the following rewording of the corollary above.
\begin{coroll} \label{coroll: parametrization overline regular connections}
There is a natural correspondence
\[ \left\{ \text{regular formal connections}\right\}/\, \mathbf{G}(\overline{F}) \; \; \; \longleftrightarrow{\; \;\; \;} \bigsqcup_{S \subset \Delta} Q_S \times \mathcal{N}_{\mathbf{M}_S} / \, \mathbf{M}_S \]
\end{coroll}
This gives us a procedure to describe all regular $\mathbf{G}(\overline{F})$-gauge equivalence classes of formal connections. For each $S \subset \Delta$, the group $\mathbf{M}_S$ is connected and reductive. The set of nilpotent orbits $\mathcal{N}_{\mathbf{M}_S}/ \, \mathbf{M}_S$ is a finite set which admits many well studied parametrizations. For example, nilpotent orbits for classical groups can be classified by partition diagrams as in  \cite{collingwood.nilpotent} Chapter 5. This yields an explicit canonical block decompositions for $\mathbf{G}(\overline{F})$-gauge equivalence classes of regular connections for classical groups.
\subsection{Connections for unipotent groups}
All connections in a unipotent group are regular. They admit canonical forms.
\begin{prop} \label{thm:unipotent regular} Let $\mathbf{G}$ be connected unipotent and let $A \in \mathfrak{g}_F$ be a connection. Then, there exists $x \in \mathbf{G}(F)$ such that $x \cdot A = t^{-1} C$ for some $C \in \mathfrak{g}$.
\end{prop}
\begin{proof}
We proceed by induction on $\text{dim}(\mathbf{G})$. Suppose that $\text{dim}(\mathbf{G})= 1$. Since $\text{char}(\mathbf{k}) = 0 $, we know that $\mathbf{G} \cong \mathbb{G}_a$. In this case the theorem follows from an elementary computation. See Example \ref{example: concrete additive group} below.

Now assume $\text{dim} \, (\mathbf{G}) \geq 2$. Since $\mathbf{k}$ is of characteristic 0, $\mathbf{G}$ is split unipotent. In particular, the center $\mathbf{Z}_{\mathbf{G}}$ contains a closed subgroup $\mathbf{H}$ isomorphic to $\mathbb{G}_a$. Let $\text{Lie}(\mathbf{H}) = \mathfrak{h}$. By induction there exists $\overline{x} \in \mathbf{G}/\mathbf{H}(F)$ such that $\overline{x} \cdot \overline{A} \in \mathfrak{g}/\mathfrak{h}$ is in canonical form. We can lift $\overline{x}$ to an element $x \in \mathbf{G}(F)$ because $H^1 \hspace{-1 mm}\left(F, \, \mathbf{H}(\overline{F})\right) = H^1 \hspace{-1mm}\left(F, \, \overline{F} \right) = 0$ (the additive version of Hilbert's Theorem 90). By construction, we have $x \cdot A = t^{-1} C + B$ for some $C \in \mathfrak{g}$ and $B \in \mathfrak{h}_F$. Now we can use the base case for $\mathbf{H} \cong \mathbb{G}_a$ to put $B$ into regular canonical form.
\end{proof}
\begin{remark}
Here we didn't need to pass to a ramified cover in order to find a good trivialization.
\end{remark}
\begin{example} \label{example: concrete additive group}
In the case of $\mathbf{G} = \mathbb{G}_a$, we can phrase this result concretely in terms of differential equations. We use the embedding $\mathbb{G}_a \hookrightarrow \text{GL}_2$, so that we can interpret the connection as system of differential equations
\begin{align*}
\frac{d}{dt}x_1 &= A x_2\\
\frac{d}{dt}x_2 &= 0
\end{align*}
Set $x_2 = c$, where $c \in \mathbf{k}$ is a constant. We are left with the nonhomogeneous equation $\frac{d}{dt}x_1 = cA$ for some Laurent series $cA$. The statement of the proposition reduces to the obvious fact that we can find a formal antiderivative for any Laurent series with residue $0$ (i.e. $A_{-1}=0$).
\end{example}
We now prove uniqueness of the canonical form up to conjugacy.
\begin{prop} \label{prop: unipotent uniqueness}
Let $\mathbf{U}$ be a unipotent group. Let $C_1, C_2$ be two elements of the Lie algebra $\mathfrak{u} \vcentcolon = \text{Lie}(\mathbf{U})$. Suppose that there exists $x \in \mathbf{U}(F)$ such that $x \cdot \left(t^{-1} C_1\right) = t^{-1} C_2$. Then, we have $x \in \mathbf{U}(\mathbf{k})$.
\end{prop}
\begin{proof}
We will argue by induction on the dimension of $\mathbf{U}$. If $\text{dim}(\mathbf{U}) =1$, then $\mathbf{U} \cong \mathbb{G}_a$. We can write $x=\sum_{j =m}^{\infty} a_j \, t^j$ for some $a_j \in \mathbf{k}$. The hypothesis then becomes
\[x \cdot \left(t^{-1} C_1\right) = t^{-1} C_1 + dx = t^{-1}C_2\]
This means that $dx = \sum_{j =m}^{\infty} j a_j \, t^{j-1} = t^{-1}\left(C_2 - C_1\right)$. In particular we must have $j a_j = 0$ for all $j \neq 0$. Hence $x = a_0 \in \mathbf{k}$.

Suppose that $\mathbf{U}$ is an arbitrary unipotent group. Assume that the result holds for all unipotent groups of smaller dimension. Let $\mathbf{H}$ be a subgroup of the center $Z_{\mathbf{U}}$ of $\mathbf{U}$ such that $\mathbf{H} \cong \mathbb{G}_a$ (this is possible because  $\text{char} (\mathbf{k}) = 0$) . Let $\overline{x} \in \mathbf{U}/\mathbf{H} \, (F)$ be the image of $x$ in the quotient. By the induction hypothesis, the proposition holds for $\mathbf{U} / \, \mathbf{H}$. Hence we have that $\overline{x} \in \mathbf{U}/\mathbf{H} \, (\mathbf{k}) $. We can lift $\overline{x}$ to an element $v \in \mathbf{U}(\mathbf{k})$, since $\mathbf{k}$ is algebraically closed.

We can therefore write $x = vu$, with $u \in \mathbf{H} (F)$. Our assumption thus becomes
\[ x\cdot \left(t^{-1}C_1\right) \; = \; t^{-1}\text{Ad}(v) \, \text{Ad}(u)\, C_1 + \text{Ad}(v) du \; = \; t^{-1} C_2  \]
Since $u \in \mathbf{H}(F) \subset Z_{\mathbf{U}}(F)$, we have $\text{Ad}(u) C_1 = C_1$. After rearranging we get
\[ du\; =\; t^{-1} \left(\text{Ad}(v^{-1})\, C_2 - C_1 \right)\]
The computation for $\mathbb{G}_a$ above implies that $u \in \mathbf{H}(\mathbf{k})$. Therefore $x \in \mathbf{U}(\mathbf{k})$.
\end{proof}
We end this section with a determinacy result for canonical forms in the unipotent case. Recall that the nilpotency class of a unipotent group is the length of the upper central series. For example, a commutative unipotent group has nilpotency class $0$.
\begin{prop} \label{prop: determinacy unipotent}
Let $\mathbf{U}$ be a unipotent group of nilpotency class $n$. Let $A = \sum_{j=m}^{\infty} A_j \, t^j \in \mathfrak{u}_F$ be a connection with $A_m \neq 0$.\vspace{-0.25cm}
\begin{enumerate}[(i)]
    \item If $m > -1$, then there exists $x \in \mathbf{U}(\mathcal{O})$ such that $x \equiv 1 \; \left(mod \; t^{m+1}\right)$ and  $x \cdot A =0$.
    \item If $m \leq -1$, then the gauge equivalence class of $A$ is determined by $A_j$ for $m\leq j < n(|m|-1)$. More precisely, suppose that $B$ is another connection with $B \equiv A \; \left( mod \; t^k\right)$ for some $k \geq n(|m|-1)$. Then there exists $x \in \mathbf{U}(\mathcal{O})$ with $x \equiv 1 \left(mod \; t^{k-n|m|+n+1} \right)$ such that $x \cdot A = B$. 
\end{enumerate}
\vspace{-0.25cm}
\end{prop}
\begin{proof}
We will induct on the nilpotency class $n$. The base case $n = 0$ means that $\mathbf{U} \cong \mathbb{G}_a^l$ for some $l$. Here we can make use of the explicit computation we have done for $\mathbb{G}_a$ a few times already. Define $u_A \vcentcolon= - \sum_{j=0}^{\infty} \frac{1}{j+1} A_{j+1} \, t^{j+1}$. We have 
\[ u_A \cdot A = A + du_A = \sum_{j=m}^{-1} A_j \, t^j  \]
Now both (i) and (ii) are clear by taking $x = -u_B + u_A$ (we use $B=0$ for part (i)).

For the induction step, let $\mathbf{U}$ be an arbitrary unipotent group of nilpotency class $n$. We will think of $\mathbf{U}$ as embedded in the group of upper triangular matrices of $\text{GL}_p$ for some $p$. By definition, the quotient $\mathbf{U} / \, Z_{\mathbf{U}}$ of $\mathbf{U}$ by its center $Z_{\mathbf{U}}$ has nilpotency class $n-1$. It follows from the matrix description that we can choose a section $s$ over $\mathbf{k}$ for the $Z_{\mathbf{U}}$-torsor $\mathbf{U} \longrightarrow \mathbf{U} / \, Z_{\mathbf{U}}$ such that $s(1) = 1$
(this is a section as $\mathbf{k}$-schemes, it is not a homomorphism).

Let's address part (i). Let $\overline{A}$ be the image of $A$ in the quotient $\text{Lie}(\mathbf{U} / \, Z_{\mathbf{U}})_{F}$. By the induction hypothesis, there exists $\overline{x} \in \mathbf{U} / \, Z_{\mathbf{U}}(\mathcal{O})$ such that $\overline{x} \equiv 1 \; \left(mod \; t^{m+1}\right)$ and $\overline{x} \cdot \overline{A} =0$. Therefore, we have $s(\overline{x}) \cdot A \in \text{Lie}(Z_{\mathbf{U}})_F$. Notice that we have $s(\overline{x}) \equiv s(\overline{x})^{-1} \equiv 1 \; \left(mod \; t^{m+1}\right)$. It follows that
\[ s(\overline{x}) \cdot A \; = \; s(\overline{x}) \, A \, s(\overline{x})^{-1}\; + \; d \,s(\overline{x}) \; s(\overline{x})^{-1} \; \equiv 0 \; \left(mod \; t^{m}\right)\]
Now we can conclude by using the base case for $Z_{\mathbf{U}}$.
For part (ii), let $\overline{A}$ and $\overline{B}$ denote the images of $A$ and $B$ in the quotient. By the induction hypothesis, there exists $\overline{x} \in \mathbf{U}/ \, Z_{\mathbf{U}}(\mathcal{O})$ with $\overline{x} \equiv 1 \; \left( mod \; t^{k-(n-1)|m| +n} \right)$ such that $\overline{x} \cdot \overline{A} = \overline{B}$.  We can now write $s(\overline{x}) \cdot A = ds\left(\overline{x} \cdot \overline{A} \right) +C$ and $B = ds\left(\overline{x} \cdot \overline{A}\right) +D$ for some $C, D \in \text{Lie}(Z_{\mathbf{U}})_F$.

Notice that $s(\overline{x}) \equiv s(\overline{x})^{-1} \equiv 1 \; \left(mod \; t^{k -(n-1)|m| +n}\right)$. Therefore,
\[ s(\overline{x}) \cdot A \; = \; s(\overline{x}) \, A \, s(\overline{x})^{-1}\; + \; d \,s(\overline{x}) \; s(\overline{x})^{-1} \; \equiv \; A\; \left(mod \; t^{k-n|m|+n}\right)\]
Since $A \equiv B \;\left(mod \; t^{k-n|m|+n}\right)$, it follows that $C \equiv D \; \left(mod \; t^{k-n|m|+n} \right)$. Now by the base case we can find $y \in Z_{\mathbf{U}}(\mathcal{O})$ with $y \equiv 1 \; \left(mod \; t^{k-n|m|+n+1}\right)$ such that $y \cdot C = D$. We conclude that $y\, s(\overline{x}) \cdot A = B$, because $y$ is in the center.  We clearly have $y \, s(\overline{x}) \equiv 1 \; \left(mod \; t^{k-n|m|+n+1}\right)$, as desired.
\end{proof}
\subsection{Regular connections for solvable groups}
We fix a projection $\pi : \mathbf{k} \longrightarrow \mathbb{Q}$ as in the proof of Proposition \ref{thm:semisimple regular}. For $\mathbf{T}$ a torus, we extend this projection to a map $\pi : \text{Lie}(\mathbf{\mathbf{T}}) \cong X_{*}(\mathbf{T}) \otimes \mathbf{k} \longrightarrow X_{*}(\mathbf{T})\otimes\mathbb{Q}$.

\begin{prop} \label{prop: reduction regular solvable}
Let $\mathbf{G}$ be of the form $\mathbf{T} \ltimes \mathbf{U}$, where $\mathbf{T}$ is a torus and $\mathbf{U}$ is unipotent. Let $A = A_{\mathbf{T}} + A_{\mathbf{U}}$ be a formal connection with $A_{\mathbf{T}} \in \text{Lie}(\mathbf{T})_F$ a connection of the first kind and $A_{\mathbf{U}} \in \text{Lie}(\mathbf{U})_F$. Let $b$ be a positive integer such that $b \, \pi\left(\, (A_{\mathbf{T}})_{-1} \, \right) \in X_{*}(\mathbf{T})$. 
Then there exists $x \in \mathbf{G}(F_b)$ such that $x \cdot A = t^{-1} \, C_{\mathbf{T}} + t^{-1} \, C_{\mathbf{U}}$ for some $C_{\mathbf{T}} \in \text{Lie}(\mathbf{T})$ and $C_{\mathbf{U}} \in \text{Lie}(\mathbf{U})$. Moreover, we can arrange that $\pi(C_{\mathbf{T}}) = 0$ and $[C_{\mathbf{T}}, C_{\mathbf{U}}] =0$.
\end{prop}
\begin{proof}
By the proof of Proposition \ref{prop:tori regular}, we can find $g \in \mathbf{T}(F)$ with $g \cdot A_{\mathbf{T}} = t^{-1} \, (A_{\mathbf{T}})_{-1}$. Set $\mu \vcentcolon = b \, \pi\left(\, (A_{\mathbf{T}})_{-1} \, \right) \in X_{*}(\mathbf{T})$. Then we have $(t^{\frac{1}{b} \mu} \, g) \cdot A_{\mathbf{T}} = t^{-1} \, C_{\mathbf{T}}$ for some $C_{\mathbf{T}} \in \text{Lie}(\mathbf{T})$  with $\pi(C_{\mathbf{T}}) =0$.

We can replace $A$ with $B \vcentcolon = (t^{\frac{1}{b} \, \mu} \, g) \cdot A$. We know that $B = t^{-1} \, C_{\mathbf{T}} + B_{\mathbf{U}}$ for some $B_{\mathbf{U}} \in \text{Lie}(\mathbf{U})_{F_b}$. By lifting to the $b$-ramified cover, we can assume that $B_{\mathbf{U}} \in \text{Lie}(\mathbf{U})_F$. We claim that we can find $u \in \mathbf{U}(F)$ such that $u \cdot B = t^{-1}\, C_{\mathbf{T}} + t^{-1} \, C_{\mathbf{U}}$ with $C_{\mathbf{U}} \in \text{Lie}(\mathbf{U})$ and $[C_{\mathbf{T}}, C_{\mathbf{U}}] =0$. We will show this by induction on the dimension of $\mathbf{U}$.

The base case is $\mathbf{U} = \mathbb{G}_a$. Then, $\mathbf{T}$ acts on $\mathbf{U}$ by a character $\chi : \mathbf{T} \longrightarrow \mathbb{G}_m$. Write $B_{\mathbf{U}} = \sum_{j=r}^{\infty} \left(B_{\mathbf{U}}\right)_j \, t^j$. For any $u = \sum_{j=r}^{\infty} u_j \, t^j \in \mathbf{U}(F)$, we have
\[u \cdot B = t^{-1}\, C_{\mathbf{T}} \, + \, B_{\mathbf{U}} \, - \, \sum_{j=r}^{\infty} \left( d\chi(C_{\mathbf{T}}) -j\right) u_j \, t^{j-1}\]
Since $\pi(C_\mathbf{T}) =0$, we have $\pi \left( d\chi(C_{\mathbf{T}})\right) = 0$. There are two options:\vspace{-0.25cm}
\begin{enumerate}[(1)]
    \item $d\chi(C_{\mathbf{T}}) \notin \mathbb{Q}$. Then, setting $u_j = \frac{1}{ d\chi(C_{\mathbf{T}}) -j} \left(B_{\mathbf{U}}\right)_{j-1}$ we get $u \cdot B = t^{-1} \, C_{\mathbf{T}}$.
    
    \item $d\chi(C_{\mathbf{T}}) = 0$. We can set $u_j = \frac{1}{ d\chi(C_{\mathbf{T}})\, - \,j} \left(B_{\mathbf{U}}\right)_{j-1}$ for $j \neq 0$ and $u_0 =0$. Then $u \cdot B = t^{-1} \, C_{\mathbf{T}} + t^{-1} \, \left(B_{\mathbf{U}}\right)_{-1}$. Notice that we have $[C_{\mathbf{T}} , \, \left(B_{\mathbf{U}}\right)_{-1}] = d\chi(C_{\mathbf{T}}) \left(B_{\mathbf{U}}\right)_{-1} = 0$.
\end{enumerate}
\vspace{-0.25cm}
This concludes the proof of the base case.

Let's proceed with the induction step. We can decompose the action of the split torus $\mathbf{T}$ on the vector space $Z_{\mathbf{U}}$ into one-dimensional spaces. Let $\mathbf{H} \cong \mathbb{G}_a \leq Z_{\mathbf{U}}$ be one of these eigenspaces. The eigenspace decomposition yields a natural $\mathbf{T}$-equivariant section of the quotient map $Z_{\mathbf{U}} \longrightarrow Z_{\mathbf{U}} / \, \mathbf{H}$. We claim that we can extend this to a  $\mathbf{T}$-equivariant section $s$ of the morphism of schemes $\mathbf{U} \longrightarrow \mathbf{U} / \, \mathbf{H}$. In order to see this claim, we can use induction on the nilpotency class to reduce to the case when $\mathbf{U}$ has nilpotency class $1$. Notice that we can find a section which is not necessarily $\mathbf{T}$-equivariant, since everything is isomorphic to an affine space. Then we can use the argument in Lemma 9.4 of \cite{borel.springer} to obtain a $\mathbf{T}$-equivariant section. We can arrange so that $s$ preserves the identities by substracting the image $s(1_{\mathbf{U} / \, \mathbf{H}})$. Let us denote by $ds$ the induced map of tangent spaces at the identity.

Let $\overline{B}$ be the image of $B$ in the quotient $\text{Lie}(\mathbf{T} \ltimes \mathbf{U} / \, \mathbf{H})_F$. By the induction hypothesis, we can find $\overline{u} \in \mathbf{U} / \, \mathbf{H} (F)$ such that $\overline{u} \cdot \overline{B} = t^{-1} \, C_{\mathbf{T}} + t^{-1} \, \overline{D}$ for some $\overline{D} \in \text{Lie}\left(\mathbf{U}/ \, \mathbf{H}\right)$ with $[C_{\mathbf{T}}, \overline{D}] =0$. We can then write
\[ s(\overline{u}) \cdot B = t^{-1}\, C_{\mathbf{T}} + t^{-1}\, ds(\overline{D}) + B_{\mathbf{H}}\]
for some $B_{\mathbf{H}} \in \text{Lie}(\mathbf{H})_F$. Since $s$ is $\mathbf{T}$-equivariant, we have $[ ds(\overline{D}), \, C_{\mathbf{T}}] = 0$. We can now use the base case for $\mathbf{H}$ in order to conclude.
\end{proof}

\begin{remark} \label{remark: not uniform b in solvable case}
We can decompose the Lie algebra $\mathfrak{u} \vcentcolon = \text{Lie}(\mathbf{U})$ into weight spaces $\mathfrak{u} = \bigoplus_{i} \mathfrak{u}_{\chi_i}$. Here each $\chi_i$ is a character of $\mathbf{T}$. Fix a basis $\{\alpha_j\}$ for the character lattice $X^*(\mathbf{T})$. For each $i$ we can write $\chi_i = \sum_{j} m^i_j \alpha_j$ for some integers $m^i_j$. Define $\text{hgt}(\chi_i) = \sum_j |m^i_j|$. Set $b =  \underset{1 \leq i \leq l}{\text{max}} \, \text{hgt}(\chi_i)$. If we don't require $\pi(C_{\mathbf{T}})$ and $[C_{\mathbf{T}}, C_{\mathbf{U}}] =0$ in Proposition \ref{prop: reduction regular solvable}, then it suffices to pass to a $b$-ramified cover. So there is a uniform upper bound on the ramification needed to put any regular $\mathbf{G}$-connection into canonical form. It only depends on the solvable group $\mathbf{G}$.
\end{remark}

Let us prove a uniqueness result for regular canonical forms in the solvable case.
\begin{prop} \label{prop: uniqueness regular solvable}
Let $\mathbf{G}$ be of the form $\mathbf{T} \ltimes \mathbf{U}$ as above. Let $C = t^{-1} \, C_{\mathbf{T}} +  t^{-1} \, C_{\mathbf{U}}$ and $D = t^{-1} \, D_{\mathbf{T}} +  t^{-1} \, D_{\mathbf{U}}$ be two regular canonical connections with $C_{\mathbf{T}}, D_{\mathbf{T}} \in \text{Lie}(\mathbf{T})$ and $C_{\mathbf{U}}, D_{\mathbf{U}} \in \text{Lie}(\mathbf{U})$. Suppose that $\pi(C_{\mathbf{T}}) = \pi(D_{\mathbf{T}}) = 0$ and $[C_{\mathbf{T}}, C_{\mathbf{U}}] = [D_{\mathbf{T}}, D_{\mathbf{U}}] =0$. If there exists $x \in \mathbf{G}(\overline{F})$ such that $x \cdot C = D$, then in fact $C_{\mathbf{T}} = D_{\mathbf{T}}$. Moreover, $x$ is in the centralizer $Z_{\mathbf{G}}(C_{\mathbf{T}})(\mathbf{k})$ of $C_{\mathbf{T}}$.
\end{prop}
\begin{proof}
By lifting to a ramified cover, we can assume that $x \in \mathbf{G}(F)$. Write $x = x_{\mathbf{U}} \, x_{\mathbf{T}}$ with $x_{\mathbf{U}} \in \mathbf{U}(F)$ and $x_{\mathbf{T}} \in \mathbf{T}(F)$. By the computation in Proposition \ref{prop: uniqueness regular tori} applied to $\mathbf{T}$, we get that $x_{\mathbf{T}} \in \mathbf{T}(\mathbf{k})$ and $C_{\mathbf{T}} = D_{\mathbf{T}}$. The same proof of Proposition \ref{prop: uniqueness regular semisimple} implies that $x \in \mathbf{G}(\mathbf{k})$ and $\text{Ad}(x) C_{\mathbf{T}} = D_{\mathbf{T}}$. Since $C_{\mathbf{T}} = D_{\mathbf{T}}$, this means that $x \in Z_{\mathbf{G}}(C_{\mathbf{T}})(\mathbf{k})$.
\end{proof}
We conclude this section with a determinacy result for regular connections in the case of solvable groups. But first we need to setup some notation. Let $\mathbf{G} = \mathbf{T} \ltimes \mathbf{U}$ solvable. We have an action of the split torus $\mathbf{T}$ on the Lie algebra $\mathfrak{u} \vcentcolon = \text{Lie}(\mathbf{U})$ via the adjoint representation. We can decompose this representation into weight spaces $\mathfrak{u} = \bigoplus_{i =1}^{l} \mathfrak{u}_{\chi_i}$ for some finite set $\{ \chi_1, \chi_2, ..., \chi_l \}$ of charaters $\chi_i :\mathbf{T} \longrightarrow \mathbb{G}_m$.

Suppose that we have a formal connection $A = A^{\mathbf{T}} + A^{\mathbf{U}}$, with $A^{\mathbf{T}} \in \text{Lie}(\mathbf{T})_F$ a connection of the first kind and $A^{\mathbf{U}} \in \text{Lie}(\mathbf{U})_F$. We can write 
$A^{\mathbf{T}} = t^{-1} A^{\mathbf{T}}_{-1} \, + \, \sum_{j =p}^{\infty} A^{\mathbf{T}}_j \, t^j$ for some $p \geq 0$ and $A^{\mathbf{U}} = \sum_{j =m}^{\infty} A^{\mathbf{U}}_j \, t^j$ for some $m \in \mathbb{Z}$. Let $b$ be a positive integer such that $\mu \vcentcolon = b \, \pi\left(\, A^{\mathbf{T}}_{-1} \, \right)$ is in $X_{*}(\mathbf{T})$. Define $L$ to be
\[L \vcentcolon = \text{max} \left(\left\{ \frac{1}{b} \langle \mu, \, \chi_i \rangle \right\}_{i =1}^l \cup \{0\} \right) \]
\begin{prop} \label{prop: determinacy regular solvable}
Keep the same notation as in the paraggraph above. Assume that $\mathbf{U}$ has nilpotency class $n$.\vspace{-0.25cm}
\begin{enumerate}[(i)]
    \item Suppose that $m > L-1$. Then there exists $x \in \mathbf{G}(\mathcal{O})$ with $x \cdot A = t^{-1} A^{\mathbf{T}}_{-1}$. More precisely, there exists $x_{\mathbf{T}} \in \mathbf{T}(\mathcal{O})$ with $x_{\mathbf{T}} \equiv 1_{\mathbf{T}} \, \left(mod \; t^{p +1}\right)$ and $x_{\mathbf{U}} \in \mathbf{U}(\mathcal{O})$ with $x_{\mathbf{U}} \equiv 1_{\mathbf{U}} \, \left(mod \; t^{m +1}\right)$ such that $(x_{\mathbf{U}} x_{\mathbf{T}}) \cdot A = t^{-1} A^{\mathbf{T}}_{-1}$.
    \item Suppose that $m \leq L-1$. The $\mathbf{G}(F)$-gauge equivalence class of $A$ is determined by the coefficients $A_j^{\mathbf{T}}$ for $-1 \leq j < (n+1)(|m|-1)+L$ and $A_j^{\mathbf{U}}$ for $m \leq j < n(|m|-1) +L$. More precisely, suppose that there is another connection $B$ and positive integer $k \geq n(|m|-1)+L$ with $B^{\mathbf{T}} \equiv A^{\mathbf{T}} \, \left( mod \; t^{k +|m|}\right)$ and $B^{\mathbf{U}} \equiv A^{\mathbf{U}} \, \left( mod \; t^{k}\right)$. Then, there exists $x \in \mathbf{G}(\mathcal{O})$ with $x \equiv 1 \, \left(mod \; t^{k - n|m| +n+ 1}\right)$ such that $x \cdot A = B$.
\end{enumerate}
\vspace{-0.25cm}
\end{prop}
\begin{proof}
\begin{enumerate}[(i)]
    \item By assumption $A^{\mathbf{T}} \equiv t^{-1}A^{\mathbf{T}}_{-1} \, \left( mod \; t^p \right)$. The proof of Proposition \ref{prop:tori regular} shows that there exists $x_{\mathbf{T}} \in \mathbf{T}(\mathcal{O})$ with $x_{\mathbf{T}} \equiv 1 \, \left(mod \; t^{k+1}\right)$ such that $x_{\mathbf{T}} \cdot A^{\mathbf{T}} = t^{-1}A^{\mathbf{T}}_{-1}$. Set $C \vcentcolon = x_{\mathbf{T}} \cdot A$. We can write $C = t^{-1}A_{-1}^{\mathbf{T}} +\text{Ad}(x_{\mathbf{T}})A^{\mathbf{U}}$. In order to ease notation, set $C^{\mathbf{U}} \vcentcolon = \text{Ad}(x_{\mathbf{T}})A^{\mathbf{U}}$. Since $A^{\mathbf{U}} \equiv 0 \, \left(mod \; t^{m}\right)$, we have $C^{\mathbf{U}} \equiv 0 \, \left(mod \; t^{m}\right)$.  We claim that there exists $x \in \mathbf{U}(\mathcal{O})$ with $x \equiv 1_{\mathbf{U}} \, \left(mod \; t^{m+1} \right)$ such that $x \cdot C = t^{-1} A^{\mathbf{T}}_{-1}$. This claim finishes the proof of part (i).

In order to prove the claim, we will induct on the nilpotency class of $\mathbf{U}$. The base case $n=0$ means that $\mathbf{U} \cong \mathbb{G}_a^d$ for some $d$. We can decompose into eigenvalues and look at each coordinate separately in order to reduce to the case $d=1$. Then there is a single weight space $\mathfrak{u}_{\chi_i}$. This case amounts to solving a recurrence as in the base case for the proof of Proposition \ref{prop: reduction regular solvable}. We want to find $x = \sum_{j=0}^{\infty} t^j u_j$ satisfying
\[ C^{\mathbf{U}}_{j-1} = \left(d\chi_i(A_{-1}^{\mathbf{T}}) - j\right)u_j \]
If $j \leq m$ then $C^{\mathbf{U}}_{j-1} = 0$ by assumption. So we can set $u_j = 0$. If $j \geq m+1$, then we have 
\[ \pi\left(d\chi_i(A^{\mathbf{T}}_{-1})\right) -j = \frac{1}{b}\langle \mu, \chi_i\rangle -j \leq L -m -1 \]
By assumption $L-m-1 <0$, so we must have $d\chi_i(A^{\mathbf{T}}_{-1}) -j \neq 0$. Hence we can set $u_j = \frac{1}{d\chi_i(A^{\mathbf{T}}_{-1}) \, - \, j} \,C^{\mathbf{U}}_{j-1}$. The base case follows.

For the induction step, let $Z_{\mathbf{U}}$ denote the center of $\mathbf{U}$. Let $s$ be a $\mathbf{T}$-equivariant section of the quotient $\mathbf{U} \longrightarrow \mathbf{U} / \, Z_{\mathbf{U}}$, as in the proof of Proposition \ref{prop: reduction regular solvable}. Let $\overline{C}$ be the image of $C$ in the quotient $\text{Lie}(\mathbf{T} \ltimes \mathbf{U} / \, Z_{\mathbf{U}})_{F}$. By the induction hypothesis, there exists $\overline{x} \in \mathbf{U} / \, Z_{\mathbf{U}}(\mathcal{O})$ such that $\overline{x} \equiv 1 \; \left(mod \; t^{m + 1} \right)$ and $\overline{x} \cdot \overline{C} = t^{-1}A^{\mathbf{T}}_{-1}$. We must then have $s(\overline{x}) \cdot C = t^{-1}A^{\mathbf{T}}_{-1} + D_{Z_{\mathbf{U}}}$ for some $D_{Z_{\mathbf{U}}} \in \text{Lie}(Z_{\mathbf{U}})_{F}$. By definition
\[ s(\overline{x}) \cdot C \; = \; t^{-1} \text{Ad}(s(\overline{x})) A^{\mathbf{T}}_{-1} \, + \, \text{Ad}(s(\overline{x})) C^{\mathbf{U}} \, + \, ds(\overline{x}) s(\overline{x})^{-1}  \]
We know that $s(\overline{x}) \equiv s(\overline{x})^{-1} \equiv 1 \; \left(mod \; t^{m + 1} \right)$. Also by assumption $C^{\mathbf{U}} \in \mathfrak{u}_{\mathcal{O}}$. It follows that
\[ s(\overline{x}) \cdot C \; \equiv \; t^{-1} A^{\mathbf{T}}_{-1} \, + \, C^{\mathbf{U}} \; \equiv \; t^{-1}C_{\mathbf{T}} \; \left(mod \; t^{m}\right)\]
Therefore $D_{Z_{\mathbf{U}}} \equiv 0 \; \left(mod \; t^{m}\right)$. Now we can conclude by using the base case for $Z_{\mathbf{U}}$.
    \item The hypothesis implies that $B^{\mathbf{T}}_{-1} = A^{\mathbf{T}}_{-1}$. The proof of Proposition \ref{prop:tori regular} shows that there exist $x_{\mathbf{T}} \in \mathbf{T}(\mathcal{O})$ with $x_{\mathbf{T}}  \equiv 1 \, \left(mod \; t^{k+|m|}\right)$ such that $x_{\mathbf{T}} \cdot A^{\mathbf{T}} = B^{\mathbf{T}}$. Set $C \vcentcolon = x_{\mathbf{T}} \cdot A$. We have $C = B^{\mathbf{T}} + \text{Ad}(x_{\mathbf{T}}) A^{\mathbf{U}}$. Define $C^{\mathbf{U}} \vcentcolon = \text{Ad}(x_{\mathbf{T}}) A^{\mathbf{U}}$.
    
    We know that $C^{\mathbf{U}} \equiv A^{\mathbf{U}} \, \left(mod \; t^{k}\right)$, because $x_{\mathbf{T}}\equiv 1 \, \left(mod \; t^{k+|m|}\right)$ and $A^{\mathbf{U}} \in t^m \mathfrak{u}_{\mathcal{O}}$. Therefore $C^{\mathbf{U}} \equiv B^{\mathbf{U}} \, \left(mod \; t^{k}\right)$ by assumption. We claim that there exists $u \in \mathbf{U}(\mathcal{O})$ with $u \equiv 1 \; \left(mod \; t^{k-n|m|+n+1} \right)$ such that $u \cdot C = B$. This claim concludes the proof of part (ii). In order to prove the claim, we will induct on the nilpotency class of $\mathbf{U}$. The base case $n=0$ follows from an argument similar to the one for part (i), we omit the details.
    
    For the induction step, let $Z_{\mathbf{U}}$ and $s$ be as in part (i). Let $\overline{C}$ and $\overline{B}$ denote the images of $C$ and $B$ in the quotient $\text{Lie}(\mathbf{T} \ltimes \mathbf{U} / \, Z_{\mathbf{U}})_{F}$. By the induction hypothesis, there exists $\overline{x} \in \mathbf{U}/ \, Z_{\mathbf{U}}(\mathcal{O})$ with $\overline{x} \equiv 1 \; \left( mod \; t^{k-(n-1)|m| +n} \right)$ such that $\overline{x} \cdot \overline{C} = \overline{B}$.  We can now write $s(\overline{x}) \cdot C = ds\left(\overline{B} \right) +E_{Z_{\mathbf{U}}}$ and $B = ds\left(\overline{B}\right) +K_{Z_{\mathbf{U}}}$ for some $E_{Z_{\mathbf{U}}}, F_{Z_{\mathbf{U}}} \in \text{Lie}(Z_{\mathbf{U}})_{F_b}$. By definition
\[ s(\overline{x}) \cdot C \; = \; t^{-1} \text{Ad}(s(\overline{x})) B^{\mathbf{T}} \, + \, \text{Ad}(s(\overline{x})) C^{\mathbf{U}} \, + \, ds(\overline{x}) s(\overline{x})^{-1}  \]
We know that $s(\overline{x}) \equiv s(\overline{x})^{-1} \equiv 1 \; \left(mod \; t^{k -(n-1)|m| +n}\right)$. Since $|m| \geq 1$, we conclude that
\[ ds\left(\overline{B}\right) +E_{Z_{\mathbf{U}}} \; = \; s(\overline{x}) \cdot C \; \equiv \; B^{\mathbf{T}} \, + \, C^{\mathbf{U}} \; = \; C \; \left(mod \; t^{k-n|m|+n}\right)\]
Since $k \geq k-n|m| +n$, we have $C \equiv B \;\left(mod \; t^{k-n|m|+n}\right)$. It follows that $E_{Z_{\mathbf{U}}} \equiv K_{Z_{\mathbf{U}}} \; \left(mod \; t^{k-n|m|+n} \right)$. Now by the base case we can find $y \in Z_{\mathbf{U}}(\mathcal{O})$ with $y \equiv 1 \; \left(mod \; t^{k-n|m|+n+1}\right)$ such that $(y\, s(\overline{x}))\cdot C = B$. We have $y \, s(\overline{x}) \equiv 1 \; \left(mod \; t^{k-n|m|+n+1}\right)$, as desired.
\end{enumerate}
\vspace{-0.5cm}
\end{proof}
\begin{remark}
Suppose that $\langle \mu, \chi_i \rangle >0$ for all $i$. It follows from the proof above that we can relax further the conditions on the coefficients of $A^{\mathbf{U}}$. Similarly, we can obtain sharper conditions for the coefficients of $A^{\mathbf{T}}$ in the case $0 \leq m \leq L-1$. We leave the details of these refinements to the interested reader.
\end{remark}
\begin{remark}
If $L = 0$, then the statement simplifies and we recover conditions similar to the unipotent case (Proposition \ref{prop: determinacy unipotent}).
\end{remark}
\begin{subsection}{Regular connections for arbitrary linear algebraic groups}
\begin{thm} \label{thm:general regular} Let $\mathbf{G}$ be a connected linear algebraic group. Let $A \in \mathfrak{g}_F$ be a regular connection. Fix a Levi subgroup $\mathbf{L}$ and maximal torus $\mathbf{T} \subset \mathbf{L}$. Then there exists $x \in \mathbf{G}(\overline{F})$ such that $x \cdot A = t^{-1} C$ for some $C \in \mathfrak{g}$. Moreover, such $x$ can be chosen so that the semisimple part $C_s$ of the Levi component satisfies $C_s \in \mathfrak{D}$ and $[C_s, C] =0$.
\end{thm}
\begin{proof}
Assume that $A$ is of the first kind. Let $\mathbf{U} \subset \mathbf{G}$ be the unipotent radical of $\mathbf{G}$ with Lie algebra $\mathfrak{u}$. Let $\mathfrak{l}$ be the Lie algebra of $\mathbf{L}$. We know that $\mathbf{G} = \mathbf{L} \ltimes \mathbf{U}$, and so $\mathfrak{g} = \mathfrak{l} \oplus \mathfrak{u}$. Decompose $A = A_{\mathfrak{l}} + A_{\mathfrak{u}}$. By the reductive group case, there exists $x \in \mathbf{L}(\overline{F})$ such that $x \cdot A_{\mathfrak{l}} = t^{-1} C$ for some $C \in \mathfrak{l}$ satisfying $C_s \in \mathfrak{D}$.

Let $C_n \in \mathfrak{l}$ denote the nilpotent part of $C$. Let $\mathbf{E}$ be the neutral component of the centralizer $Z_{\mathbf{T}}(C_n)$ of $C_n$ in $\mathbf{T}$. Note that $\mathbf{E}$ is a subtorus of $\mathbf{T}$ and $C_s \in \text{Lie}(\mathbf{E})$. Since $\text{char} (\mathbf{k}) = 0$, there is a unique connected one-dimensional unipotent subgroup $\mathbf{N}$ of $\mathbf{L}$ with $C_n \in \text{Lie}(\mathbf{N})$. We have that $x \cdot A$ is a formal connection for the solvable group $(\mathbf{E} \times \mathbf{N})\ltimes \mathbf{U}$. Now the result follows from the solvable case (Proposition \ref{prop: reduction regular solvable}).
\end{proof}
\begin{remark}
In the beginning of the proof above, let $X$ denote the semisimple part of $(A_{\mathfrak{l}})_{-1}$. After conjugating by an element of $\mathbf{L}(\mathbf{k})$, we can suppose that $X \in \text{Lie}(\mathbf{T})$. Let $b$ be a positive integer such that $\mu \vcentcolon = b \, \pi(X)$ is in $X_{*}(\mathbf{T})$. Then, we can take $x \in G(F_b)$ in the proof above. In order to see this we can first apply $t^{-\frac{1}{b} \, \mu}$. So we can assume that $\pi(X) = 0$. By the proofs of Theorem \ref{thm:semisimple regular} and Proposition \ref{prop: reduction regular solvable}, it follows that we don't need any further ramification to put $A$ into canonical form.
\end{remark}
\begin{prop} \label{prop: uniqueness regular arbitrary}
Let $\mathbf{G}$ be a connected linear algebraic group. Fix a Levi subgroup $\mathbf{L}$ and maximal torus $\mathbf{T} \subset \mathbf{L}$. Let $C,D \in \mathfrak{g}$. Write  $C_s, D_s$ for the semisimple parts of the Levi components $C_{\mathfrak{l}}, D_{\mathfrak{l}}$. Assume that $C_s, D_s \in \mathfrak{D}$ and $[C_s, C] = [D_s, D] = 0$. Suppose that there exists $x \in \mathbf{G}(\overline{F})$ such that $x \cdot \left(t^{-1} \, C \right) = t^{-1} \, D$. Then, we have $C_s = D_s$. Moreover $x$ is in the centralizer $Z_{\mathbf{G}}(C_s)(\mathbf{k})$.
\end{prop}
\begin{proof}
Write $x = x_{\mathbf{U}} \, x_{\mathbf{L}}$ with $x_{\mathbf{U}} \in \mathbf{U}(\overline{F})$ and $x_{\mathbf{L}} \in \mathbf{L}(\overline{F})$. By Corollary \ref{coroll: uniqueness regular reductive} applied to $\mathbf{L}$, we get that $x_{\mathbf{L}} \in Z_{\mathbf{L}}(C_s)(\mathbf{k})$ and $C_s = D_s$. The same proof as in Proposition \ref{prop: uniqueness regular semisimple} shows that $x \in \mathbf{G}(\mathbf{k})$ and $\text{Ad}(x)C_s = D_s$. Since $C_s = D_s$, we conclude that $x \in Z_{\mathbf{G}}(C_s)(\mathbf{k})$.
\end{proof}
Let $A\in \mathfrak{g}_{F}$ be a regular formal connection. Proposition \ref{prop: uniqueness regular arbitrary} implies that we can define the semisimple $\overline{F}$-monodromy $m^s_A \in \left(X_{*}(\mathbf{T}) \otimes \mathbf{k}/ \, \mathbb{Q}\right)/ \, W$ just as we did in Definition \ref{defn: semisimple monodromy overline}. The same reasoning as in the reductive case yields the following.
\begin{coroll}
Fix $m \in \left(X_{*}(\mathbf{T}) \otimes \mathbf{k}/ \, \mathbb{Q}\right)/ \, W$. Let $\mathcal{N}_{Z_{\mathbf{G}}(m)}$ denote the nilpotent cone in the Lie algebra of $Z_{\mathbf{G}}(m)$. There is a natural correspondence
\[ \left\{ \text{regular} \; A \in \mathfrak{g}_{\overline{F}} \; \text{with} \; m^s_A = m\right\}/\, \mathbf{G}(\overline{F}) \; \; \; \longleftrightarrow{\; \;\; \;} \mathcal{N}_{Z_{\mathbf{G}}(m)} / \, Z_{\mathbf{G}}(m) \]
\end{coroll}
Since $\mathbf{T}$ is a split torus, we have a weight decomposition $\mathfrak{g} = \bigoplus_{\chi \in V} \mathfrak{g}_{\chi}$ of $\mathfrak{g}$ under the adjoint action of $\mathbf{T}$. Here $V$ is a set of characters of $\mathbf{T}$. Let $W$ be the Weyl group of $\mathbf{L}$ with respect to $\mathbf{T}$. There is a natural action of $W$ on $V$. For any subset $S$ of the quotient $V/ \, W$, we can define the set $Q_S$ of elements in $\left(X_{*}(\mathbf{T}) \otimes \mathbf{k}/ \, \mathbb{Q}\right)/ \, W$ of type $S$, just as we did for reductive groups. Let $Z_{\mathbf{G}}(S)$ be the centralizer of any element in $Q_S$. The same reasoning as in the reductive case yields the following concrete parametrization.
\begin{coroll}
There is a natural correspondence
\[ \left\{ \text{regular formal connections}\right\}/\, \mathbf{G}(\overline{F}) \; \; \; \longleftrightarrow{\; \;\; \;} \bigsqcup_{S \subset V / \, W} Q_S \times \mathcal{N}_{Z_{\mathbf{G}}(S)} / \, Z_{\mathbf{G}}(S) \]
\end{coroll}
\subsection{Descent for gauge equivalence classes} \label{subsection: descent for gauge}
All of the theorems above give us a description of connections up to trivializations over $\text{Spec} \,\overline{F}$.  We would like to get a classification over $\text{Spec} \, F$.  This amounts to a problem in Galois cohomology.

We have an action of $\mathbf{G}(\overline{F})$ on $\mathfrak{g}_{\overline{F}}$ that is compatible with the action of the absolute Galois group $\text{Gal}(F)$. Choose a regular connection in canonical form $B = t^{-1} C$ with $C_s \in \mathfrak{D}$ and $[C_s, C] =0$. It is a direct consequence of Proposition \ref{prop: uniqueness regular arbitrary} that the centralizer of $B$ in $\mathbf{G}(\overline{F})$ is $Z_{G}(C) \vcentcolon = Z_{\mathbf{G}}(C)(\mathbf{k})$. Therefore, we get an exact sequence of sheaves of sets over the etale site of $\text{Spec} \, F$
\[ 1 \longrightarrow Z_{G}(C) \longrightarrow \mathbf{G} \longrightarrow \mathbf{G} \cdot B \longrightarrow 1  \]
Here $Z_{G}(C)$ is the constant sheaf associated to the group of $\mathbf{k}$-points of the centralizer $Z_{\mathbf{G}}(C)$ of $C$. This yields a long exact sequence of pointed sets:
\[ 1 \longrightarrow Z_{G}(C) \longrightarrow \mathbf{G}(F) \longrightarrow \mathbf{G} \cdot B (F) \longrightarrow H^{1}_{\text{Gal}(F)} (Z_{G}(C)) \longrightarrow H^{1}_{\text{Gal}(F)} (\mathbf{G}) \]
The theorems of Tsen and Springer mentioned in the preliminaries imply that the right-most Galois cohomology group vanishes. This means that the set of connections over $\text{Spec} \, F$ that admit a trivialization over $\text{Spec} \, \overline{F}$ with canonical form $t^{-1}C$ is in bijection with $H^{1}_{\text{Gal}(F)} (Z_{G}(C))$. Since the action of $\text{Gal}(F)$ on $Z_{G}(C)$ is trivial, $H^{1}_{\text{Gal}(F)} (Z_{G}(C))$ is in (noncanonical) bijection with the set conjugacy classes of elements of finite order in $Z_{G}(C)$. Such bijection comes from the choice of a topological generator of $\text{Gal}(F)\cong \hat{\mathbb{Z}}$. Such a generator corresponds to the compatible choice of a generator $\omega_b$ of $\mu_b$ for all positive integers $b$. Here is a summary of the classification we have obtained.

\begin{prop}[Regular Connections over $D^*$] \label{prop: regular connections galois}
Let $B = t^{-1}\, C$ be a regular canonical connection with $C_s \in \mathfrak{D}$ and $[C_s, C] = 0$. The set of $\mathbf{G}$-connections over $\text{Spec}(F)$ that become gauge equivalent to $B$ over $\text{Spec}(\overline{F})$ is in (noncanonical) bijection with the set of conjugacy classes of elements of finite order in $Z_{\mathbf{G}}(C)(\mathbf{k})$ as described below.
\end{prop}
The correspondence goes as follows. Let $x \in Z_{\mathbf{G}}(C)(\mathbf{k})$ of order $b$. By the vanishing of $H^{1}_{\text{Gal}(F)} (\mathbf{G})$, we can find an element $y \in \mathbf{G}(F_b)$ such that $\omega_b\cdot y = y \, x$. The connection associated to $x$ will be $A = y \cdot \left( t^{-1}\, C \right) \; \in \mathfrak{g}_F$. Conversely, suppose that $A = y \cdot B$ is a connection in $\mathfrak{g}_F$ for some $y \in \mathbf{G}(F_b)$. We set $x \vcentcolon = y^{-1} \, \left(\omega_b \cdot y\right)$.

Using the descriptions of regular $\mathbf{G}(\overline{F})$- gauge equivalence classes we have given previously, we can parametrize regular formal connections. Let $(m, u)$ be a pair with $m \in \left(X_{*}(\mathbf{T}) \otimes \mathbf{k}/ \, \mathbb{Q}\right)$ and $u$ a nilpotent element in $\mathcal{N}_{Z_{\mathbf{G}}(m)}$. A cohomology cocycle as described above is given by an element $t$ of finite order in the centralizer $Z_{\mathbf{G}}(m, u)$ of $u$ in $Z_{\mathbf{G}}(m)$. Since $Z_{\mathbf{G}}(m)$ is connected, we can conjugate by an element of $Z_{\mathbf{G}}(m)$ in order to assume that  the semisimple element $t$ lies in $\mathbf{T} \subset Z_{\mathbf{G}}(m)$. It follows that the set of regular formal $\mathbf{G}$ connections over $D^*$ is in natural correspondence with equivalence classes of triples $(m, x, u)$, where\vspace{-0.25cm}
\begin{enumerate}[(i)]
    \item $m \in \left(X_{*}(\mathbf{T}) \otimes \mathbf{k}/ \, \mathbb{Q}\right)$.
    \item $x$ is an element of finite order in $\mathbf{T}(\mathbf{k})$.
    \item $u \in \mathcal{N}_{Z_{\mathbf{G}}(m)}$ with $\text{Ad}(t)(u) = u$.
\end{enumerate}
\vspace{-0.25cm}
Two such triples are considered equivalent if they can be conjugated by an element of $\mathbf{G}(\mathbf{k})$.

Recall that there is a canonical isomorphism $\mathbf{T}(\mathbf{k}) \cong X_{*}(\mathbf{T}) \otimes \mathbf{k}^{\times}$. Under this identification, the set $\mathbf{T}(\mathbf{k})^{tor}$ of elements of finite order in $\mathbf{T}(\mathbf{k})$ correspond to $X_{*}(\mathbf{T}) \otimes \mu_{\infty}$. The compatible choice of primitive roots of unity $\omega_b$ yields an isomorphism $\mu_{\infty} \cong \mathbb{Q}/ \, \mathbb{Z}$. Hence we get an identification $\mathbf{T}(\mathbf{k})^{tor} \cong X_{*}(\mathbf{T}) \otimes \, \mathbb{Q}/ \, \mathbb{Z}$. This means that the set of pairs $(m, x) \in (X_{*}(\mathbf{T}) \otimes \mathbf{k}/ \, \mathbb{Q}) \times \mathbf{T}(\mathbf{k})^{tor}$ is in natural bijection with $X_{*}(\mathbf{T}) \otimes \, \mathbf{k}/ \, \mathbb{Z}$.

For an element $v \in X_{*}(\mathbf{T}) \otimes \, \mathbf{k}/ \, \mathbb{Z}$, we will let $Z_{\mathbf{G}}(v)$ denote the centralizer $Z_{\mathbf{G}}(m, x)$ of the corresponding pair $(m,x) \in (X_{*}(\mathbf{T}) \otimes \mathbf{k}/ \, \mathbb{Q}) \times \mathbf{T}(\mathbf{k})^{tor}$. Conjugate elements of $X_{*}(\mathbf{T}) \otimes \, \mathbf{k}/ \, \mathbb{Z}$ yield isomorphic centralizers, so it makes sense to define $Z_{\mathbf{G}}(v)$ for $v \in (X_{*}(\mathbf{T}) \otimes \, \mathbf{k}/ \, \mathbb{Z})/ \, W$. We end up the following parametrization of regular formal connections.
\begin{coroll} \label{coroll: regular connections parametrization}
There is a natural bijection between regular formal connections over $D^*$ and pairs $(v, O)$, where \vspace{-0.25cm}
\begin{enumerate}[(i)]
    \item $v \in (X_{*}(\mathbf{T}) \otimes \, \mathbf{k}/ \, \mathbb{Z})/ \, W$.
    \item $O$ is a nilpotent orbit in $\mathcal{N}_{Z_{\mathbf{G}}(v)}/ \, Z_{\mathbf{G}}(v)$.
\end{enumerate}
\end{coroll}
\begin{defn}
Let $A$ be a regular formal connection over $D^*$. We will denote by $(m_A^s, m_A^n)$ the corresponding pair granted by Corollary \ref{coroll: regular connections parametrization}. $m_A^s$ (resp. $m_A^n$) is called the semisimple (resp. unipotent) monodromy of $A$. 
\end{defn}
\begin{example}
Suppose that $\mathbf{k} = \mathbb{C}$. The set of pairs $(m_A^s, m_A^n)$ as above is in correspondence with the set conjugacy classes in $\mathbf{G}(\mathbb{C})$. For a representative $(C, U) \in \text{Lie}(\mathbf{T}) \times \mathcal{N}_{\mathbf{G}}$ of the pair $(m_A^s, m_A^n)$, the corresponding element of $\mathbf{G}(\mathbb{C})$ is given by $\text{exp}(2 \pi i C+U)$. This just the monodromy class of the regular formal connection.
\end{example}
We can use the theorem in \cite{humphreys-conjugacy} pg. 26 to give a description of $Z_{\mathbf{G}}(v)$. We can decompose $\mathfrak{g} = \text{Lie}(\mathbf{T}) \oplus \bigoplus_{i = 1}^l \mathfrak{u}_{\chi_i}$, where each $\mathfrak{u}_{i}$ is a one dimensional eigenspace of $\mathbf{T}$ consisting of nilpotent elements (we allow repetition in the $\chi_i$). Suppose that $\mathbf{T}$ acts on $\mathfrak{u}_i$ through the character $\chi_i: \mathbf{T} \longrightarrow \mathbb{G}_m$. Since we are working in characteristic $0$, each $\mathfrak{u}_{i}$ is the Lie algebra of a unique unipotent subgroup $\mathbf{U}_{i}$ isomorphic to $\mathbb{G}_a$. Let $v \in (X_{*}(\mathbf{T}) \otimes \, \mathbf{k}/ \, \mathbb{Z})$. For each character $\chi \in X^{*}(\mathbf{T})$ it makes sense to ask whether $\langle v, \chi \rangle \in \mathbb{Z}$, even though $v$ is only defined up to an element of $X_{*}(\mathbf{T})$. The connected component of $Z_{\mathbf{G}}(v)$ is generated by $\mathbf{T}$ and those unipotent weight spaces $\mathbf{U}_{i}$ such that $\langle v, \chi_i \rangle \in \mathbb{Z}$. The full group $Z_{\mathbf{G}}(v)$ is generated by its neutral component and the reflections $w_{\alpha}$ in the Weyl group $W$ of $\mathbf{L}$ corresponding to roots $\alpha \in \Phi$ such that $\langle v, \, \alpha \rangle \in \mathbb{Z}$.

In order to classify formal regular connections, it is convenient to group them depending on the type of their semisimple monodromy. For each $v \in (X_{*}(\mathbf{T}) \otimes \, \mathbf{k}/ \, \mathbb{Z})/ \, W$, we can define the type of $v$ to be a subset $S \subset V/ \, W$ just as we did when working over $\overline{F}$. For any $S \subset V/ \, W$, let us denote by $P_S \subset (X_{*}(\mathbf{T}) \otimes \, \mathbf{k}/ \, \mathbb{Z})/ \, W$ the set of all element of type $S$. By the description given in the last paragraph, it follows that the isomorphism class of $Z_{\mathbf{G}}(v)$ is the same for all $v \in P_S$. We will denote this group by $Z_{\mathbf{G}}(S)_F$. We can now rewrite the last corollary.
\begin{coroll}
There is a natural correspondence
\[ \left\{ \text{regular formal connections over $D^*$}\right\} \; \; \; \longleftrightarrow{\; \;\; \;} \bigsqcup_{S \subset V / \, W} P_S \times \mathcal{N}_{Z_{\mathbf{G}}(S)_F} / \, Z_{\mathbf{G}}(S)_F \]
\end{coroll}

This yields procedure to describe all regular $\mathbf{G}(F)$-gauge equivalence classes of formal connections. The different types $S$ can be established by declaring a subsets of characters that evaluate to integer eigenvalues. For each $S$, one needs to parametrize the nilpotent orbits in the Lie algebra of  $\mathbf{M}_S$. This works especially well when the group is reductive.

\begin{example} \label{example: classical groups regular connections}
Suppose that $\mathbf{G}$ is reductive. Then each $\mathbf{M}_S$ is connected reductive. If the group $\mathbf{G}$ is classical, we can parametrize the finite set of nilpotent orbits $\mathcal{N}_{\mathbf{M}_S}/ \, \mathbf{M}_S$ using partition diagrams as in  \cite{collingwood.nilpotent} Chapter 5. This yields explicit canonical block decompositions of regular connections for classical groups.
\end{example}
\end{subsection}
\section{Irregular connections for $\mathbf{G}$ reductive} \label{section: irregular reductive}
\subsection{Connections in canonical form}
Let $\mathbf{G}$ be connected reductive over $\mathbf{k}$.
\begin{defn}[Canonical form] \label{defn: canform} A connection $B \in \mathfrak{g}_{\overline{F}}$ is said to be in canonical form if we have $B = \sum_{j=1}^l D_{j} \, t^{r_j} + t^{-1} \, C$, where\vspace{-0.25cm}
\begin{enumerate}[(1)]
    \item $r_j \in \mathbb{Q}$ for all $j$, and they satisfy satisfy  $r_1< r_2< .. r_l <-1$.
    \item $D_j \neq 0$ is a semisimple element in $\mathfrak{g}$ for all $j$.
    \item $D_1, D_2, ... , C$ are pairwise commuting (the brackets vanish).
\end{enumerate}
\vspace{-0.25cm}
\end{defn}
The $r_j$ above are called the levels of the canonical form. The smallest of them $r_1$ is called the principal level. The initial sum $\sum_{j=1}^l D_{j} \, t^{r_j}$ is called the irregular part of the connection; we denote it by $B_{\text{irr}}$. 
\begin{remark}
The irregular part could be $0$ if the summation is empty. We then recover the notion of canonical form for a regular connection.
\end{remark}
We prove the existence of canonical form for reductive groups first.
\begin{thm}[Reduction Theory for Reductive Groups] \label{thm: reduction conn reductive} Let $\mathbf{G}$ be connected reductive and $A\in \mathfrak{g}_{\overline{F}}$. Then there exists $x \in \mathbf{G}(\overline{F})$ such that $x \cdot A = \sum_{j=1}^l D_{j} \, t^{r_j} + t^{-1} \, C$ is in canonical form.
\end{thm}
The argument proceeds by induction on $\text{dim} \, \mathbf{G}$. The base case (when $\mathbf{G} = \mathbb{G}_m$) follows from the computation done in the proof of Proposition \ref{prop:tori regular}. We state this result for future reference.

\begin{prop} \label{prop:tori irregular} Let $\mathbf{G}$ be a tori.\vspace{-0.25cm}
\begin{enumerate}[(a)]
    \item Let $A = \sum_{j=r}^{\infty} A_j \, t^j$ a formal connection in $\mathfrak{g}_F$. Then there exists $x \in \mathbf{G}(\mathcal{O})$ such that $x \cdot A = \sum_{j=r}^{-1} A_j \, t^{j}$. Moreover there is a unique such $x$ with $x \equiv 1 \, \left( mod \; t \right)$.
    \item Let $B = \sum_{j=r}^{-1} B_j \, t^j$ and $C= \sum_{j=r}^{-1} C_j \, t^j$ be two connections in canonical form. Suppose that there exists $x \in \mathbf{G}(F)$ such that $x 
\cdot C = B$. Then  $x = g \, t^{\mu}$ for some cocharacter $\mu \in X_{*}(\mathbf{G})$ and some $g \in \mathbf{G}(\mathbf{k})$. In this case, we have $B_j = C_j$ for all $r \leq j < -1$ and $C_{-1} = B_{-1} - \mu$.
\end{enumerate}
\vspace{-0.25cm}
\end{prop}
\begin{proof}
This is the same computation as in Propositions \ref{prop:tori regular} and \ref{prop: uniqueness regular tori}. We omit the details.
\end{proof}

\begin{remark}
In particular we see that two canonical connections $B = \sum_{j=r}^{-1} B_j \, t^j$ and $C= \sum_{j=r}^{-1} C_j \, t^j$ for a torus are gauge equivalent over $F$ if and only if $B_j = C_j$ for all $r \leq j < -1$ and $C_{-1} -B_{-1} \in X_{*}(\mathbf{G})$. By lifting, we conclude that they are equivalent over $\overline{F}$ if and only if $B_j = C_j$ for all $r \leq j < -1$ and $C_{-1} -B_{-1} \in X_{*}(\mathbf{G})\otimes \mathbb{Q}$.
\end{remark}

Let's get started with the argument for Theorem \ref{thm: reduction conn reductive}. By the structure theory of reductive groups, we know that $\mathbf{G}$ admits an isogeny from the product of its maximal central torus and its derived subgroup $\mathbf{G}_{der}$. By Proposition \ref{prop:tori irregular}, we can deal with the central part. We may therefore assume that $\mathbf{G}$ is semisimple.

By lifting to a ramified cover, we can assume $A = \sum_{j =r}^{\infty} A_j \, t^j \in \mathfrak{g}_F$ with $A_r \neq 0$. If $r \geq -1$ we can use the theory for regular connections developed in Section 2. So we can assume $r < -1$. There are two substantially different possibilities: $A_r$ could be nilpotent or not. The case when $A_r$ is not nilpotent turns out to be the easiest; we do it first.

\subsection{The case when $A_r$ is not nilpotent} \label{subsection: not nilpotent case}
We need the following lemma (cf. \cite[9.3]{Babbitt.varadarajan.formal}).

\begin{lemma} \label{lemma: not nilpotent case}
Let $A= \sum_{j = r} A_j \, t^j$ a connection in $\mathfrak{g}_F$ with $r < -1$. Let $V = (A_r)_s$ be the semisimple part of $A_r$. Then, there exist $x \in \mathbf{G}(F)$ such that $x \cdot A$ is in $\mathfrak{g}_{V}(F)$.
\end{lemma}
\begin{proof}This is very similar to Lemma \ref{lemma:aligned}. We will inductively build a sequence $(B_j)_{j=1}^{\infty}$ of elements of $\mathfrak{g}$ such that the gauge transformation $x \vcentcolon = \lim_{n \rightarrow \infty} \prod_{j =0}^{n-1} \text{exp}(t^{n-j} \, B_{n-j})$ satisfies the conclusion of the lemma. Suppose that we have chosen $B_j$ for $j \leq k$ such that the connection $A^{(k)} = \sum_{l=r}^{\infty} A^{(k)}_{l } \,t^l$ defined by $A^{(k)} \vcentcolon = \prod_{j =0}^{k-1} \text{exp}(t^{n-j} \, B_{k-n}) \cdot A$ satisfies $A_{l}^{(k)} \in \mathfrak{g}_{V}$ for all $l \leq k+r$. Notice that the base case $k=0$ is trivial and $A^{(k)}_{r} = A_{r}$. Let's try to determine $B_{k+1}$.

Recall that $\text{exp}(t^{k+1} \, B_{k+1}) \equiv 1 + t^{k+1} \, B_{k+1} \; (\text{mod} \; t^{k+2})$. By an elementary matrix computation (choose an embedding of $\mathbf{G} \hookrightarrow \mathbf{\text{GL}_n}$), one can see that
\[\text{exp}(t^{k+1} B_{k+1}) \cdot A^{(k)} \equiv \sum_{l=r}^{k+r} A^{(k)}_l \, t^l + [A^{(k)}_{k+1+r} - ad(A_{r})B_{k+1}] \, t^{k+1+r} \; \; (\text{mod} \; t^{k+2 +r})   \]
Let $\mathfrak{g} = \bigoplus_{\lambda} \mathfrak{g}_{\lambda}$ be the spectral decomposition of $ad(X) = (ad(A_r))_s$. By definition the operator $ad(A_{-r})$ restricts to an automorphism of $\mathfrak{g}_{\lambda}$ for all $\lambda \neq 0$. In particular, we can choose $B_{k+1} \in \mathfrak{g}$ such that $A^{(k)}_{k+1-r} - ad(A_{-r})B_{k+1}$ is in $\mathfrak{g}_{0} = \mathfrak{g}_{V}$. This concludes the construction of the sequence $(B_j)_{j=1}^{\infty}$. It follows from the construction that $x \vcentcolon = \lim_{n \rightarrow \infty} \prod_{j =0}^{n-1} \text{exp}(t^{n-j} \, B_{n-j})$ satisfies $x \cdot A \; \in \, \mathfrak{g}_{V}(F)$.
\end{proof}

Let us continue with the proof of Theorem \ref{thm: reduction conn reductive}. Suppose that $A_r$ is not nilpotent. Then the semisimple part $X = (A_r)_s$ is not $0$. Since we are assuming that $\mathbf{G}$ is semisimple, the connected reductive centralizer $\mathbf{Z}_{G}(X)$ is a proper subgroup of $\mathbf{G}$. By Lemma \ref{lemma: not nilpotent case} we can assume that $A \in \mathfrak{g}_{V}(F)$. We win by induction, because $\text{dim}\, \mathbf{Z}_{G}(V) \, < \, \text{dim} \, \mathbf{G}$. This settles the case when $A_r$ is not nilpotent in the proof of Theorem \ref{thm: reduction conn reductive}.

Recall that the principal level of a canonical connection is defined to be the order $r_1$ (see the paragraph after Definition \ref{defn: canform}). We can define the principal level of a connection $A$ to be the principal level of any canonical connection equivalent to $A$. This is well defined by Lemma $\ref{lemma: for uniqueness of canonical}$ in the next section. The inductive argument given above implies the following interesting fact.

\begin{prop} \label{prop: semisimple.prinlevel}
Suppose that $A = \sum_{j = r}^{\infty} A_j t^j$ with $r < -1$ and $A_r$ not nilpotent. Then $r$ is the principal level of A.
\end{prop}
\begin{proof}
We induct on the dimension of the group. The base case $\mathbb{G}_m$ is clear by direct computation. Notice that in the proof above we have that $A_r$ is still not nilpotent in the smaller group $\mathbf{Z}_{G}(V)$, since its semisimple part $V$ is not $0$. We can then conclude by induction.
\end{proof}
\begin{remark}
 As we will soon see, this is not necessarily the case when $A_r$ is nilpotent. In the nilpotent case the principal level can be larger than $r$.
\end{remark}
\subsection{The case when $A_r$ is nilpotent}
In this section we conclude the proof of Theorem \ref{thm: reduction conn reductive} by dealing with the case when $A_r$ is nilpotent. We closely follow an argument sketched in \cite{Babbitt.varadarajan.formal}, which we include here for the sake of completeness. For this section $\mathbf{G}$ will be semisimple, as we may assume for our proof of Theorem \ref{thm: reduction conn reductive}. Let's set up the notation we will use. Suppose that we have $A= \sum_{j=r}^{\infty} A_j \, t^j$ with $A_r$ nilpotent and $r < -1$. Let $(H, X ,  \, Y=A_r)$ be a Jacobson-Morozov $\mathfrak{sl}_2$-triple coming from an algebraic homomorphism $\Phi: \text{SL}_2 \longrightarrow \mathbf{G} $. For an integer $n$, we will denote by $t^{nH}$ the element $t^{\mu}$, where $\mu$ is the natural composition $\mathbb{G}_m \xrightarrow{[n]} \mathbb{G}_m \hookrightarrow \text{SL}_2 \xrightarrow{\Phi} \mathbf{G}$.
\begin{lemma} \label{lemma: nilpotent case aligned}
With notation as above, there exists $x \in \mathbf{G}(F)$ such that $x \cdot A = \sum_{j=r}^{\infty} B_j \, t^j$ satisfies\vspace{-0.25cm}
\begin{enumerate}[(1)]
    \item $B_r= A_r$.
    \item $B_j \in \mathfrak{g}_X$ for all $j > r$.
\end{enumerate}
\vspace{-0.25cm}
\end{lemma}
\begin{proof}
This is a carbon copy of the proof of Lemma \ref{lemma: not nilpotent case}. The only difference is that in the last paragraph we have to use the fact that the range of $ad(A_r)$ is complementary to $\mathfrak{g}_X$. This follows from the theory of representations of $\mathfrak{sl}_2$. We omit the details.
\end{proof}
By Lemma \ref{lemma: nilpotent case aligned}, we can assume that $A_j \in \mathfrak{g}_X$ for $j >r$. For the purpose of having an algorithm that works in finitely many steps, we won't actually use the full force of the lemma. Instead, we will use a weaker hypothesis as an input for the next proposition. Let $\Lambda \vcentcolon = \Lambda\left(A_r\right)$ be as in Definition \ref{defn: lambda}. We will henceforth suppose that  $A_{r+m} \in \mathfrak{g}_X$ for $1 \leq m < \Lambda (|r|-1)$.

Let $(Z_{l})_{l=1}^{q}$ be a basis of eigenvectors of $ad(H)$ acting on  $\mathfrak{g}_X$. This means that $\mathfrak{g}_X = \bigoplus_{l=1}^{q} \mathbf{k}\, Z_{l}$ and that there exist $\lambda_l$ such that $[H, Z_l] = \lambda_l \, Z_l$. It turns out that  the $\lambda_{l}$s are nonnegative integers, by the theory of representations of $\mathfrak{sl}_2$. By the assumption on $A$, we can write $A_{r+m} = \sum_{l=1}^{q} a_{r+m, \, l}\, Z_l$  for all $1 \leq m \leq \Lambda (|r|-1)$ and some constants $a_{r+m, \, l} \in \mathbf{k}$.
\begin{defn} \label{defn: delta} In the situation above, define $\delta = \delta(A)$ to be given by:
\[ \delta = \text{inf} \; \left\{ \frac{m}{\frac{1}{2} \lambda_l +1} \; \vcentcolon \; \; 1 \leq m < \Lambda (|r|-1), \; \; 1 \leq l \leq q, \; \; a_{r+m, \, l} \neq 0 \right\} \]
We set $\delta = \infty$ if $A_{r+m} = 0$ for all $1 \leq m < \Lambda$. We also define the set
\[ P \vcentcolon = \left\{ (m, l) \; \vcentcolon \; \;1 \leq m < \Lambda (|r|-1), \; \; 1 \leq l \leq q, \; \; a_{r+m, \, l} \neq 0, \; \; \frac{m}{\frac{1}{2} \lambda_l +1} = \delta \right \} \]
In plain words, $P$ is the set of pairs $(m,l)$ of indices in the definition of $\delta$ where the infimum is actually achieved.
\end{defn}
\begin{remark} \label{remark: ramification bound key lemma} By the definition of $\Lambda(A_r)$, it follows that the denominators appearing in the set defining $\delta$ are always less than $\Lambda$. This implies that there exists a positive integer $b \leq 2\Lambda -1$ such that 
$b\delta \in \mathbb{Z}$. This fact will be used later to determine a bound for the ramification needed to put $A$ in canonical form.
\end{remark}
The following proposition is one of the main steps in the argument in $\cite{Babbitt.varadarajan.formal}$. We are going to get a step closer to canonical form by applying a transformation of the type $t^{nH}$. These elements are called shearing transformations. The statement and proof in the case of $\mathbf{GL_n}$ can be found in \cite{Babbitt.varadarajan.formal} pages 33-34. We have decided to include a detailed treatment of the general case for the convenience of the reader.
\begin{prop}[Main Proposition for the Induction Step] \label{prop: main prop reduction} 
Let the notation/set-up be as discussed above. \vspace{-0.25cm}
\begin{enumerate}[(C1)]
    \item Suppose $|r|-1 \leq \delta \leq \infty$. Let $\tilde{A}$ be the $2$-lift of $A$. Then $B \vcentcolon = t^{(r+1)H} \cdot \tilde{A}$ is of the first kind, and $B_{-1}$ only depends on $A_{r +m}$ for $0 \leq m \leq \Lambda (|r|-1)$.
    \item Suppose $0 < \delta < |r|-1$. We know that $b \delta \in \mathbb{Z}$ for some $b \in \mathbb{N}$. Let $\tilde{A}$ be the $2b$-lift of $A$. We have that $B \vcentcolon = t^{- b \delta H} \cdot \tilde{A}$ has order $r' \vcentcolon = 2br + 2b\delta +2b -1 < -1$. Moreover,
    \[B_{r'} = 2bA_{r} + 2b \sum_{(m,l) \in P} a_{r+m, \, l} \, Z_{l} \; \neq \; 2bA_r \]
    In particular we have that $B_{r'}$ is determined by $A_{r +m}$ for $0 \leq m < \Lambda (|r|-1)$. If $B_{r'}$ is nilpotent, then $\text{dim} \, (G \cdot B_{r'}) \, > \, \text{dim} \,(G \cdot A_r)$.
\end{enumerate}
\vspace{-0.25cm}
\end{prop}
\begin{proof}
The computation is similar to the one we did in the proof of Theorem \ref{thm:semisimple regular}. Recall from the discussion in that proof that for all $W \in \mathfrak{g}_{\beta}$ we have 
\[ \text{Ad}(t^{nH}) \, W = t^{n \beta(H)}W \; \; \; (*)\]
\begin{enumerate}[(C1)]
    \item By using the definitions and expanding
    \begin{align*}
    t^{(r+1)H} \cdot \tilde{A} & = \; \; 2 \sum_{m=0}^{\Lambda( |r| -1) -1} \text{Ad}(t^{(r+1)H}) A_{r+m} \, t^{2(r+m)+1} \\[2ex]
    & \quad \; + \,\text{Ad}(t^{(r+1)H}) A_{r+\Lambda( |r| -1)} \, t^{2(r+\Lambda( |r| -1))+1} \\[2ex]
        & \quad + \, 2 \sum_{m=\Lambda( |r| -1)+1}^{\infty} \text{Ad}(t^{(r+1)H}) A_{r+m} \, t^{2(r+m)+1} \\[2ex]
        & \quad \; + \, \frac{d}{dt} \, (t^{(r+1)H}) \, t^{-(r+1)H}
    \end{align*}
    The fourth summand is just $(r+1) Ht^{-1} $, which is of the first kind. We can see that the third summand is actually in $\mathfrak{g}(\mathcal{O})$ by using $(*)$ and the fact that $(r+1)\beta(H) \geq (2\Lambda -2)(r+1)$ for all roots $\beta$. The same reasoning implies that the second summand is of the first kind. For the first summand, we can write $A_{r+m} = \sum_{l=1}^{q} a_{r+m, \, l}\, Z_l$. We can expand and use $(*)$ plus the definition of $\lambda_l$. We get that the first summand is:
    \[ 2 \sum_{m=0}^{\Lambda( |r| -1) -1} \sum_{l=1}^q a_{r+m, \, l}\, Z_{l} \, t^{2(r+m) +1 + (r+1)\lambda_{l}}  \]
    This expression is also of the first kind. This can be shown by doing some algebra with the exponents of $t$, keeping in mind the definition of $\delta$ and the fact that $\delta \geq |r|-1$. The remark about $B_{-1}$ follows plainly from the argument, because the third summand did not contribute to $B_{-1}$.
    
    \item This is very similar to the first case. We expand:
    \begin{align*}
    t^{-b \delta H} \cdot \tilde{A} & = \; \; 2b \sum_{m=0}^{\Lambda( |r| -1) -1} \text{Ad}(t^{-b \delta H}) A_{r+m} \, t^{2b(r+m)+2b-1} \\[2ex]
        & \quad + \, 2b \sum_{m=\Lambda( |r| -1)}^{\infty} \text{Ad}(t^{-b \delta H}) A_{r+m} \, t^{2b(r+m)+2b-1} \\[2ex]
        & \quad \; + \, \frac{d}{dt} \, (t^{(-b \delta)H}) \, t^{(b\delta)H}
    \end{align*}
    The third summand is $-b \delta H t^{-1}$, which is of the first kind. We can therefore ignore the third summand. The bound $-b \delta \beta(H) \geq -2b\delta(\Lambda -1)$ and equation $(*)$ show that the order of the second summand is at least $r' +1 = 2br+2b\delta +2b$. The computation is almost the same as for the third summand in Case 1 above. For the first summand, we can again use $A_{r+m} = \sum_{l=1}^{q} a_{r+m, \, l}\, Z_l$ and expand using $(*)$ to get:
    \[ 2b \sum_{m=0}^{\Lambda( |r| -1) -1} \sum_{l=1}^q a_{r+m, \, l}\, Z_{l} \, t^{2b(r+m) +2b -1 -b\delta \lambda_{l}}  \]
    We are reduced to check that the exponent of $t$ in the sum above has minimal value $r' = 2br+2b\delta +2b -1$ exactly for the pairs $(m, l)$ in $P$. This is an exercise in elementary algebra.
    
    The claim about $B_{r'}$ follows from the argument, because the second summand does not contribute to $B_{r'}$. The claim about the increase of the dimension of nipotent orbits is a direct consequence of Proposition \ref{prop: dim increase}.
\end{enumerate}
\vspace{-0.5cm}
\end{proof}
\begin{remark}
The claim about the dimension of the orbit in $C2$ is essential. This guarantees that the process of applying shearing transformations eventually stops. Hence we are provided with a terminating algorithm. See the proof of Theorem \ref{thm: reduction conn reductive} given below for details.
\end{remark}

\begin{example}
Let's see how this works in the case of $\text{SL}_2$. Up to inner automorphism, we can assume that $A_{r} = Y = \begin{bmatrix} 0 & 0 \\ 1 & 0 \end{bmatrix}$. Then $X= \begin{bmatrix} 0 & 1 \\ 0 & 0 \end{bmatrix}$ and $H= \begin{bmatrix} 1 & 0 \\ 0 & -1 \end{bmatrix}$. In this case $\Lambda =2$, and $\mathfrak{g}_X = \mathbf{k} X$. So there is a single eigenvalue $\lambda = 2$. Our assumption just says that $A$ is of the form
\[ A = Y \,t^r + \sum_{m = 1}^{2|r|-3} a_{r+m}X \, t^{r+m} + \text{higher order terms}\]
We have that $\delta$ is $\frac{n}{2}$, where $n$ is the smallest index such that $a_{r+n} \neq 0$. The set $P$ only contains this index $n$. So in fact $A$ can be written in the form
\[ A = Y\, t^r + \sum_{m = n}^{2|r|-3} a_{r+m}X \, t^{r+m} + \text{higher order terms}\]
There are two cases.\vspace{-0.25cm}
\begin{enumerate}[(C1)]
    \item The first case is $n\geq 2(|r|-1)$. This just means that all $a_{i}$ above are $0$. Then we can use the change of trivialization $t^{\frac{r+1}{2}H} = \begin{bmatrix} t^{\frac{r+1}{2}} & 0 \\ 0 & t^{-\frac{r+1}{2}} \end{bmatrix}$ to transform $A$ into a connection of the first kind.
    \item The second case is when $n < 2(|r|-1)$. Here at least some of the $a_i$ are not $0$. We can apply the transformation $t^{-\frac{n}{4}H} = \begin{bmatrix} t^{-\frac{n}{4}} & 0 \\ 0 & t^{\frac{n}{4}} \end{bmatrix}$. The resulting connection will have order $r + \frac{n}{2}$. The principal term will be $B_{r + \frac{n}{2}} = Y + a_{r+n}X$, which is semisimple. Hence we can use Lemma \ref{lemma: not nilpotent case} to reduce to the group $\mathbb{G}_m$. We can then apply Proposition \ref{prop:tori irregular} to find the canonical form.
\end{enumerate}
\end{example}
\begin{proof}[Proof of Theorem \ref{thm: reduction conn reductive}]
By Lemma \ref{lemma: nilpotent case aligned}, we can put ourselves in the situation of Proposition \ref{prop: main prop reduction} above. We have three possibilities:
\begin{enumerate}[(i)]
    \item If $|r|-1 \leq \delta \leq \infty$, then we can use Proposition \ref{prop: main prop reduction} Case 1. We are done by the theory of regular connections we have already developed.
    \item If $0 < \delta < |r|-1$, we can use \ref{prop: main prop reduction} Case 2. Suppose that $B_{r'}$ is not nilpotent. Then we are in the case worked out in Subsection \ref{subsection: not nilpotent case}.
    \item Suppose that $0 < \delta < |r|-1$ and $B_{r'}$ is nilpotent with $\text{dim} \, (G \cdot B_{r'}) \, > \, \text{dim} \,(G \cdot A_r)$. We can apply Proposition \ref{prop: main prop reduction} Case 2 again with $B$ instead of $A$. We can keep iterating this procedure until we are in one of the first two possibilities above. Notice that this process cannot go on indefinitely, because the dimensions of nilpotent orbits in $\mathbf{G}$ are bounded.
\end{enumerate}
\vspace{-0.5cm}
\end{proof}
\begin{remark}
The dimension of adjoint nilpotent orbits in $\mathbf{G}$ is always even \cite{collingwood.nilpotent}. Therefore we need to apply at most $\left \lfloor{ \frac{1}{2} \text{dim}(\mathbf{G})} \right \rfloor$ shearing transformations as in Proposition \ref{prop: main prop reduction} Case 2 before we land in one of the first two possibilities.
\end{remark}
\subsection{Algorithm for reductive groups and some quantitative results} \label{subsection: quantitative}
Let's give a detailed description of the reduction algorithm that we obtain from the proof of Theorem \ref{thm: reduction conn reductive}. 
\begin{algor}[Algorithm for reduction of a formal connection for a reductive group] \label{algor: reduction reductive}
There is a set of six possible operations that we will use as steps in our algorithm.\vspace{-0.25cm}
\begin{enumerate}[(i)]
    \item Apply Lemma \ref{lemma: not nilpotent case}.
    \item Apply Lemma \ref{lemma: nilpotent case aligned}.
    \item Apply Proposition \ref{prop: main prop reduction} Case 1.
    \item Apply Proposition \ref{prop: main prop reduction} Case 2.
    \item Find the canonical form of a connection in a torus.
    \item Find the canonical form for a connection of the first kind in a semisimple group (as in Theorem \ref{thm:semisimple regular} of Section \ref{section: regular connections}).
\end{enumerate}
The algorithm proceeds as follows. The input is a given reductive group $\mathbf{G}$ and a formal connection $A = \sum_{j=r}^{\infty} A_j \, t^j$. First, we know that $\mathbf{G}$ is isogenous to the product $Z^0\left(\mathbf{G} \right) \times \mathbf{G}_{\text{der}}$ of its maximal central torus and its derived subgroup. Apply operation (v) to the central part of the connection $A_{Z^0(\mathbf{G})}$. We can record the (canonical form) output of (v) and ignore it from now on, since it is not going to be altered by the subsequent steps in the algorithm. Replace $\mathbf{G}$ by $\mathbf{G}_{\text{der}}$ and $A$ by $A_{\text{der}}$. We have two cases. \vspace{-0.25cm}
\begin{enumerate}[(1)]
    \item If $A_{der}$ is of the first kind, apply step (vi) to reduce this connection to canonical form. Add any ``central" parts we might have split off earlier in the algorithm and output the result. End of the algorithm.
    \item If $A_{der}$ is not of the first kind, check whether $A_r$ is nilpotent or not. There are now two ways to proceed:
    \begin{enumerate}[(2-1)]
    \item If $A_r$ is not nilpotent, use operation (i). Replace $\mathbf{G}$ by $Z_{\mathbf{G}} \left( \, \left( A_r \right)_s \, \right)$ and replace $A$ by the output of operation (i). Return to the beginning of the algorithm with this new input.
    \item If $A_r$ is nilpotent, compute $\Lambda\left(A_r\right)$. Apply operation $(ii)$ and replace $A$ with the output. Now compute $\delta$.
    
    \begin{enumerate}[(2-2a)]
        \item If $|r|-1 \leq \delta$, apply operation (iii) and replace $A$ with the output. This is a connection of the first kind. Return to the beginning of the algorithm.
        
        \item If $\delta < |r|-1$, apply operation (iv). Go to the beginning of the algorithm.
    \end{enumerate}
\end{enumerate}

\end{enumerate}
\end{algor}
\begin{remark}
In the algorithm above the order of the pole of $A$ only ever gets smaller (taking into account $b$-lifting whenever passing to a ramified cover). This shows that the principal level of $A$ determines the mildest pole in the $\mathbf{G}(\overline{F})$-gauge equivalence class of $A$.
\end{remark}
A careful analysis of Algorithm \ref{algor: reduction reductive} yields a bound for the ramification needed to put a given connection into canonical form. We can get a uniform bound that only depends on the group, and not on the connection. Before giving the proof of such bound, we need a lemma.

Note that in each step of Algorithm \ref{algor: reduction reductive} we are ultimately working with a semisimple subgroup of $\mathbf{G}$. These subgroups are of the form $\mathbf{H}_{der}$, where $\mathbf{H}$ is the centralizer of a finite set $\{ D_1, D_2, ..., D_l\}$ of pairwise commuting semisimple elements in $\mathfrak{g}$. We will first try to understand $J(\mathbf{H}_{der})$ (see Definition \ref{defn: exponent coweights cooroots}).
\begin{lemma} \label{lemma: index of centralizers}
Let $\mathbf{G}$ be a connected semisimple group. Let $R$ be the rank of $\mathbf{G}$. Suppose that $\{D_1, D_2, ..., D_l\}$ is a finite set of pairwise commuting semisimple elements in $\mathfrak{g}$. Set $\mathbf{H} = Z_{\mathbf{G}}\left(\{D_1, D_2, ..., D_l\}\right)$, the centralizer of all $D_i$. Let $\mathbf{H}_{der}$ be the derived subgroup of $\mathbf{H}$. Then we have $J(\mathbf{H}_{der}) \leq \text{hgt}(\mathfrak{g})^{2R-2} \cdot J(\mathbf{G})$.
\end{lemma}
\begin{proof}
The lemma is clearly true if $\mathbf{H} = \mathbf{G}$. We can therefore assume that $\mathbf{H} \neq \mathbf{G}$. Let $\mathbf{T}$ be a maximal torus of $\mathbf{G}$ such that $\text{Lie}(\mathbf{T})$ contains the set $\{D_i\}$ of pairwise commuting semisimple elements. Note that $\mathbf{T} \subset \mathbf{H}$ is a maximal torus. Let $\Phi$ (resp. $\Sigma$) be the set of roots of $\mathbf{G}$ (resp. $\mathbf{H}$) with respect to $\mathbf{T}$. We will denote by $\Sigma^{\vee}$ and $\Phi^{\vee}$ the corresponding sets of coroots. It follows by definition that $\Sigma \subset \Phi$ and $\Sigma^{\vee} \subset \Phi^{\vee}$. 

Write $Q_{\mathbf{H}_{der}}$ for the coweight lattice of $\mathbf{H}_{der}$. Let $\lambda \in Q_{\mathbf{H}_{der}}$. We want to show that there exists $b \leq \text{hgt}(\mathfrak{g})^{2R-1} \cdot J(\mathbf{G}_{\text{der}})$ such that $b \lambda \in \mathbb{Z}\Sigma^{\vee}$.

Fix a choice of positive roots $\Phi^+$ in $\Phi$. Let $\Delta_{\Phi}$ be the corresponding set of simple roots. By definition $|\Delta_{\Phi}| = R$. Notice that this induces a set of positive roots $\Sigma^+ \vcentcolon = \Phi^+ \cap \Sigma$. Let $\Delta_{\Sigma}$ be the corresponding set of simple roots in $\Sigma$. Set $c \vcentcolon = |\Delta_{\Sigma}|$. We know that $c \leq R-1$ because $\mathbf{H} \neq \mathbf{G}$. Consider the short exact sequence
\[0 \longrightarrow \mathbb{Z} \Sigma \xrightarrow{\; \; M \; \;} \mathbb{Z} \Phi \longrightarrow \mathbb{Z} \Phi / \, \mathbb{Z} \Sigma \longrightarrow 0  \]
The theory of Smith normal form implies that $\mathbb{Z} \Phi / \, \mathbb{Z} \Sigma \cong E \oplus \mathbb{Z}^{R -c}$, where $E$ is a finite group. The exponent of $E$ is given by the biggest elementary divisor $d$ of the inclusion $M$ of free $\mathbb{Z}$-modules. Applying the functor $\text{Hom}(-, \mathbb{Z})$ to the short exact sequence yields an exact sequence
\[ 0 \longrightarrow \mathbb{Z}^{R-c} \longrightarrow Q_{\mathbf{G}} \longrightarrow Q_{\mathbf{H}_{der}} \longrightarrow E \longrightarrow 0 \]
Hence we have that $d\lambda$ can be extended to an element of $Q_{\mathbf{G}}$. By the definition of $J(\mathbf{G})$, it follows that $d \, J(\mathbf{G}) \, \lambda$ extends to an element of $\mathbb{Z} \Phi^{\vee}$.

Let $\varphi: \mathbb{Z}\Phi^{\vee} \longrightarrow Q_{\mathbf{H}_{der}}$ be the composition
\[\varphi: \mathbb{Z} \Phi^{\vee} \, \hookrightarrow \, Q_{\mathbf{G}} \, \longrightarrow Q_{\mathbf{H}_{der}}  \]
Set $L \vcentcolon =\text{Im} \, \varphi$ and $K \vcentcolon = \text{Ker} \, \varphi$. The discussion above implies that the exponent of the finite group $Q_{\mathbf{H}_{der}} / \, L$ is bounded by $d \, J(\mathbf{G})$. By definition we have a short exact sequence
\[ 0 \, \longrightarrow \, K \, \longrightarrow \, \mathbb{Z} \Phi^{\vee} \, \longrightarrow \, L \, \longrightarrow \, 0 \]
Since $L$ is a torsion-free $\mathbb{Z}$-module, the above exact sequence splits. Fix a splitting $\mathbb{Z} \Phi^{\vee} \cong K \oplus L$. Let's look now at the inclusion of lattices $\mathbb{Z} \Sigma^{\vee} \subset \mathbb{Z} \Phi^{\vee}$. The composition
\[ \mathbb{Z} \Sigma^{\vee} \, \hookrightarrow \, \mathbb{Z} \Phi^{\vee} \, \xlongequal{\;} \, K \oplus L \, \xrightarrow{\; \; pr_2 \; \;} \, L \, \xhookrightarrow{ \; \;} \, Q_{\mathbf{H}_{der}} \]
is the natural inclusion $\mathbb{Z} \Sigma^{\vee} \hookrightarrow Q_{\mathbf{H}_{der}}$. Hence the morphism $\psi$ given by the composition
\[\psi \,: \; \; \mathbb{Z} \Sigma^{\vee} \, \hookrightarrow \, \mathbb{Z} \Phi^{\vee} \, \xlongequal{\;} \, K \oplus L \, \xrightarrow{\; \; pr_2 \; \;} \, L \]
is injective. So we have an inclusion $\psi : \mathbb{Z}\Sigma^{\vee} \hookrightarrow L$. Let $e$ denote the exponent of the  finite group $L / \, \mathbb{Z} \Sigma^{\vee}$. By definition $e$ is the biggest elementary divisor of the inclusion $\psi: \mathbb{Z} \Sigma^{\vee} \hookrightarrow L$. Notice that this is also the biggest elementary divisor of the natural inclusion $\mathbb{Z} \Sigma^{\vee} \subset \mathbb{Z} \Phi^{\vee} \xlongequal{} K \oplus L$. The discussion up to now implies that $J(\mathbf{H}_{der}) \leq e \, d \, J(\mathbf{G})$.

 We are left to compute the elementary divisors $d$ and $e$ of the inclusions of the root and coroot lattices. We first claim that $d \leq \text{hgt}(\mathfrak{g})^{R-1}$. In order to prove the claim, we will use $\Delta_{\Sigma}$ and $\Delta_{\Phi}$ as bases for the root lattices. For each $\alpha \in \Delta_{\Sigma}$, we can write $\alpha = \sum_{\beta \in \Delta_{\Phi}} m_{\beta}^{\alpha} \, \beta $ for some nonnegative integers $m_{\beta}^{\alpha}$. Set $M \vcentcolon = (m_{\beta}^{\alpha})_{\beta \in \Delta_{\Phi}, \, \alpha \in \Delta_{\Sigma} }$. This is a $R\times c$ matrix representing the inclusion $\mathbb{Z} \Sigma \hookrightarrow \mathbb{Z}\Phi$. By the theory of Smith normal form, $d$ divides all $c \times c$-minors of $M$. Since all $m_{\beta}^{\alpha}$ are nonnegative, such $c \times c$-minor is bounded by
\[ \prod_{\alpha \in \Delta_{\Sigma}} \left(\sum_{\beta \in \Delta_{\Phi}} m_{\beta}^{\alpha} \right) \, \leq \, \text{hgt}(\mathfrak{g})^c \, \leq \, \text{hgt}(\mathfrak{g})^{R-1}  \]
The claim $d \leq \text{hgt}(\mathfrak{g})^{R-1}$ follows. We can apply the same argument to the inclusion $\mathbb{Z}\Sigma^{\vee} \subset \mathbb{Z} \Phi^{\vee}$. The maximal height in the dual root system $\Phi^{\vee}$ is also $\text{hgt}(\mathfrak{g})$. Therefore the same proof yields a bound $e \leq \text{hght}(\mathfrak{g})^{R-1}$. This implies
\[J(\mathbf{H}_{der}) \, \leq \,  e\, d\, J(\mathbf{G}) \, \leq \, \text{hgt}(\mathfrak{g})^{2R-2} \cdot J(\mathbf{G}) \]
\end{proof}
\begin{prop} \label{prop: ramification reduction reductive}
Let $\mathbf{G}$ be connected reductive. Let $A \in \mathfrak{g}_F$ be a connection. Then there exist $x \in \mathbf{G}(F_b)$ for some positive integer $b$ such that $x \cdot A$ is in canonical form. If we set $R \vcentcolon = \text{rank} (\mathbf{G}_{der})$, then $b$ can be chosen so that
\[ b \; \leq \; 2 \, \text{hgt}(\mathfrak{g})^{2R-1} \cdot J(\mathbf{G}_{\text{der}}) \; \cdot \prod_{j=0}^{\left \lfloor{\frac{\text{dim}(\mathbf{G}_{\text{der}})}{3}} \right \rfloor} \left( 4\, \text{hgt}(\mathfrak{g}) +2 \right)^{\left \lfloor{} \frac{1}{2}\left(\text{dim}(\mathbf{G}_{\text{der}}) -3j \right) \right \rfloor} \]
\end{prop}
\begin{proof}
We have to keep track of how much ramification is needed to perform each of the steps in Algorithm \ref{algor: reduction reductive}. Recall our six operations:\vspace{-0.25cm}
\begin{enumerate}[(i)]
    \item Apply Lemma \ref{lemma: not nilpotent case}. No ramification is needed for this operation, as is apparent from the proof of the lemma.
    \item Apply Lemma \ref{lemma: nilpotent case aligned}. No ramification is needed for this operation. This also follows directly from the proof of the lemma.
    \item Apply Proposition \ref{prop: main prop reduction} Case 1. We need to pass to a $2$-cover.
    \item Apply Proposition \ref{prop: main prop reduction} Case 2. We need to pass to a $2b$-cover, where $b$ is such that $b \delta \in \mathbb{Z}$. By Remark \ref{remark: ramification bound key lemma}, we know that we can choose $b\leq 2\Lambda -1 \leq 2\text{hgt}(\mathfrak{g}) +1$.
    \item Find the canonical form of a connection in a torus. No ramification is needed to perform this operation, by the proof of Proposition \ref{prop:tori regular}.
    \item Find the canonical form for a connection of the first kind (as in Theorem \ref{thm:semisimple regular}). By Lemma \ref{lemma: ramification bound regular semisimple}, we can perform this operation after passing to a $b$-cover with $b \leq \text{hgt}(\mathfrak{g}) \cdot I(\mathbf{G})$.
\end{enumerate}
\vspace{-0.25cm}
We know that operations (iii) and (vi) will be used only once, at the end of the algorithm. This gives us a factor of $2\, \text{hgt}(\mathfrak{g}) \cdot I(\mathbf{H}_{\text{der}})$, where $\mathbf{H}$ is the centralizer $Z_{\mathbf{G}_{der}}\left(\{D_1, D_2, ..., D_l\}\right)$ of a finite set of pariwise commuting semisimple elements $D_i$ in $\mathfrak{g}_{der}$. Since $I(\mathbf{H}_{der}) \leq J(\mathbf{H}_{der})$, this is bounded by $2\, \text{hgt}(\mathfrak{g}) \cdot J(\mathbf{H}_{\text{der}})$. By Lemma \ref{lemma: index of centralizers} we known that $J(\mathbf{H}_{\text{der}}) \leq \text{hgt}(\mathfrak{g})^{2R-2} \cdot J(\mathbf{G}_{\text{der}})$. This yields the first factor in the bound above.

We are now left to count the amount of times that we need to apply operation (iv) in our algorithm. Each time we apply it we need to pass to a cover of ramification at most $4\, \text{hgt}(\mathfrak{g}) + 2$. By the remark after the proof of Theorem \ref{thm: reduction conn reductive}, we need to apply operation (iv) at most $\left \lfloor{ \frac{1}{2} \text{dim}(\mathbf{G}_{\text{der}})} \right \rfloor$ times before we are in the case when $A_r$ is not nilpotent. We therefore pick up a ramification of at most $\left( 4\, \text{hgt}(\mathfrak{g}) +2 \right)^{\left \lfloor{} \frac{1}{2}\text{dim}(\mathbf{G}_{\text{der}}) \right \rfloor}$. After that we change our group.

We can apply operation (i) and split off the central part in order to pass to a proper semisimple subgroup $\mathbf{H}_{der} \vcentcolon = \left(Z_{\mathbf{G}}\left(\,(A_r)_s \,\right)\right)_{\text{der}}$. Notice that $\text{dim}(\mathbf{H}_{der}) \leq \text{dim}(\mathbf{G}_{\text{der}})-3$, because we are removing at least two root spaces (positive and negative pair) and the nontrivial central torus of the centralizer $\mathbf{H}$. Now we start all over again. We know that we need to apply operation (iv) at most $\left \lfloor{ \frac{1}{2} \left(\text{dim}(\mathbf{G}_{\text{der}})  -3 \right)} \right \rfloor$-many times until $A_r$ is not nilpotent. So we pick up a ramification of at most $\left( 4\, \text{hgt}(\mathfrak{g}) +2 \right)^{\left \lfloor{ \frac{1}{2} \left(\text{dim}(\mathbf{G}_{\text{der}})  -3 \right)} \right \rfloor}$. Iterating this procedure, we get the product appearing in the bound above.
\end{proof}
\begin{remark}
In terms of dimension, the right hand side is $J(\mathbf{G}_{\text{der}}) \cdot e^{O\left(\text{dim}(\mathbf{G}_{\text{der}})^2 \, \text{log} \,\text{dim}(\mathbf{G}_{\text{der}}) \right)}$.
\end{remark}
We proceed to establish a quantitative refinement of Theorem \ref{thm: reduction conn reductive}. It essentially follows from keeping track of some indices in the operations for the algorithm above. It makes sense once we know that the irregular part of a canonical form is unique up to conjugacy, as stated in Theorem \ref{thm: uniqueness reduction arbitrary} below. The following result over $\mathbb{C}$ can already be found in the work of Babbitt and Varadarajan \cite[9.7]{Babbitt.varadarajan.formal}.
\begin{prop}[Determinacy for the Irregular part of the Canonical Form] \label{prop: determinacy.irregular}
Let $\mathbf{G}$ be connected reductive. Let $A= \sum_{j=r}^{\infty} A_j t^j$ be a connection in $\mathfrak{g}_{F}$. The irregular part of the canonical form of $A$ depends only on $A_{r+m}$ for $0 \leq m < \left(\text{hgt}(\mathfrak{g})+1\right)(|r|-1)$.
\end{prop}
\begin{proof}
It suffices to check the steps in Algorithm \ref{algor: reduction reductive}. In some steps we replace the group $\mathbf{G}$ by a proper subgroup (either a centralizer or the derived subgroup). This can only decrease the quantity $\left(\text{hgt}(\mathfrak{g})+1\right)(|r|-1)$, so we can safely ignore these changes of groups. We are left to study the effect of operations (i)-(vi) in Algorithm \ref{algor: reduction reductive}.

The last operation (vi) has no effect on the irregular part of the connection, so there is nothing to do here. Step (v) takes a connection $A = \sum^{\infty}_{j = r} A_j \, t^j$ and outputs its truncation $A = \sum^{-1}_{j = r} A_j \, t^j$ (see the proof of Proposition \ref{prop:tori regular}). The output is therefore determined by the coefficients given in the statement of the proposition.

Step (iii) outputs a connection with no irregular part. Notice that in Proposition \ref{prop: main prop reduction}  we can determine whether we are in Case 1 (i.e. when we have to perform Step (iii)) based on the value of $\delta$. This depends only on the $A_{r+m}$ for $0 \leq m < \Lambda\left(A_r\right)(|r|-1)$. Since $\Lambda\left(A_r\right) \leq \text{hgt}(\mathfrak{g}) +1$ by Example \ref{example: regular nilp}, this case can be determined by the coefficients provided.

For the remaining operations (i), (ii) and (iv), we start with a given connection $A$ with lowest coefficient $A_r$ and output an irregular connection $B$ with lowest coefficient $B_{r'}$. The proposition will follow if we can prove that for each of these operations the coefficients 
$B_{r' +m}$ for $0 \leq m < \left(\text{hght} +1\right)\left(|r'| - 1\right)$ are completely determined by the coefficients $A_{r+m}$ for $0 \leq m < \left(\text{hgt}(\mathfrak{g})+1\right)(|r|-1)$.

Operations (i) and (ii) are very similar. In this case we have $r = r'$.  Let $m$ be an integer. From the proofs of Lemma \ref{lemma: not nilpotent case} and Lemma \ref{lemma: nilpotent case aligned}, it follows that $B_m$ is determined by $A_j$ for $j \leq m$.  So we are done with these operations.

We are left with operation (iv). Recall from the proof of Proposition \ref{prop: main prop reduction} Case 2 that we have
\begin{align*}
    B = t^{-b \delta H} \cdot \tilde{A} & = \; \; 2b \sum_{m=0}^{\infty} \text{Ad}(t^{-b \delta H}) A_{r+m} \, t^{2b(r+m)+2b-1} \; - \;b\delta Ht^{-1}
    \end{align*}
    The term $b\delta Ht^{-1}$ is determined by the knowledge of $\delta$ and $H = \left(A_r)\right)_s$. 
    For the infinite sum, we can write the root decompositions $A_{r+m} = \sum_{\beta \in \Phi } A^{\beta}_{r+m}$ and use that $\text{Ad}(t^{-b \delta H}) \, A^{\beta}_{r+m} = t^{ -b \delta \beta(H)} \, A^{\beta}_{r+m}$ in order to get
    \begin{align*}
     \text{Ad}(t^{-b \delta H}) A^{\beta}_{r+m} \, t^{2b(r+m)+2b-1} \; = A^{\beta}_{r+m} \, t^{-b \delta \beta(H)+2b(r+m)+2b -1} = A^{\beta}_{r+m} \, t^{r' +2bm -b \delta \beta(H)}
    \end{align*}
By Example \ref{example: regular nilp} we know $\beta\left(H\right) \leq 2 \, \text{hgt}(\mathfrak{g})$. Suppose that a positive integer $m$ satisfies $2bm -b \delta \beta(H) < \left(\text{hgt}(\mathfrak{g})+1\right)(|r'|-1)$. Some algebraic manipulations show that $m < \left(\text{hght} +1\right)\left(|r| - 1\right)$. So indeed the coefficients $B_{r' +m}$ for $0 \leq m < \left(\text{hght} +1\right)\left(|r'| - 1\right)$ are completely determined by the coefficients $A_{r+m}$ for $0 \leq m < \left(\text{hgt}(\mathfrak{g})+1\right)(|r|-1)$.
\end{proof}

\begin{remark}
We can think of Proposition \ref{prop: determinacy.irregular} as a continuity statement. It says that a small perturbation of the original connection will not alter the irregular part of its canonical form. This is analogous to the finite determinacy theorem for analytic singularities, as in \cite{djong.localanalytic} Theorem 9.4 in page 313.
\end{remark}
One can obtain a similar continuity statement for the residue of the canonical connection (i.e. the coefficient of $t^{-1}$ in the canonical form). However the explicit bound for the number of terms needed is complicated and not very illuminating. We refrain from including a formula for the bound.
\begin{prop} \label{prop: determinacy everything} Let $\mathbf{G}$ be connected reductive and let $A \in \mathfrak{g}_F$ be a connection. There exist a positive integer $n$ such that all connections $C \in \mathfrak{g}_F$ satisfying $C \equiv A \; (mod \; t^{n})$ are $\mathbf{G}(\overline{F})$-gauge equivalent to $A$.
\end{prop}
\begin{proof}
In Algorithm \ref{algor: reduction reductive} we apply operations (iii) and (vi) exactly once at the very end. Suppose that we are given the coefficients $A_{r+m}$ for $0 \leq m \leq \left(\text{hgt}(\mathfrak{g})+1\right)(|r|-1)$ in a given connection $A$. Let $D$ be the output of applying one of the operations (i), (ii), (iv) or (v) to $A$. The proof of Proposition \ref{prop: determinacy.irregular} implies that we can determine the corresponding coefficients $D_{r'+m}$ for $0 \leq m \leq \left(\text{hgt}(\mathfrak{g})+1\right)(|r'|-1)$.

We can iterate this reasoning. Suppose that $D= \sum_{j = r'}^{\infty} D_j \, t^j$ is the ouput of the algorithm before applying the last two steps (operations (iii) and (vi)). Then we see that the coefficients $0 \leq m < \left(\text{hgt}(\mathfrak{g})+1\right)(|r'|-1)$ are completely determined by $A_{r+m}$ for $0 \leq m \leq \left(\text{hgt}(\mathfrak{g})+1\right)(|r|-1)$, where $A$ is the original connection we start with. The number of steps needed in the algorithm is also completely determined.

By the statement of Proposition \ref{prop: main prop reduction} Case 1, we will be able to determine the residue (i.e. the coefficient of $t^{-1}$) when we apply operation (iii) to $D$. The output of operation (iii) will be a connection of the first kind $B= \sum_{j=-1}^{\infty} B_j \, t^j$, and we can compute $k\left(B_{-1}\right)$ (see the remark after Lemma \ref{lemma:aligned}).

Now we need to determine the result of applying operation (vi) to $B$ as above. By Remark \ref{remark: determinacy terms regular semisimple}, this will be determined by $B_j$ for $-1 \leq j \leq k(B_{-1})$. We can then work backwards using an argument similar to the proof of Proposition \ref{prop: determinacy.irregular} to find a number $n$ big enough so that the coefficients $B_j$ for $-1 \leq j \leq k(B_{-1})$ are determined by $A_j$ for $r \leq j < n$.
\end{proof}
\section{Irregular connections for arbitrary linear algebraic groups} \label{section: general irregular}
We proceed as in the regular case. Just as before, we start with solvable groups.
\subsection{Irregular connections for solvable groups}
 We will again make use of the map $\pi : \text{Lie}(\mathbf{\mathbf{T}}) \cong X_{*}(\mathbf{T}) \otimes \mathbf{k} \longrightarrow X_{*}(\mathbf{T})\otimes\mathbb{Q}$ as in Proposition \ref{prop: reduction regular solvable}.
\begin{prop} \label{prop: reduction solvable}
Let $\mathbf{G}$ be of the form $\mathbf{T} \ltimes \mathbf{U}$, where $\mathbf{T}$ is a torus and $\mathbf{U}$ is unipotent. Let $A \in \mathfrak{g}_F$ be a formal connection. Write $A = A_{\mathbf{T}} + A_{\mathbf{U}}$ for  some $A_{\mathbf{T}} \in \text{Lie}(\mathbf{T})_F$ and $A_{\mathbf{U}} \in \text{Lie}(\mathbf{U})_F$. Let $b$ be a positive integer such that $b \, \pi\left(\, (A_{\mathbf{T}})_{-1} \, \right) \in X_{*}(\mathbf{T})$. 
Then there exists $x \in \mathbf{G}(F_b)$ such that $x \cdot A = \sum_{j=1}^{l} D_j \, t^{r_j} \, + \, t^{-1}\, C$ with:\vspace{-0.25cm}
\begin{enumerate}[(1)]
    \item $r_j \in \mathbb{\mathbb{Z}}_{<-1}$ for all $j$.
    \item $D_j \in \text{Lie}(\mathbf{T})$ for all $j$.
    \item $[D_j, \, C] =0$ for all $j$.
    \item $\pi(C_s) =0$.
\end{enumerate}
\vspace{-0.25cm}
\end{prop}
\begin{proof}
The global structure of the proof is very similar to the argument in Proposition \ref{prop: reduction regular solvable}. Write $A_{\mathbf{T}} = \sum_{j = -q}^{\infty} t^{j} \, D_j$. By Proposition \ref{prop:tori irregular} (a), we can find $g \in \mathbf{T}(F)$ with $g \cdot A_{\mathbf{T}} = \sum_{j = -q}^{-1} t^j \, D_j$. Set $\mu \vcentcolon = b \, \pi\left(D_{-1}\right) \in X_{*}(\mathbf{T})$. Then we have $(t^{\frac{1}{b} \mu} \, g) \cdot A_{\mathbf{T}} = \sum_{j=-q}^{-2} D_j \, t^{j} \, + \,t^{-1} \, C_{\mathbf{T}}$ for some $C_{\mathbf{T}} \in \text{Lie}(\mathbf{T})$  with $\pi(C_{\mathbf{T}}) =0$.

Replace $A$ with $B \vcentcolon = (t^{\frac{1}{b} \, \mu} \, g) \cdot A$. We know that $B = \sum_{j=-q}^{-2} D_j \, t^j \, + \, t^{-1} \, C_{\mathbf{T}} + B_{\mathbf{U}}$ for some $B_{\mathbf{U}} \in \text{Lie}(\mathbf{U})_{F_b}$. By lifting to the $b$-ramified cover, we can assume that $B_{\mathbf{U}} \in \text{Lie}(\mathbf{U})_F$. We claim that we can find $u \in \mathbf{U}(F)$ such that $u \cdot B = \sum_{j =-q}^{-2} D_j \, t^j \, + \, t^{-1}\, C_{\mathbf{T}} + t^{-1} \, C_{\mathbf{U}}$ with $C_{\mathbf{U}} \in \text{Lie}(\mathbf{U})$ and $[C_{\mathbf{T}}, C_{\mathbf{U}}] = [D_j, C_{\mathbf{U}}] =0$ for all $j$. We will show this by induction on the dimension of $\mathbf{U}$.

The base case is $\mathbf{U} = \mathbb{G}_a$. Then, $\mathbf{T}$ acts on $\mathbf{U}$ by a character $\chi : \mathbf{T} \longrightarrow \mathbb{G}_m$. For $u = \sum_{j=r}^{\infty} u_j \, t^j \in \mathbf{U}(F)$, we have
\[u \cdot B = t^{-1}\, C_{\mathbf{T}} \, + \, B_{\mathbf{U}} \, - \, \sum_{j=r-q}^{\infty} \left[\left( d\chi(C_{\mathbf{T}}) -j\right) u_j  + \sum_{i = 2}^{q} d\chi(D_i) u_{j+i-1}\right]\, t^{j-1}   \]
We have two cases\vspace{-0.25cm}
\begin{enumerate}[(1)]
    \item Suppose that $d\chi(D_j) \neq 0$ for some $j$. Then, we can solve the recurrence 
    \[\left(d\chi(C_{\mathbf{T}}) -j\right) u_j  + \sum_{i = 2}^{q} d\chi(D_i) u_{j+i} = B_{j-1}\]
    with initial values $u_j = 0$ for $j \ll 0$. This yields an element $u \in \mathbf{U}(F)$ with $u \cdot B = \sum_{j = -q}^{-2} D_j \, t^j \, + \, t^{-1} \, C_{\mathbf{T}}$.
    \item Suppose that $d\chi(D_j) = 0$ for all $j$. The argument for the base case in Proposition \ref{prop: reduction regular solvable} shows that there is an element $u \in \mathbf{U}(F)$ such that $u \cdot B = \sum_{j=-q}^{-2} D_j \, t^j \, + \, t^{-1} \, C_{\mathbf{T}} \,+ \,t^{-1} \,C_{\mathbf{U}}$ for some $C_{\mathbf{U}} \in \text{Lie}(\mathbf{U})$ satisfying $[C_{\mathbf{T}}, C_{\mathbf{U}}] = 0$. Notice that we have $[D_j, \, C_{\mathbf{U}}] = d\chi(D_j)\, C_{\mathbf{U}} = 0$ by assumption. So we are done in this case.
\end{enumerate}
\vspace{-0.25cm}
This concludes the proof of the base case.

Let's proceed with the induction step. We can decompose the action of the split torus $\mathbf{T}$ on the vector space $Z_{\mathbf{U}}$ into one-dimensional spaces. Let $\mathbf{H} \cong \mathbb{G}_a \leq Z_{\mathbf{U}}$ be one of these eigenspaces. Let $s$ be a $\mathbf{T}$-equivariant section of the morphism of schemes $\mathbf{U} \longrightarrow \mathbf{U} / \, \mathbf{H}$ as in the proof of Proposition \ref{prop: reduction regular solvable}.

Let $\overline{B}$ be the image of $B$ in the quotient $\text{Lie}(\mathbf{U} / \, \mathbf{H})_F$. By the induction hypothesis, we can find $\overline{u} \in \mathbf{U} / \, \mathbf{H} (F)$ such that $\overline{u} \cdot \overline{B} = \sum_{j = -q}^{-2} D_j \, t^j \, + \, t^{-1} \, C_{\mathbf{T}} + t^{-1} \, \overline{E}$ for some $\overline{E} \in \text{Lie}\left(\mathbf{U}/ \, \mathbf{H}\right)$ with $[D_j, \overline{E}] = [C_{\mathbf{T}}, \overline{E}] =0$. We can then write
\[ s(\overline{u}) \cdot B = \sum_{j = -q}^{-2} D_j \, t^j \, + \, t^{-1}\, C_{\mathbf{T}} + t^{-1}\, ds(\overline{E}) + B_{\mathbf{H}}\]
for some $B_{\mathbf{H}} \in \text{Lie}(\mathbf{H})_F$. Since $s$ is $\mathbf{T}$-equivariant, we have $[ds(\overline{E}), D_j] = [ ds(\overline{E}), \, C_{\mathbf{T}}] = 0$. We can use the base case for $\mathbf{H}$ in order to conclude.
\end{proof}
We end with a generalization of Proposition \ref{prop: determinacy regular solvable}. We will use the same notation as in the regular case. Let $A = A^{\mathbf{T}} + A^{\mathbf{U}}$ be a formal connection with $A^{\mathbf{T}} \in \text{Lie}(\mathbf{T})_F$ and $A^{\mathbf{U}} \in \text{Lie}(\mathbf{U})_F$. Write $A^{\mathbf{T}} = \sum_{j = -q}^{-1} A^{\mathbf{T}}_j \, t^j \, + \,  \sum_{j = p}^{\infty} A^{\mathbf{T}}_j \, t^j$ for some $q, p \geq 0$. Also, write $A^{\mathbf{U}} = \sum_{j = m}^{\infty} A^{\mathbf{U}}_j \, t^j$.
\begin{prop} \label{prop: determinacy irregular solvable}
Keep the same notation as above. Assume that $\mathbf{U}$ has nilpotency class $n$.\vspace{-0.25cm}
\begin{enumerate}[(i)]
    \item Suppose that $m > L-1$. Then there exists $x \in \mathbf{G}(\mathcal{O})$ such that $x \cdot A = \sum_{-q}^{-1} A^{\mathbf{T}}_j \, t^j$. More precisely, there exist $x_{\mathbf{T}} \in \mathbf{T}(\mathcal{O})$ with $x_{\mathbf{T}} \equiv 1_{\mathbf{T}} \, \left(mod \; t^{p+1}\right)$ and $x_{\mathbf{U}} \in \mathbf{U}(\mathcal{O})$ with $x_{\mathbf{U}} \equiv 1_{\mathbf{U}} \, \left(mod \; t^{m+1}\right)$ such that $(x_{\mathbf{U}} x_{\mathbf{T}}) \cdot A = \sum_{-q}^{-1} A^{\mathbf{T}}_j \, t^j$.
    \item Suppose that $m \leq L-1$. Then the $\mathbf{G}(F)$-gauge equivalence class of $A$ is determined by the coefficients $A^{\mathbf{T}}_j$ for $-q \leq j < (n+1)(|m|-1) +L$ and $A^{\mathbf{U}}_j$ for $-q \leq j < n(|m|-1) +L$. More precisely, suppose that there is another connection $B$ and an integer $k \geq n(|m|-1)+L$ satisfying $A^{\mathbf{T}} \equiv B^{\mathbf{T}} \, \left( mod \; t^{k +|m| -1}\right)$ and $A^{\mathbf{U}} \equiv B^{\mathbf{U}} \, \left( mod \; t^k\right)$. Then, there exists $x \in \mathbf{G}(\mathcal{O})$ with $x \equiv 1 \, \left(mod \; t^{k - n|m| +n+ 1}\right)$ such that $x \cdot A = B$.
\end{enumerate}
\vspace{-0.25cm}
\end{prop}
\begin{proof}
The proof is similar in spirit to the argument for Proposition \ref{prop: determinacy regular solvable}, but it involves an extra subtle twist to deal with the negative powers.\vspace{-0.25cm}
\begin{enumerate}[(i)]
    \item Just as in Proposition \ref{prop: determinacy regular solvable}, we can find $x_{\mathbf{T}} \in \mathbf{T}(\mathcal{O})$ with $x_{\mathbf{T}} \equiv 1_{\mathbf{T}} \, \left(mod \; t^{p+1}\right)$ such that
    \[ C \vcentcolon = x_{\mathbf{T}} \cdot A = \sum_{-q}^{-1} A^{\mathbf{T}}_j \, t^j \, + \, C^{\mathbf{U}} \]
    for some $C^{\mathbf{U}} \in \mathfrak{u}_{\mathcal{O}}$. Moreover we have $C^{\mathbf{U}} \equiv 0 \, \left(mod \; t^{m} \right)$. We claim that there exists $u \in \mathbf{U}(\mathcal{O})$ with $u \equiv 1_{\mathbf{U}} \, \left(mod \; t^{m+1} \right)$ such that $u \cdot C = \sum_{-q}^{-1} A^{\mathbf{T}}_j \, t^j$. This claim finishes the proof of part (i).
    
    In order to prove the claim, we will actually show something stronger. Let us fix some notation. By \cite{borel.springer} Corollary 9.12, there is a $\mathbf{T}$-equivariant map of $\mathbf{k}$-schemes $\psi_{\mathbf{U}}: \mathbf{U} \longrightarrow \mathfrak{u}$. We can define this map so that the following diagram commutes
    \[ \xymatrix{
\mathbf{U}  \ar[r] \ar[d]^{\psi_{\mathbf{U}}} & \mathbf{U} / \, Z_{\mathbf{U}} \ar[d]^{\psi_{\mathbf{U}/ \, Z_{\mathbf{U}}}} \\
\mathfrak{u} \ar[r] & \mathfrak{u}/ \, \mathfrak{z}} \]
    Here $Z_{\mathbf{U}}$ is the center of $\mathbf{U}$ and $\mathfrak{z} = \text{Lie}(Z_{\mathbf{U}})$. Notice that $Z_{\mathbf{U}}$ is just a direct sum of copies of $\mathbb{G}_a$. The corresponding map $\psi_{Z_{\mathbf{U}}}$ can be taken to be the usual identification of a vector space with its tangent space at the identity. By iterating, we can arrange so that we get a corresponding compatibility at each step of the upper central series of $\mathbf{U}$.
    
   Recall that we have a weight decomposition $\mathfrak{u} = \bigoplus_{i =1}^l \mathfrak{u}_{\chi_i}$. Via the isomorphism $\psi_{\mathbf{U}}$, we can get a decomposition $\mathbf{U} = \prod_{\chi_i} \mathbf{U}_{\chi_i}$ as a product of schemes. For $u \in \mathbf{U}(\mathbf{k})$, we will denote by $u_{\chi_i}$ the corresponding component in $\mathbf{U}_{\chi_i}$.
    
    For each $i$, define $a_i$ to be the biggest positive integer $j$ such that $d\chi_i\left(A^{\mathbf{T}}_{-j}\right) \neq 0$. If $d\chi_i\left(A^{\mathbf{T}}_{-j}\right) = 0$ for all $j>0$, we set $a_i = 1$. Then, we claim that we can find $u \in \mathbf{U}(\mathcal{O})$ with $u_{\chi_i} \equiv 1_{\mathbf{U}} \, \left(mod \; t^{m +a_i} \right)$ such that $u \cdot C = \sum_{-q}^{-1} A^{\mathbf{T}}_j \, t^j$. We will prove this stronger claim by induction on the nilpotency class of $\mathbf{U}$.
    
    For the base case $n=0$, we have $\mathbf{U} \cong \mathbb{G}_a^d$ for some $d$. By decomposing into one-dimensional $\mathbf{T}$-modules and looking at each coordinate, we can reduce to the case $d =1$. So we have a single weight space $\mathfrak{u}_{\chi_i}$. This case amounts to solving a recurrence as in the computation for the base case in Proposition \ref{prop: reduction solvable}. We want to find $u = \sum_{j =0}^{\infty} u_j \, t^j$ satisfying
    \[\left(d\chi_i(A^{\mathbf{T}}_{-1}) -j\right) u_j  + \sum_{k = 2}^{q} d\chi_i(A^{\mathbf{T}}_{-k}) u_{j+k} = C^{\mathbf{U}}_{j-1}\]
    By the definition of $a_i$, this is the same as
     \[\left(d\chi_i(A^{\mathbf{T}}_{-1}) -j\right) u_j  + \sum_{k = 2}^{a_i} d\chi_i(A^{\mathbf{T}}_{-k}) u_{j+k} = C^{\mathbf{U}}_{j-1}\]
     There are two different cases.\vspace{-0.25cm}
     \begin{enumerate}[(1)]
         \item If $a_i =1$, then the recurrence reduces to 
             \[\left(d\chi_i(A^{\mathbf{T}}_{-1}) -j\right) u_j  = C^{\mathbf{U}}_{j-1}\]
             The claim follows by the argument for the base case in Proposition \ref{prop: determinacy regular solvable}.
             \item Suppose that $a_i \neq 1$. We know that $d\chi_i(A^{\mathbf{T}}_{-a_i}) \neq 0$. We can solve the recurrence by rewriting
             \[d\chi_i(A^{\mathbf{T}}_{-a_i}) \, u_{j + a_i} = C^{\mathbf{U}}_{j-1} -\left(d\chi_i(A^{\mathbf{T}}_{-1}) -j\right) u_j  - \sum_{k = 2}^{a_i -1} d\chi_i(A^{\mathbf{T}}_{-k}) u_{j+k}\]
             Since $C^{\mathbf{U}}_j = 0$ for all $j \leq m-1$, we can set $u_j =0$ for all $j \leq m +a_i$. Then we can solve for the rest of the $u_j$ using the recursion formula above.
     \end{enumerate}
    Let's proceed with the induction step. Notice that $\mathfrak{z}$ is a direct sum of some one-dimensional $\mathbf{T}$-submodules of $\mathfrak{u}$. We can get an identification of $\mathfrak{u}/ \, \mathfrak{z}$ with the direct sum of some choice of remaining one-dimensional $\mathbf{T}$-submodules. This way we get a $\mathbf{T}$-equivariant inclusion $\mathfrak{u}/ \, \mathfrak{z} \hookrightarrow \mathfrak{u}$. We can get a $\mathbf{T}$-equivariant section $s: \mathbf{U}/ \, Z_{\mathbf{U}} \longrightarrow \mathbf{U}$ defined by the composition
    \[ s: \mathbf{U} / \, Z_{\mathbf{U}} \, \xrightarrow{ \; \psi_{\mathbf{U}/ \, Z_{\mathbf{U}}} \;} \, \mathfrak{u}/ \, \mathfrak{z} \, \xhookrightarrow{\; \; \; \;} \, \mathfrak{u} \, \xrightarrow{ \; \psi_{\mathbf{U}}^{-1} \;} \, \mathbf{U}  \]
    Let $\overline{C}$ be the image of $C$ in the quotient $\text{Lie}(\mathbf{T} \ltimes \mathbf{U} / \, Z_{\mathbf{U}})_{F_b}$. By the induction hypothesis, there exists $\overline{x} \in \mathbf{U} / \, Z_{\mathbf{U}}(\mathcal{O})$ such that $\overline{x}_{\chi_i} \equiv 1 \; \left(mod \; t^{m + a_i} \right)$ and $\overline{x} \cdot \overline{C} = \sum_{-q}^{-1} A^{\mathbf{T}}_j \, t^j$. By the $\mathbf{T}$-equivariance of $s$, we must then have $s(\overline{x}) \cdot C = \sum_{-q}^{-1} A^{\mathbf{T}}_j \, t^j \, + \, D_{Z_{\mathbf{U}}}$ for some $D_{Z_{\mathbf{U}}} \in \text{Lie}(Z_{\mathbf{U}})_{F}$. By definition
\[ s(\overline{x}) \cdot C \; = \; \sum_{-q}^{-1} t^j \text{Ad}(s(\overline{x})) A^{\mathbf{T}}_j \, + \, \text{Ad}(s(\overline{x})) C^{\mathbf{U}} \, + \, ds(\overline{x}) s(\overline{x})^{-1}  \]
Since $s(\overline{x}) \equiv s(\overline{x})^{-1} \equiv 1 \; \left(mod \; t^{m + 1} \right)$, it follows that $ds(\overline{x}) s(\overline{x})^{-1} \equiv 0 \, \left(mod \; t^{m+1}\right)$. Also $\text{Ad}(s(\overline{x})) C^{\mathbf{U}} \equiv C^{\mathbf{U}} \, \left(mod \; t^{m+1}\right)$, because by assumption $C_{\mathbf{U}} \in \mathfrak{u}_{\mathcal{O}}$. We are left to study  $\text{Ad}(s(\overline{x})) A^{\mathbf{T}}_j$.

Consider the map of $\mathbf{k}$-schemes $\varphi_j : \mathbf{U} \longrightarrow \mathfrak{u}$ given by $\varphi_j(u) \vcentcolon = \text{Ad}(u)A^{\mathbf{T}}_j - A^{\mathbf{T}}_j$. By construction $\varphi_j$ is $\mathbf{T}$-equivariant. This means that it must respect the decomposition into weight spaces. In other words, the $\chi_i$-coordinate of $\varphi_j(u)$ is given by $\varphi_j(u_{\chi_i})$. In particular, this means that
\[\text{Ad}(s(\overline{x})) A^{\mathbf{T}}_j = A^{\mathbf{T}}_j \,  + \, \sum_{i=1}^l \left( \, \text{Ad}(s(\overline{x})_{\chi_i}) A^{\mathbf{T}}_j - A^{\mathbf{T}}_j \,\right)\]
 We have that $\text{Ad}(s(\overline{x})_{\chi_i}) A^{\mathbf{T}}_j = A^{\mathbf{T}}_j$ whenever $d\chi_i\left(A^{\mathbf{T}}_j \right) = 0$. By definition this happens whenever $-j > a_i$. So we get
 \[\text{Ad}(s(\overline{x})) A^{\mathbf{T}}_j = A^{\mathbf{T}}_j \,  + \, \sum_{-j \leq a_i} \left( \, \text{Ad}(s(\overline{x})_{\chi_i}) A^{\mathbf{T}}_j - A^{\mathbf{T}}_j \,\right)\]
 Suppose that $-j \leq a_i$. By assumption $s(\overline{x})_{\chi_i} \equiv 1 \; \left(mod \; t^{m + a_i} \right)$, so in particular $s(\overline{x})_{\chi_i} \equiv 1 \; \left(mod \; t^{m - j} \right)$. Hence we have $\text{Ad}(s(\overline{x})_{\chi_i}) A^{\mathbf{T}}_j \, \equiv \, A^{\mathbf{T}}_j \, \left(mod \; t^{m-j}\right)$. The sum above becomes
 \[\text{Ad}(s(\overline{x})) A^{\mathbf{T}}_j \; \equiv \; A^{\mathbf{T}}_j \; \left(mod \; t^{m-j}\right)\]
Hence $t^j \, \text{Ad}(s(\overline{x}) A^{\mathbf{T}}_j \, \equiv \, t^j \, A^{\mathbf{T}}_j \; \left(mod \; t^{m}\right)$. We can put together all of the discussion above to conclude that
\[ s(\overline{x}) \cdot C \; \equiv \; \sum_{-q}^{-1} A^{\mathbf{T}}_j \, t^j \, + C^{\mathbf{U}} \; = \; C \; \left(mod \; t^{m}\right) \]
Therefore $D_{Z_{\mathbf{U}}} \equiv 0 \; \left(mod \; t^{m}\right)$. Now we can conclude by using the base case for $Z_{\mathbf{U}}$.
    \item The hypothesis implies that we have equality of singular parts $\sum_{j=-q}^{-1} B^{\mathbf{T}}_j \, t^j = \sum_{j=-q}^{-1} A^{\mathbf{T}}_j \, t^j$. The proof of Proposition \ref{prop:tori regular} shows that there exist $x_{\mathbf{T}} \in \mathbf{T}(\mathcal{O})$ with $x_{\mathbf{T}} \equiv 1_{\mathbf{T}} \, \left(mod \; t^{p+1}\right)$ such that $x_{\mathbf{T}} \cdot A^{\mathbf{T}} = B^{\mathbf{T}}$. Set $C \vcentcolon = x_{\mathbf{T}} \cdot A$. We have $C = B^{\mathbf{T}} \, + \, \text{Ad}(x_{\mathbf{T}}) A^{\mathbf{U}}$. Define $C^{\mathbf{U}} \vcentcolon = \text{Ad}(x_{\mathbf{T}}) A^{\mathbf{U}}$. We know that $C^{\mathbf{U}} \equiv A^{\mathbf{U}} \, \left(mod \; t^{k}\right)$, because $x_{\mathbf{T}}\equiv 1 \, \left(mod \; t^{k+|m|}\right)$ and $A^{\mathbf{U}} \in t^m \mathfrak{u}_{\mathcal{O}}$. Therefore $C^{\mathbf{U}} \equiv B^{\mathbf{U}} \ \left(mod \; t^{k}\right)$ by assumption.
    
    Let $s$, $\mathbf{U}_{\chi_i}$ and $a_i$ be defined as in part (i). We claim that there exists $u \in \mathbf{U}(\mathcal{O})$ with $u_{\chi_i} \equiv 1 \; \left(mod \; t^{k-n|m|+n+ a_i} \right)$ such that $u \cdot C = B$. This implies that $u \equiv 1 \, \left(mod \; t^{k-n|m|+n +1}\right)$, so this claim concludes the proof of part (ii). In order to prove the claim, we will induct on the nilpotency class of $\mathbf{U}$. The base case $n=0$ follows again from the explicit computation done in Proposition \ref{prop: reduction solvable}, we omit the details.
    
    Let's proceed with the induction step. Let $\overline{C}$ and $\overline{B}$ denote the images of $C$ and $B$ in the quotient $\text{Lie}(\mathbf{T} \ltimes \mathbf{U} / \, Z_{\mathbf{U}})_{F}$. By the induction hypothesis, there exists $\overline{x} \in \mathbf{U}/ \, Z_{\mathbf{U}}(\mathcal{O})$ with $\overline{x}_{\chi_i} \equiv 1 \; \left( mod \; t^{k-(n-1)|m| +n-1+ a_i} \right)$ such that $\overline{x} \cdot \overline{C} = \overline{B}$.  We can now write $s(\overline{x}) \cdot C = ds\left(\overline{B} \right) +E_{Z_{\mathbf{U}}}$ and $B = ds\left(\overline{B}\right) +K_{Z_{\mathbf{U}}}$ for some $E_{Z_{\mathbf{U}}}, F_{Z_{\mathbf{U}}} \in \text{Lie}(Z_{\mathbf{U}})_{F}$. By definition
\[ s(\overline{x}) \cdot C \; = \; \sum_{j = -q}^{\infty} t^j \text{Ad}(s(\overline{x})) B^{\mathbf{T}}_j \, + \, \text{Ad}(s(\overline{x})) C^{\mathbf{U}} \, + \, ds(\overline{x}) s(\overline{x})^{-1}  \]
Since $s(\overline{x}) \equiv 1 \; \left(mod \; t^{k -(n-1)|m| +n}\right)$, it follows that\\ $t^j \, \text{Ad}(s(\overline{x})) B^{\mathbf{T}}_j \, \equiv \, t^j \,B^{\mathbf{T}}_j \, \left(mod \; t^{k - (n-1)|m| +n}\right)$ for all $j \geq 0$. The same reasoning as in part (i) shows that $t^j \, \text{Ad}(s(\overline{x})) B^{\mathbf{T}}_j \, \equiv \, t^j \,B^{\mathbf{T}}_j \, \left(mod \; t^{k - (n-1)|m| + n-1}\right)$ for all $j <0$. Also we know that $\text{Ad}(s(\overline{x}) C^{\mathbf{U}} \, \equiv \, C^{\mathbf{U}} \, \left(mod \; t^{k -n|m| +n}\right)$, because $s(\overline{x}) \equiv 1 \; \left(mod \; t^{k -(n-1)|m| +n}\right)$ and $C_{\mathbf{U}} \in t^{m} \mathfrak{u}_{\mathcal{O}}$. We conclude that
\[ ds\left(\overline{B}\right) +E_{Z_{\mathbf{U}}} \; = \; s(\overline{x}) \cdot C \; \equiv \; B^{\mathbf{T}} \, + \, C^{\mathbf{U}} \; = \; C \; \left(mod \; t^{k-n|m|+n}\right)\]
Since $k \geq k-n|m|$, we have $C \equiv B \;\left(mod \; t^{k-n|m|}\right)$. It follows that $E_{Z_{\mathbf{U}}} \equiv K_{Z_{\mathbf{U}}} \; \left(mod \; t^{k-n|m|+n} \right)$. Now by the base case we can find $y \in Z_{\mathbf{U}}(\mathcal{O})$ with $y_{\chi_i} \equiv 1 \; \left(mod \; t^{k-n|m|+n+a_i}\right)$ such that $(y\, s(\overline{x}))\cdot C = B$. By the definition of $\mathbf{U}_{\chi_i}$ and its compatibility with the center, we can see that $(y \, s(\overline{x}))_{\chi_i} \equiv 1 \; \left(mod \; t^{k-n|m|+n + a_i}\right)$. The claim follows.
\end{enumerate}
\vspace{-0.5cm}
\end{proof}
\vspace{-0.5cm}
\subsection{Irregular connections for arbitrary linear algebraic groups} \label{subsection: canonical irregular general}
\begin{thm} \label{thm: reduction conn general} Let $\mathbf{G}$ be a connected linear algebraic group. Fix a Levi subgroup $\mathbf{L}$ and a maximal torus $\mathbf{T} \subset \mathbf{L}$. Let $A\in \mathfrak{g}_{\overline{F}}$ be a formal connection. Then there exists $x \in \mathbf{G}(\overline{F})$ such that $x \cdot A = \sum_{j=1}^{l} D_j \, t^{r_j} \, + \, t^{-1}\, C$ with\vspace{-0.25cm}
\begin{enumerate}[(1)]
    \item $r_j \in \mathbb{Q}_{<-1}$ for all $j$.
    \item $D_j \in \text{Lie}(\mathbf{T})$ for all $j$.
    \item $[D_j, \, C] =0$ for all $j$.
    \item $C_s \in \mathfrak{D}$.
    \item $[C_s, C] = 0$.
    \end{enumerate}
    \vspace{-0.25cm}
\end{thm}
\begin{proof}
The same steps as in the proof of Theorem \ref{thm:general regular} reduce the result to the solvable case (Proposition \ref{prop: reduction solvable}).
\end{proof}
A connection of the form $B= \sum_{j=1}^{l} D_j \, t^{r_j} \, + \, t^{-1}\, C$ satisfying conditions (1)-(3) above is said to be in canonical form. Let's formulate some uniqueness results for such irregular canonical forms. Before doing this, we need a lemma.
\begin{lemma} \label{lemma: for uniqueness of canonical}
Let $B  = \sum_{j=1}^l D_{j} \, t^{r_j} + t^{-1} \, C$ and $B' = \sum_{j=1}^s D'_{j} \, t^{r'_j} + t^{-1} \, C'$ be two connections in canonical form. Suppose that $x \in \mathbf{G}(\overline{F})$ satisfies $x \cdot B = B'$. Then all the following statements are true\vspace{-0.25cm}
\begin{enumerate}[(1)]
    \item $l =s$ and $r_j = r'_j$.
    \item $\text{Ad}(x) D_{j} = D'_{j}$ for all $j$.
    \item $x \cdot \, (t^{-1}\, C) = t^{-1} \, C'$.
\end{enumerate}
\vspace{-0.25cm}
\end{lemma}
\begin{proof}
If we know both $(1)$ and $(2)$, then part $(3)$ follows. So we will focus on the first couple of statements. By lifting everything to a ramified cover, we can assume that $x \in \mathbf{G}(F)$. Choose a faithful representation $\mathbf{G} \hookrightarrow \mathbf{\text{GL}_n}$. We can view $x \in \mathbf{\text{GL}_n}(F)$ and $B, B' \in \mathfrak{gl}_n (\overline{F})$.

To simplify notation, let us add some trivial $D_j$ and $D'_j$ so that we have the same indexes and exponents for both $B_{\text{irr}}$ and $B'_{\text{irr}}$. We therefore write $B  = \sum_{j=1}^l D_{j} \, t^{r_j} + t^{-1} \, C$ and $B' = \sum_{j=1}^l D'_{j} \, t^{r_j} + t^{-1} \, C'$. Now $D_j$ and $D'_j$ are (possibly $0$) semisimple elements in $\mathfrak{g}$. We claim that $\text{Ad}(x) D_{j} = D'_{j}$ for all $j$. This claim would imply that none of the new $D_j$ and $D'_j$ are $0$. This would mean that we didn't actually add any extra terms. So both (1) and (2) would follow. We are left to show the claim.

Let us consider the linear transformation $W$ in $\text{End}(\mathfrak{gl}_n) \, (\overline{F})$ given by $W \, v \vcentcolon  = B' v \, - \, v B$ for all $v \in \mathfrak{gl}_n$. We can write $W = \sum_{j=1}^{l} W_j \, t^{r_j} \, + \, t^{-1} \, U$, where
\begin{align*}
    W_j \in \text{End}(\mathfrak{gl}_n) \; \text{is given by} \; W_j \, v \vcentcolon = D'_j v - v D_j\\
    U \in \text{End}(\mathfrak{gl}_n) \; \text{is given by} \; U \, v \vcentcolon = C'v - vC
\end{align*}
Each $W_j$ is semisimple by definition. Also we have that the $W_j$s and $U$ pairwise commute. Therefore there is a simultaneous spectral decomposition $\mathfrak{gl}_n = \bigoplus_{\vec{\lambda}} (\mathfrak{gl}_n)_{\vec{\lambda}}$ for the $W_j$s, where $\vec{\lambda} = (\lambda_j)_{j=1}^{l}$ ranges over a set of $l$-tuples of eigenvalues of the $W_j$s. Note that $W$ preserves this spectral decomposition, because $U$ commutes with all $W_j$s.

The condition $x \cdot B = B'$ can be expressed as $\frac{d}{dt}x = W \, x$. Here we are viewing $x$ as an invertible matrix in $\mathfrak{gl}_n(F)$. We can restrict to the $\vec{\lambda}$-eigenspace and use the decomposition for $W$ in order to see that the component $x_{\vec{\lambda}} \in (\mathfrak{gl}_{n})_{\vec{\lambda}}$ of $x$ satisfies
\[ \frac{d}{dt}x_{\vec{\lambda}} =  \sum_{j=1}^{l} \lambda_j \, t^{r_j} \,x_{\vec{\lambda}} \,  + \, t^{-1} \, U \, x_{\vec{\lambda}}\]
Recall that $r_j < -1$ for all $j$. By comparing the smallest exponent of $t$ in both sides, we conclude that $x_{\vec{\lambda}} = 0$ unless $\vec{\lambda} = \vec{0}$. Hence $x \in (\mathfrak{gl}_n)_{\vec{0}} \, (F)$. This means that $\text{Ad}(x) D_{j} = D'_{j}$ for all $j$.
\end{proof}

As a consequence, we get the following uniqueness result for all irregular canonical forms that satisfy (1)-(5) as in Therorem \ref{thm: reduction conn general}.
\begin{thm} \label{thm: uniqueness reduction arbitrary}
Let $\mathbf{G}$ be a connected linear algebraic group. Fix a Levi subgroup $\mathbf{L}$ and a maximal torus $\mathbf{T}\subset \mathbf{L}$. Let $A = \sum_{j=1}^{l} D_j \, t^{r_j} \, + \, t^{-1}\, C$ and $B = \sum_{j=1}^{l} D'_j \, t^{r'_j} \, + \, t^{-1}\, C'$ be two connections in canonical form. Suppose that $C_s, \,C'_s \in \mathfrak{D}$ and $[C_s, C] = [C'_s, C'] =0$. If there exists $x \in \mathbf{G}(\overline{F})$ with $x \cdot A =B$, then we have \vspace{-0.25cm}
\begin{enumerate}[(1)]
    \item $C_s = C'_s$.
    \item $x \in Z_{\mathbf{G}}(C_s)(\mathbf{k})$.
    \item $\text{Ad}(x) \,D_j = D'_j$ for all $j$.
\end{enumerate}
\vspace{-0.25cm}
\end{thm}
\begin{proof}
This follows from Lemma \ref{lemma: for uniqueness of canonical} combined with Proposition \ref{prop: uniqueness regular arbitrary}.
\end{proof}
We conclude this section with a determinacy result for arbitrary linear algebraic groups.
\begin{prop} \label{prop: determinacy aribitrary linear} Let $\mathbf{G}$ be connected linear algebraic group. Let $A \in \mathfrak{g}_F$ be a connection. There exist a positive integer $n$ such that all connections $C \in \mathfrak{g}_F$ satisfying $C \equiv A \; (mod \; t^{n})$ are $\mathbf{G}(\overline{F})$-gauge equivalent to $A$.
\end{prop}
\begin{proof}
This follows from the corresponding determinacy results for reductive groups (Proposition \ref{prop: determinacy everything}) and solvable groups (Proposition \ref{prop: determinacy irregular solvable}) via a reduction as in the proof of Theorem \ref{thm: reduction conn general}.
\end{proof}
\subsection{Galois cohomology for irregular connections} \label{subsection: galois for irregular connections}
For this section $\mathbf{G}$ will be a connected linear algebraic group. We will fix a choice of Levi subgroup $\mathbf{L} \subset \mathbf{G}$ and maximal torus $\mathbf{T} \subset \mathbf{L}$. Let $B = \sum_{j=1}^l D_{j} \, t^{r_j} + t^{-1} C$ be a connection in canonical form with $C_s \in \mathfrak{D}$ and $[C_s, C] = 0$, as in the statement of Theorem \ref{thm: uniqueness reduction arbitrary}. If $B_{\text{irr}} \neq 0$, then we don't necessarily have $B \in \mathfrak{g}_F$. Suppose that $B$ is in  $\mathfrak{g}_{F_b}$, with $b$ a given positive integer. Then we have a Galois action of $\mathbf{\mu}_b \cong \text{Gal}(F_b/ \,F)$ on $B$ by the formula $\gamma \cdot B = \sum_{j=1}^{l} \gamma^{-b r_j} \, D_j \, t^{r_j} \, + \, t^{-1}\, C$. Because this action is not necessarily trivial, we have to consider twisted cocyles in order to classify connections over $\text{Spec} \, F$ with canonical form $B$.
\begin{defn} \label{defn: twisted cocycles}
Let $b$ be a natural number. Let $B = \sum_{j=1}^l D_{j} \, t^{r_j} + t^{-1} C \; \in \mathfrak{g}_{F_b}$ be a connection in canonical form. A $B$-twisted $\mu_b$-cocycle is a map $\phi \vcentcolon \mathbf{\mu}_b \longrightarrow Z_{G}(C)$ satisfying \vspace{-0.25cm}
\begin{enumerate}[(i)]
    \item $\text{Ad}(\phi_{\gamma})B = \gamma \cdot B$ for all $\gamma \in \mathbf{\mu}_b$.
    \item $\phi_{\gamma \gamma'} = \phi_{\gamma} \phi_{\gamma'}$ for all $\gamma , \gamma' \in \mathbf{\mu}_b$.
\end{enumerate}
\end{defn}
\vspace{-0.25cm}
Fix a compatible choice of generators $\omega_b$ of $\mu_b$ for all $b$ positive, just as we did in the regular case. Note that a $B$-twisted $\mu_b$ cocycle $\phi$ is completely determined by $\phi_{\omega_b} \in Z_G(C)$. This is a an element of finite order dividing $b$, and it satisfies $\text{Ad}(\phi_{\omega_b}) B = \omega_b \cdot B$. Conversely, for any element $\phi_{\omega_b} \in Z_G(C)$ satisfying $\text{Ad}(\phi_{\omega_b}) B = \omega_b \cdot B$ we can get a corresponding $B$-twisted cocycle.

Notice that the centralizer $Z_{G}(\{D_1, ..., D_l, C\})$ acts on the set of $B$-twisted $\mu_b$-cocycles by conjugation. By the same type of general Galois cohomology argument as in the regular case, we get the following couple of propositions. The proofs are omitted.
\begin{prop}[Criterion for Descent to $D^*$] \label{prop: orbit rat}
Let $b$ be a natural number. Let $B = \sum_{j=1}^l D_{j} \, t^{r_j} + t^{-1} C \; \in \mathfrak{g}_{F_b}$ be a connection in canonical form with $C_s \in \mathfrak{D}$ and $[C_s, C] =0$.  Then $B$ is equivalent to a connection in $\mathfrak{g}_F$ via an element of $\mathbf{G}(F_b)$ if and only there exists a $B$-twisted $\mu_b$-cocycle.
\end{prop}
\begin{prop}[Classification of Connections over $D^*$] \label{prop: classif. rat. conn}
Let $B = \sum_{j=1}^l D_{j} \, t^{r_j} + t^{-1} C \; \in \mathfrak{g}_{F_b}$ be a connection in canonical form with $C_s \in \mathfrak{D}$ and $[C_s, C] =0$. Suppose that $B$ satisfies the equivalent statements in Proposition \ref{prop: orbit rat} above for some $b$. Then the set of equivalence classes of $\mathbf{G}$-connections over $D^*$ that become gauge equivalent over $\text{Spec} \, F_b$ are in bijection with the set of $B$-twisted $\mu_b$-cocycles up to $Z_{G}(\{D_1, \,... , D_{l}, \, C\})$-conjugacy.
\end{prop}
The correspondence in Proposition \ref{prop: classif. rat. conn} can be described as follows. Let $\phi_{\omega_b} \in Z_{\mathbf{G}}(C)(\mathbf{k})$ be such that $\text{Ad}(\phi_{\omega_b}) B = \omega_b \cdot B$. By the vanishing of $H^{1}_{\text{Gal}(F)} (\mathbf{G})$, we can find an element $y \in \mathbf{G}(F_b)$ such that $\omega_b\cdot y = y \, \phi_{\omega_b}$. Then the connection associated to $\phi_{\omega_b}$ will be $A = y \cdot B \; \in \mathfrak{g}_F$. Conversely, suppose that $A = y \cdot B$ is a connection in $\mathfrak{g}_F$ for some $y \in \mathbf{G}(F_b)$. We set $\phi_{\omega_b} \vcentcolon = y^{-1} \, \left(\omega_b \cdot y\right)$.

As a consequence of this Galois cohomology classification, we can put a bound on the denominators of the levels $r_j$ of a canonical form for a connection in $\mathfrak{g}_F$. Let $W$ denote the Weyl group of $\mathbf{L}$ with respect to $\mathbf{T}$. A Coxeter element of $W$ is an element of largest length in $W$. All Coxeter elements are conjugate to each other. The Coxeter number $h_{\mathbf{L}}$ of $\mathbf{L}$ is the order of a Coxeter element in $W$.
\begin{prop} \label{prop: coxeter bound levels}
Let $A \in \mathfrak{g}_F$ be a formal connection. Let $B = \sum_{j=1}^l D_{j} \, t^{r_j} + t^{-1} C$ be a connection in canonical form with $C _s \in \mathcal{D}$ and $[C_s, C] = 0$. Suppose that B is $\mathbf{G}(\overline{F})$-gauge equivalent to $A$. Let $b$ be the smallest positive integer such that $B \in \mathfrak{g}_{F_b}$. Then\vspace{-0.25cm}
\begin{enumerate}[(1)]
    \item $b$ divides a fundamental degree of $\text{Lie}(\textbf{L})$. In particular $b \leq h_{\mathbf{L}}$.
    \item If $b = h_{\mathbf{L}}$, then $C_s \in \text{Lie}(Z_{\mathbf{L}}^0)$.
\end{enumerate}
\vspace{-0.25cm}
\end{prop}
\begin{proof}
Recall the notation $B_{irr} = \sum_{j=1}^l D_{j} \, t^{r_j}$. We have $\mathbf{G} = \mathbf{L} \ltimes \mathbf{U}$, where $\mathbf{U}$ is the unipotent radical. Write $\mathfrak{l} \vcentcolon = \text{Lie}(\mathbf{L})$ and $\mathfrak{u} \vcentcolon = \mathbf{U}$. We can decompose $A = A_{\mathfrak{l}} + A_{\mathfrak{u}}$. It follows from the proof of Theorem \ref{thm: reduction conn general} that $B_{irr}$ is given by the irregular part of the canonical form of $A_{\mathfrak{l}}$. Therefore, we can assume without loss of generality that $\mathbf{G} = \mathbf{L}$.

By assumption, we have $B = \mathbf{G}(F_d) \cdot A$ for some $d$ dividing $b$. By Proposition \ref{prop: orbit rat}, we know that there exists a $B$-twisted $\mu_d$-cocycle $\phi$. This means in particular that $\text{Ad}(\phi_{\omega_d}) (B_{irr} +t^{-1} C_s) = \omega_d \cdot B_{irr} +t^{-1} C_s$. We can consider $B_{irr} +t^{-1}C_s$ as an element of $\text{Lie}(\mathbf{T}_{\overline{F}})$. Also $\phi_{\omega_d}$ can be viewed as an element of $\mathbf{G}(\overline{F})$. This means that $B_{irr} +t^{-1}$ and $\omega_d \cdot B_{irr}+t^{-1}$ are $\mathbf{G}(\overline{F})$-conjugate elements of $\text{Lie}(\mathbf{T}_{\overline{F}})$. By \cite{collingwood.nilpotent} Chapter 2, there is an element $w \in W$ such that $w \cdot (B_{irr} t^{-1}C_s) = \omega_d \cdot B_{irr} + t^{-1}C_s$. By definition, $b$ is the least positive integer such that $(\omega_d)^b \cdot B_{irr} = B_{irr}$. We conclude that some of the eigenvalues of $w$ are primitive $b$ roots of unity. It follows that $b$ divides a fundamental degree of $\mathfrak{l}$ by \cite{springer-regular-reflection} Theorem 3.4. If $b = h_{\mathbf{L}}$, then $w$ must be a Coxeter element by \cite{kane-reflections} Theorem 32.2-C. Since $w\cdot C_s = C_s$, we must have $C_s \in \text{Lie}(Z_{\mathbf{L}}^0)$ by the lemma in page 76 of \cite{humphreys-coxeter}.
\end{proof}
\begin{remark}
This does not yield a bound on the ramification needed to put $A$ into canonical form. For example, if $A$ is regular then $B \in \mathfrak{g}_F$. But we have seen that it is sometimes necessary to pass to a ramified cover in order to put a a regular connection into canonical form.
\end{remark}
We remark that part (i) of Proposition \ref{prop: coxeter bound levels} was proven in \cite{chen-depthpreservation} via the existence of oper structures for any connection \cite{frenkel.zhu.opers}. Here we note that there is a direct argument using some facts about Coxeter groups.
\subsection*{Acknowledgements}
This paper grew out of a suggestion from Nicolas Templier to write a modern exposition to \cite{Babbitt.varadarajan.formal}. I am happy to thank him for his very valuable input on the redaction of the manuscript.
\bibliography{formal_reduction_theory_connections_short.bib}

\begin{thebibliography}{Hum95}

\bibitem[Ati57]{atiyah.connections}
M.~F. Atiyah.
\newblock Complex analytic connections in fibre bundles.
\newblock {\em Trans. Amer. Math. Soc.}, 85:181--207, 1957.

\bibitem[BS68]{borel.springer}
A.~Borel and T.~A. Springer.
\newblock Rationality properties of linear algebraic groups. {II}.
\newblock {\em Tohoku Math. J. (2)}, 20:443--497, 1968.

\bibitem[BV83]{Babbitt.varadarajan.formal}
Donald~G. Babbitt and V.~S. Varadarajan.
\newblock Formal reduction theory of meromorphic differential equations: a
  group theoretic view.
\newblock {\em Pacific J. Math.}, 109(1):1--80, 1983.

\bibitem[CK17]{chen-depthpreservation}
Tsao-Hsien Chen and Masoud Kamgarpour.
\newblock Preservation of depth in the local geometric {L}anglands
  correspondence.
\newblock {\em Trans. Amer. Math. Soc.}, 369(2):1345--1364, 2017.

\bibitem[CM93]{collingwood.nilpotent}
David~H. Collingwood and William~M. McGovern.
\newblock {\em Nilpotent orbits in semisimple {L}ie algebras}.
\newblock Van Nostrand Reinhold Mathematics Series. Van Nostrand Reinhold Co.,
  New York, 1993.

\bibitem[Del70]{deligne.regulier}
Pierre Deligne.
\newblock {\em \'{E}quations diff\'{e}rentielles \`a points singuliers
  r\'{e}guliers}.
\newblock Lecture Notes in Mathematics, Vol. 163. Springer-Verlag, Berlin-New
  York, 1970.

\bibitem[DG80]{demazure-gabriel}
Michel Demazure and Peter Gabriel.
\newblock {\em Introduction to algebraic geometry and algebraic groups},
  volume~39 of {\em North-Holland Mathematics Studies}.
\newblock North-Holland Publishing Co., Amsterdam-New York, 1980.
\newblock Translated from the French by J. Bell.

\bibitem[dJP00]{djong.localanalytic}
Theo de~Jong and Gerhard Pfister.
\newblock {\em Local analytic geometry}.
\newblock Advanced Lectures in Mathematics. Friedr. Vieweg \& Sohn,
  Braunschweig, 2000.
\newblock Basic theory and applications.

\bibitem[FZ10]{frenkel.zhu.opers}
Edward Frenkel and Xinwen Zhu.
\newblock Any flat bundle on a punctured disc has an oper structure.
\newblock {\em Math. Res. Lett.}, 17(1):27--37, 2010.

\bibitem[Hum90]{humphreys-coxeter}
James~E. Humphreys.
\newblock {\em Reflection groups and {C}oxeter groups}, volume~29 of {\em
  Cambridge Studies in Advanced Mathematics}.
\newblock Cambridge University Press, Cambridge, 1990.

\bibitem[Hum95]{humphreys-conjugacy}
James~E. Humphreys.
\newblock {\em Conjugacy classes in semisimple algebraic groups}, volume~43 of
  {\em Mathematical Surveys and Monographs}.
\newblock American Mathematical Society, Providence, RI, 1995.

\bibitem[Kan01]{kane-reflections}
Richard Kane.
\newblock {\em Reflection groups and invariant theory}, volume~5 of {\em CMS
  Books in Mathematics/Ouvrages de Math\'{e}matiques de la SMC}.
\newblock Springer-Verlag, New York, 2001.

\bibitem[Lev75]{levelt.cyclic.vector}
A.~H.~M. Levelt.
\newblock Jordan decomposition for a class of singular differential operators.
\newblock {\em Ark. Mat.}, 13:1--27, 1975.

\bibitem[Sch07]{schnurer.regular}
Olaf~M. Schn\"{u}rer.
\newblock Regular connections on principal fiber bundles over the infinitesimal
  punctured disc.
\newblock {\em J. Lie Theory}, 17(2):427--448, 2007.

\bibitem[Ser02]{serre.galoiscoh}
Jean-Pierre Serre.
\newblock {\em Galois cohomology}.
\newblock Springer Monographs in Mathematics. Springer-Verlag, Berlin, english
  edition, 2002.
\newblock Translated from the French by Patrick Ion and revised by the author.

\bibitem[Spr74]{springer-regular-reflection}
T.~A. Springer.
\newblock Regular elements of finite reflection groups.
\newblock {\em Invent. Math.}, 25:159--198, 1974.

\end{thebibliography}
\bibliographystyle{alpha}
\end{document}